\documentclass[11pt,leqno,a4paper,twoside,final]{article}

\usepackage[T1]{fontenc}

\usepackage{exscale,ifthen,amsfonts,amsmath,amscd,amssymb,amsthm}

\usepackage{natbib}

\usepackage{enumerate, color}

\allowdisplaybreaks[2]

\DeclareMathOperator{\Prb}{\mathbf{P}}
\DeclareMathOperator{\Mean}{\mathbf{E}}

\DeclareMathOperator{\Law}{Law}

\DeclareMathOperator{\argmin}{argmin}

\numberwithin{equation}{section}

\begin{document}

% Befehls- und Umgebungsdefinitionen

% Auslassungszeichen
\newcommand{\auslass}{(\ldots\hspace{-0.2em})}

% T fuer Matrizentransposition
\newcommand{\trans}[1]{{#1}^\mathsf{T}}

\newcommand{\BB}[1]{\boldsymbol{#1}}
\newcommand{\prbms}[2][]{\mathcal{P}_{#1}(#2)}
\newcommand{\Borel}[1]{\mathcal{B}(#1)}

% Neue Umgebungen (mit zugehoerigen Zaehlern)

\newcounter{hypcount}
\newenvironment{hypenv}{\renewcommand{\labelenumi}{(A\arabic{enumi})}\begin{enumerate}\setcounter{enumi}{\value{hypcount}}}{\setcounter{hypcount}{\value{enumi}}\end{enumerate}}

\newcounter{hypcount2}
\newenvironment{Hypenv}{\renewcommand{\labelenumi}{\textsc{(H\arabic{enumi})}}\begin{enumerate}\setcounter{enumi}{\value{hypcount2}}}{\setcounter{hypcount2}{\value{enumi}}\end{enumerate}}

\newenvironment{enumrm}{\begin{enumerate}\renewcommand{\labelenumi}{\textup{(\roman{enumi})}}}{\end{enumerate}}

% Neue Verweisbefehle (vgl. neue Umgebungen)

\newcommand{\hypref}[1]{(A\ref{#1})}
\newcommand{\hyprefall}{(A1)\,--\,(A\arabic{hypcount})}
\newcommand{\Hypref}[1]{(H\ref{#1})}
\newcommand{\Hyprefall}{(H1)\,--\,(H\arabic{hypcount2})}

\newcounter{dummy}
\newcommand{\hyprefallbutlast}{\setcounter{dummy}{\value{hypcount}}\addtocounter{dummy}{-1}(A1)\,--\,(A\arabic{dummy})}

\newcommand{\rmref}[1]{\setcounter{dummy}{\ref{#1}}(\roman{dummy})}

% AMS-Umgebungen

\theoremstyle{plain}
\newtheorem{thrm}{Theorem}[section]
\newtheorem{prop}{Proposition}[section]
\newtheorem{lemma}{Lemma}[section]
\newtheorem{crll}[thrm]{Corollary}

\theoremstyle{definition}
\newtheorem{defn}{Definition}[section]

\theoremstyle{remark}
\newtheorem{case}{Case}[section]
\newtheorem{rem}{Remark}[section]
\newtheorem{exmpl}{Example}[section]

% Skript phi / epsilon
\renewcommand{\phi}{\varphi}
\renewcommand{\epsilon}{\varepsilon}

% Ende der Befehls- und Umgebungsdefinitionen

% Titel, Danksagung
\title{On the connection between symmetric $N$-player games and mean field games}

\author{Markus Fischer\thanks{Department of Mathematics, University of Padua, via Trieste 63, 35121 Padova, Italy. The starting point for the present work was a course on mean field games given by Pierre Cardaliaguet at the University of Padua in 2013. The author is grateful to Martino Bardi and Pierre Cardaliaguet for stimulating discussions. The author thanks two anonymous Referees for their critique, comments, and helpful suggestions. Financial support was provided by the University of Padua through the Project ``Stochastic Processes and Applications to Complex Systems'' (\mbox{CPDA123182}).}
}

\date{May 14, 2014; revised September 7, 2015}

\maketitle

\begin{abstract}
Mean field games are limit models for symmetric $N$-player games with interaction of mean field type as $N\to\infty$. The limit relation is often understood in the sense that a solution of a mean field game allows to construct approximate Nash equilibria for the corresponding $N$-player games. The opposite direction is of interest, too: When do sequences of Nash equilibria converge to solutions of an associated mean field game? In this direction, rigorous results are mostly available for stationary problems with ergodic costs. Here, we identify limit points of sequences of certain approximate Nash equilibria as solutions to mean field games for problems with It{\^o}-type dynamics and costs over a finite time horizon. Limits are studied through weak convergence of associated normalized occupation measures and identified using a probabilistic notion of solution for mean field games.
\end{abstract}

\medskip {\small \textbf{2000 AMS subject classifications:} 60B10, 60K35, 91A06, 93E20}

\medskip {\small \textbf{Key words and phrases:} Nash equilibrium; mean field game; McKean-Vlasov limit; weak convergence; martingale problem; optimal control}

\section{Introduction} \label{SectIntro}

Mean field games, as introduced by J.M.~Lasry and P.-L.~Lions \citep[][]{lasrylions06a, lasrylions06b, lasrylions07} and, independently, by M.~Huang, R.P.~Malham{\'e}, and P.E.~Caines \citep[][and subsequent works]{huangetal06}, are limit models for symmetric non-zero-sum non-cooperative $N$-player games with interaction of mean field type as the number of players tends to infinity. The limit relation is often understood in the sense that a solution of the mean field game allows to construct approximate Nash equilibria for the corresponding $N$-player games if $N$ is sufficiently large; see, for instance, \citet{huangetal06}, \citet{kolokoltsovetal11}, \citet{carmonadelarue13}, and \citet{carmonalacker15}. This direction is useful from a practical point of view since the model of interest is commonly the $N$-player game with $N$ big so that a direct computation of Nash equilibria is not feasible. 

The opposite direction in the limit relation is of interest, too: When and in which sense do sequences of Nash equilibria for the $N$-player games converge to solutions of a corresponding mean field game? An answer to this question is useful as it provides information on what kind of Nash equilibria can be captured by the mean field game approach. In view of the theory of McKean-Vlasov limits and propagation of chaos for uncontrolled weakly interacting systems \citep[cf.][]{mckean66,sznitman89}, one may expect to obtain convergence results for broad classes of systems, at least under some symmetry conditions on the Nash equilibria. This heuristic was the original motivation in the introduction of mean field games by Lasry and Lions. Rigorous results supporting it are nonetheless few, and they mostly apply to stationary problems with ergodic costs and special structure (in particular, affine-linear dynamics and convex costs); see \citet{lasrylions07}, \citet{feleqi13}, \citet{bardipriuli13, bardipriuli14}. For non-stationary problems, the passage to the limit has been established rigorously in \citet{gomesetalii13} for a class of continuous-time finite horizon problems with finite state space, but only if the time horizon is sufficiently small. Moreover, in the situation studied there, Nash equilibria for the $N$-player games are unique in a class of symmetric Markovian feedback strategies. The above cited works on the passage to the limit all employ methods from the theory of ordinary or partial differential equations, in particular, equations of Hamilton-Jacobi-Bellman-type. In \citet{lacker15b}, which appeared as preprint three months after submission of the present paper, a general characterization of the limit points of $N$-player Nash equilibria is obtained through probabilistic methods. We come back to that work, which also covers mean field games with common noise, in the second but last paragraph of this section. 

Here, we study the limit relation between symmetric $N$-player games and mean field games in the direction of the Lasry-Lions heuristic for continuous time finite horizon problems with fairly general cost structure and It{\^o}-type dynamics. The aim is to identify limit points of sequences of symmetric Nash equilibria for the $N$-player games as solutions of a mean field game. For a general introduction to mean field games, see \cite{cardaliaguet13} or \citet{carmonaetalii13}. The latter work also explains the difference in the limit relation that distinguishes mean field games from optimal control problems of McKean-Vlasov-type.

To describe the prelimit systems, let $X^{N}_{i}(t)$ denote the state of player $i$ at time $t$ in the $N$-player game, and denote by $u_{i}(t)$ the control action that he or she chooses at time $t$. Individual states will be elements of $\mathbb{R}^{d}$, while control actions will be elements of some closed set $\Gamma \subset \mathbb{R}^{d_{2}}$. The evolution of the individual states is then described by the It{\^o} stochastic differential equations
\begin{equation} \label{EqPrelimitSDE}
	dX^{N}_{i}(t) = b\bigl(t,X^{N}_{i}(t),\mu^{N}(t),u_{i}(t)\bigr)dt + \sigma\bigl(t,X^{N}_{i}(t),\mu^{N}(t)\bigr)dW^{N}_{i}(t),
\end{equation}
$i\in \{1,\ldots,N\}$, where $W^{N}_{1},\ldots,W^{N}_{N}$ are independent standard Wiener processes, and $\mu^{N}(t)$ is the empirical measure of the system at time $t$:
\[
	\mu^{N}(t)\doteq \frac{1}{N} \sum_{i=1}^{N} \delta_{X^{N}_{i}(t)}.
\]
Notice that the coefficients $b$, $\sigma$ in Eq.~\eqref{EqPrelimitSDE} are the same for all players. We will assume $b$, $\sigma$ to be continuous in the time variable, Lipschitz continuous in the state and measure variable, where we use the square Wasserstein metric as a distance on probability measures, and of sub-linear growth. The dispersion coefficient $\sigma$ does not depend on the control variable, but it may depend on the measure-variable. Moreover, $\sigma$ is allowed to be degenerate. Deterministic systems are thus covered as a special case.
 
The individual dynamics are explicitly coupled only through the empirical measure process $\mu^{N}$. There is also an implicit coupling, namely through the strategies $u_{1},\ldots,u_{N}$, which may depend on non-local information; in particular, a strategy $u_{i}$ might depend, in a non-anticipative way, on $X^{N}_{j}$ or $W^{N}_{j}$ for $j\neq i$. In this paper, strategies will always be stochastic open-loop. In particular, strategies will be processes adapted to a filtration that represents the information available to the players. We consider two types of information: full information, which is the same for all players and is represented by a filtration at least as big as the one generated by the initial states and the Wiener processes, and local information, which is player-dependent and, for player $i$, is represented by the filtration generated by his/her own initial state and the Wiener process $W^{N}_{i}$. 

Let $\boldsymbol{u} = (u_{1},\ldots,u_{N})$ be a strategy vector, i.e., an $N$-vector of $\Gamma$-valued processes such that $u_{i}$ is a strategy for player $i$, $i\in \{1,\ldots,N\}$. Player $i$ evaluates the effect of the strategy vector $\boldsymbol{u}$ according to the cost functional
\[
	J^{N}_{i}\left(\boldsymbol{u}\right)\doteq \Mean\left[ \int_{0}^{T} f\left(s,X^{N}_{i}(s),\mu^{N}(s),u_{i}(s)\right)ds + F\left(X^{N}_{i}(T),\mu^{N}(T) \right) \right],
\]
where $T > 0$ is the finite time horizon, $(X^{N}_{1},\ldots,X^{N}_{N})$ the solution of the system \eqref{EqPrelimitSDE} under $\boldsymbol{u}$, and $\mu^{N}$ the corresponding empirical measure process. The cost coefficients $f$, $F$, which quantify running and terminal costs, respectively, are assumed to be continuous in the time and control variable, locally Lipschitz continuous in the state and measure variable, and of sub-quadratic growth. The action space $\Gamma$ is assumed to be closed, but not necessarily compact; in the non-compact case, $f$ will be quadratically coercive in the control. The assumptions on the coefficients are chosen so that they cover some linear-quadratic problems, in addition to many genuinely non-linear problems.

If there were no control in Eq.~\eqref{EqPrelimitSDE} (i.e., $b$ independent of the control variable) and if the initial states for the $N$-player games were independent and identically distributed with common distribution not depending on $N$, then $(X^{N}_{1},\ldots,X^{N}_{N})$ would be exchangeable for every $N\in \mathbb{N}$ and, under our assumptions on $b$, $\sigma$, the sequence $(\mu^{N})$ of empirical measure processes would converge to some deterministic flow of probability measures:
\begin{align*}
	&\mu^{N}(t) \stackrel{N\to\infty}{\longrightarrow} \mathfrak{m}(t),& &\text{in distribution\,/\,probability}.&
\end{align*}
This convergence would also hold for the sequence of path-space empirical measures, which, by symmetry and the Tanaka-Sznitman theorem, is equivalent to the propagation of chaos property for the triangular array $(X^{N}_{i})_{i\in \{1,\ldots,N\},N\in \mathbb{N}}$. In particular, $\Law(X^{N}_{i}(t)) \to \mathfrak{m}(t)$ as $N\to \infty$ for each fixed index $i$, and $\mathfrak{m}$ would be the flow of laws for the uncontrolled McKean-Vlasov equation
\begin{align*}
	&dX(t) = b\bigl(t,X(t),\mathfrak{m}(t)\bigr)dt + \sigma\bigl(t,X(t),\mathfrak{m}(t)\bigr)dW(t),& &\mathfrak{m}(t) = \Law(X(t)).&
\end{align*}
The above equation would determine the flow of measures $\mathfrak{m}$.

Now, for $N\in \mathbb{N}$, let $\boldsymbol{u}^{N}$ be a strategy vector for the $N$-player game. For the sake of argument, let us suppose that $\boldsymbol{u}^{N} = (u^{N}_{1},\ldots,u^{N}_{N})$ is a symmetric Nash equilibrium for each $N$ (symmetric in the sense that the finite sequence $((X^{N}_{1}(0),u^{N}_{1},W^{N}_{1}),\ldots,(X^{N}_{N}(0),u^{N}_{N},W^{N}_{N}))$ is exchangeable). If the mean field heuristic applies, then the associated sequence of empirical measure processes $(\mu^{N})_{N\in\mathbb{N}}$ converges in distribution to some deterministic flow of probability measures $\mathfrak{m}$. In this case, $\mathfrak{m}$ should be the flow of measures induced by the solution of the controlled equation
\begin{equation} \label{EqLimitSDE}
	dX(t) = b\bigl(t,X(t),\mathfrak{m}(t),u(t))\bigr)dt + \sigma\bigl(t,X(t),\mathfrak{m}(t))\bigr)dW(t),
\end{equation}
where the control process $u$ should, by the Nash equilibrium property of the $N$-player strategies, be optimal for the control problem
\begin{equation} \label{EqLimitCP}
\begin{split}
	&\text{minimize } J_{\mathfrak{m}}\left(v\right)\doteq \Mean\left[ \int_{0}^{T} f\left(s,X(s),\mathfrak{m}(t),v(s)\right)ds + F\left(X(T),\mathfrak{m}(T) \right) \right] \\
	&\text{over all admissible $v$ subject to: $X$ solves Eq.~\eqref{EqLimitSDE} under $v$.}
\end{split}
\end{equation}

The mean field game, which is the limit system for the $N$-player games, can now be described as follows: For each flow of measures $\mathfrak{m}$, solve the optimal control problem \eqref{EqLimitCP} to find an optimal control $u^{\mathfrak{m}}$ with corresponding state process $X^{\mathfrak{m}}$. Then choose a flow of measures $\mathfrak{m}$ according to the mean field condition $\mathfrak{m}(\cdot) = \Law(X^{\mathfrak{m}}(\cdot))$. This yields a solution of the mean field game, which can be identified with the pair $(\Law(X^{\mathfrak{m}},u^{\mathfrak{m}},W),\mathfrak{m})$; see Definition~\ref{DefMFGSol} below. We include the driving noise process $W$ in the definition of the solution, as it is the joint distribution of initial condition, control process and noise process that determines the law of a solution to Eq.~\eqref{EqLimitSDE}. If $(\Law(X^{\mathfrak{m}},u^{\mathfrak{m}},W),\mathfrak{m})$ is a solution of the mean field game, then, thanks to the mean field condition, $X^{\mathfrak{m}}$ is a McKean-Vlasov solution of the controlled equation \eqref{EqLimitSDE}; moreover, $X^{\mathfrak{m}}$ is an optimally controlled process for the standard optimal control problem \eqref{EqLimitCP} with cost functional $J_{\mathfrak{m}}$. Clearly, neither existence nor uniqueness of solutions of the mean field game are a priori guaranteed.

In order to connect sequences of Nash equilibria with solutions of the mean field game in a rigorous way, we associate strategy vectors for the $N$-player games with normalized occupation measures or path-space empirical measures; see Eq.~\eqref{ExOccupationMeasure} in Section~\ref{SectConvergence} below. Those occupation measures are random variables with values in the space of probability measures on an extended canonical space $\mathcal{Z}\doteq \mathcal{X}\times \mathcal{R}_{2}\times \mathcal{W}$, where $\mathcal{X}$, $\mathcal{W}$ are path spaces for the individual state processes and the driving Wiener processes, respectively, and $\mathcal{R}_{2}$ is a space of $\Gamma$-valued relaxed controls. Observe that $\mathcal{Z}$ contains a component for the trajectories of the driving noise process. Let $(\boldsymbol{u}^{N})$ be a sequence such that, for each $N\in \mathbb{N}$, $\boldsymbol{u}^{N}$ is a strategy vector for the $N$-player game (not necessarily a Nash equilibrium). Let $(Q_{N})$ be the associated normalized occupation measures; thus, $Q_{N}$ is the empirical measure of $((X^{N}_{1},u^{N}_{1},W^{N}_{1}),\ldots,(X^{N}_{N},u^{N}_{N},W^{N}_{N}))$ seen as a random element of $\prbms{\mathcal{Z}}$. We then show the following:
\begin{enumerate}
	\item \label{PointTightness} The family $(Q_{N})_{N\in \mathbb{N}}$ is pre-compact under a mild uniform integrability condition on strategies and initial states; see Lemma~\ref{LemmaTightness}.
	
	\item \label{PointMcKeanVlasov} Any limit random variable $Q$ of $(Q_{N})$ takes values in the set of McKean-Vlasov solutions of Eq.~\eqref{EqLimitSDE} with probability one; see Lemma~\ref{LemmaConvergence}.
	
	\item \label{PointOptimality} Suppose that $(\boldsymbol{u}^{N})$ is a sequence of local approximate Nash equilibria (cf.\ Definition~\ref{DefNash}). If $Q$ is a limit point of $(Q_{N})$ such that the flow of measures induced by $Q$ is deterministic with probability one, then $Q$ takes values in the set of solutions of the mean field game with probability one; see Theorem~\ref{ThConnection}.
\end{enumerate}

The hypothesis in Point~\ref{PointOptimality} above that the flow of measures induced by $Q$ is deterministic with probability one means that the corresponding subsequence of $(\mu^{N})$, the empirical measure processes, converges in distribution to a deterministic flow of probability measures $\mathfrak{m}$. This is a strong hypothesis, essentially part of the mean field heuristic; nonetheless, it is satisfied if $\boldsymbol{u}^{N}$ is a vector of independent and identically distributed individual strategies for each $N$, where the common distribution is allowed to vary with $N$; see Corollary~\ref{CorConnection}. While Nash equilibria for the $N$-player games with independent and identically distributed individual strategies do not exist in general, local approximate Nash equilibria with i.i.d.\ components do exist, at least under the additional assumption of compact action space $\Gamma$ and bounded coefficients; see Proposition~\ref{PropNashExistence}. In this situation, the passage to the mean field game limit is justified.

For the passage to the limit required by Point~\ref{PointMcKeanVlasov} above, we have to identify solutions of Eq.~\eqref{EqLimitSDE}, which describes the controlled dynamics of the limit system. To this end, we employ a local martingale problem in the spirit of \citet{stroockvaradhan79}. The use of martingale problems, together with weak convergence methods, has a long tradition in the analysis of McKean-Vlasov limits for uncontrolled weakly interacting systems \citep[for instance,][]{funaki84,oelschlaeger84} as well as in the study of stochastic optimal control problems. Controlled martingale problems are especially powerful in combination with relaxed controls; see \citet{elkarouietalii87, kushner90}, and the references therein. In the context of mean field games, a martingale problem formulation has been used by \citet{carmonalacker15} to establish existence and uniqueness results for non-degenerate systems and, more recently, by \citet{lacker15}, where existence of solutions is established for mean field games of the type studied here; the assumptions on the coefficients are rather mild, allowing for degenerate as well as control-dependent diffusion coefficient. The notion of solution we give in Definition~\ref{DefMFGSol} below corresponds to the notion of ``relaxed mean field game solution'' introduced in \citet{lacker15}.

The martingale problem formulation for the controlled limit dynamics we use here  is actually adapted from the joint work \citet{budhirajaetal12}, where we studied large deviations for weakly interacting It{\^o} processes through weak convergence methods. While the passage to the limit needed there for obtaining convergence of certain Laplace functionals is analogous to the convergence result of Point~\ref{PointMcKeanVlasov} above, the limit problems in \citet{budhirajaetal12} are not mean field games; they are, in fact, optimal control problems of McKean-Vlasov type, albeit with a particular structure. As a consequence, optimality has to be verified in a different way: In order to establish Point~\ref{PointOptimality} above, we construct an asymptotically approximately optimal competitor strategy in noise feedback form (i.e., as a function of time, initial condition, and the trajectory of the player's noise process up to current time), which is then applied to exactly one of the $N$ players for each $N$; this yields optimality of the limit points thanks to the Nash equilibrium property of the prelimit strategies. If the limit problem were of McKean-Vlasov type, one would use a strategy selected according to a different optimality criterion and apply it to all components (or players) of the prelimit systems.

In the work by \citet{lacker15b} mentioned in the second paragraph, limit points of normalized occupation measures associated with a sequence of $N$-player Nash equilibria are shown to be concentrated on solutions of the corresponding mean field game even if the induced limit flow of measures is stochastic (in contrast to Point~\ref{PointOptimality} above). This characterization is established for mean field systems over a finite time horizon as here, but possibly with a common noise (represented as an additional independent Wiener process common to all players). There as here, Nash equilibria are considered in stochastic open-loop strategies, and the methods of proof are similar to ours. 
The characterization of limit points in \citet{lacker15b} relies, even in the situation without common noise studied here, on a new notion of solution of the mean field game (``weak \mbox{MFG} solution'') that applies to probability measures on an extended canonical space (extended with respect to our $\mathcal{Z}$ to keep track of the possibly stochastic flow of measures). In terms of that notion of solution a complete characterization of limit points is achieved. In particular, the assumption in Point~\ref{PointOptimality} that the flow of measures induced by $Q$ is deterministic can be removed. However, if that assumption is dropped, then the claim that $Q$ takes values in the set of solutions of the mean field game with probability one will in general be false. A counterexample illustrating this fact can be deduced from the discussion of subsection~3.3 in \citet{lacker15b}. The notion of ``weak \mbox{MFG} solution'' is indeed strictly weaker than what one obtains by randomization of the usual notion of solution (``strong'' solution with probability one), and this is what makes the complete characterization of Nash limit points possible.

The rest of this work is organized as follows. Notation, basic objects as well as the standing assumptions on the coefficients $b$, $\sigma$, $f$, $F$ are introduced in Section~\ref{SectNotation}. Section~\ref{SectPrelimitSystems} contains a precise description of the $N$-player games. Nash equilibria are defined and an existence result for certain local approximate Nash equilibria is given; see Proposition~\ref{PropNashExistence}. In Section~\ref{SectLimitSystems}, the limit dynamics for the $N$-player games are introduced. The corresponding notions of McKean-Vlasov solution and solution of the mean field game are defined and discussed. An approximation result in terms of noise feedback strategies, needed in the construction of competitor strategies, is given in Lemma~\ref{LemmaNoiseFeedback}. In Section~\ref{SectConvergence}, the convergence analysis is carried out, leading to Theorem~\ref{ThConnection} and its Corollary~\ref{CorConnection}, which are our main results. Existence of solutions of the mean field game falls out as a by-product of the analysis.

%-------

\section{Preliminaries and assumptions} \label{SectNotation}

Let $d, d_{1}, d_{2}\in \mathbb{N}$, which will be the dimensions of the space of private states, noise values, and control actions, respectively. Choose $T > 0$, the finite time horizon. Set
\begin{align*}
	&\mathcal{X}\doteq \mathbf{C}([0,T],\mathbb{R}^{d}),& &\mathcal{W}\doteq \mathbf{C}([0,T],\mathbb{R}^{d_{1}}),&
\end{align*}
and, as usual, equip $\mathcal{X}$, $\mathcal{W}$ with the topology of uniform convergence, which turns them into Polish spaces. Let $\|\cdot\|_{\mathcal{X}}$, $\|\cdot\|_{\mathcal{W}}$ denote the supremum norm on $\mathcal{X}$ and $\mathcal{W}$, respectively. The spaces $\mathbb{R}^{n}$ with $n\in \mathbb{N}$ are equipped with the standard Euclidean norm, always indicated by $|.|$.

For $\mathcal{S}$ a Polish space, let $\prbms{\mathcal{S}}$ denote the space of probability measures on $\Borel{\mathcal{S}}$, the Borel sets of $\mathcal{S}$. For $s\in \mathcal{S}$, let $\delta_{s}$ indicate the Dirac measure concentrated in $s$. Equip $\prbms{\mathcal{S}}$ with the topology of weak convergence of probability measures. Then $\prbms{\mathcal{S}}$ is again a Polish space. Let $\mathrm{d}_{\mathcal{S}}$ be a metric compatible with the topology of $\mathcal{S}$ such that $(\mathcal{S},\mathrm{d}_{\mathcal{S}})$ is a complete and separable metric space. A metric that turns $\prbms{\mathcal{S}}$ into a complete and separable metric space is then given by the bounded Lipschitz metric
\begin{align*}
	\mathrm{d}_{\prbms{\mathcal{S}}}(\nu,\tilde{\nu})&\doteq \sup\left\{ \int_{\mathcal{S}} g\,d\nu - \int_{\mathcal{S}} g\,d\tilde{\nu} : g\!: \mathcal{S} \rightarrow \mathbb{R} \text{ such that } \|g\|_{\mathrm{bLip}} \leq 1 \right\}, \\
\intertext{where}
	\|g\|_{\mathrm{bLip}}&\doteq \sup_{s\in \mathcal{S}} |g(s)| + \sup_{s,\tilde{s}\in \mathcal{S}: s\neq \tilde{s}} \frac{|g(s)-g(\tilde{s})|}{\mathrm{d}_{\mathcal{S}}(s,\tilde{s})}.
\end{align*}

Given a complete compatible metric $\mathrm{d}_{\mathcal{S}}$ on $\mathcal{S}$, we also consider the space of probability measures on $\Borel{\mathcal{S}}$ with finite second moments:
\[
	\prbms[2]{\mathcal{S}}\doteq \left\{\nu\in \prbms{\mathcal{S}}:\; \exists s_{0}\in \mathcal{S}: \int_{\mathcal{S}} \mathrm{d}_{\mathcal{S}}(s,s_{0})^{2}\, \nu(ds) < \infty \right\}.
\]
Notice that $\int \mathrm{d}_{\mathcal{S}}(s,s_{0})^{2} \nu(ds) < \infty$ for some $s_{0}\in \mathcal{S}$ implies that the integral is finite for every $s_{0}\in \mathcal{S}$. The topology of weak convergence of measures plus convergence of second moments turns $\prbms[2]{\mathcal{S}}$ into a Polish space. A compatible complete metric is given by
\[
	\mathrm{d}_{\prbms[2]{\mathcal{S}}}(\nu,\tilde{\nu})\doteq \left(\inf_{\alpha\in \prbms{\mathcal{S}\times\mathcal{S}}: [\alpha]_{1} = \nu\text{ and } [\alpha]_{2} = \tilde{\nu}} \int_{\mathcal{S}\times\mathcal{S}} \mathrm{d}_{\mathcal{S}}(s,\tilde{s})^{2}\, \alpha(ds,d\tilde{s}) \right)^{1/2}, 
\]
where $[\alpha]_{1}$ ($[\alpha]_{2}$) denotes the first (second) marginal of $\alpha$; $\mathrm{d}_{\prbms[2]{\mathcal{S}}}$ is often referred to as the square Wasserstein (or Vasershtein) metric. An immediate consequence of the definition of $\mathrm{d}_{\prbms[2]{\mathcal{S}}}$ is the following observation: for all $N\in \mathbb{N}$, $s_{1},\ldots,s_{N}, \tilde{s}_{1},\ldots,\tilde{s}_{N}\in \mathcal{S}$,
\begin{equation} \label{EqEmpWasserstein}
	\mathrm{d}_{\prbms[2]{\mathcal{S}}}\left( \frac{1}{N}\sum_{i=1}^{N} \delta_{s_{i}}, \frac{1}{N}\sum_{i=1}^{N} \delta_{\tilde{s}_{i}} \right) \leq \sqrt{\frac{1}{N}\sum_{i=1}^{N} \mathrm{d}_{\mathcal{S}}\bigl(s_{i},\tilde{s}_{i}\bigr)^{2}}.  
\end{equation}
The bounded Lipschitz metric and the square Wasserstein metric on $\prbms{\mathcal{S}}$ and $\prbms[2]{\mathcal{S}}$, respectively, depend on the choice of the metric $\mathrm{d}_{\mathcal{S}}$ on the underlying space $\mathcal{S}$. This dependence will be clear from context. If $\mathcal{S} = \mathbb{R}^{d}$ with the metric induced by Euclidean norm, we may write $\mathrm{d}_{2}$ to indicate the square Wasserstein metric $\mathrm{d}_{\prbms[2]{\mathbb{R}^{d}}}$. 

Let $\mathcal{M}$, $\mathcal{M}_{2}$ denote the spaces of continuous functions on $[0,T]$ with values in $\prbms{\mathbb{R}^{d}}$ and $\prbms[2]{\mathbb{R}^{d}}$, respectively:
\begin{align*}
	& \mathcal{M}\doteq \mathbf{C}([0,T],\prbms{\mathbb{R}^{d}}),& &\mathcal{M}_{2}\doteq \mathbf{C}([0,T],\prbms[2]{\mathbb{R}^{d}}).&
\end{align*}

Let $\Gamma$ be a closed subset of $\mathbb{R}^{d_{2}}$, the set of control actions, or action space. Given a probability space $(\Omega,\mathcal{F},\Prb)$ and a filtration $(\mathcal{F}_{t})$ in $\mathcal{F}$, let $\mathcal{H}_{2}((\mathcal{F}_{t}),\Prb; \Gamma)$ denote the space of all $\Gamma$-valued $(\mathcal{F}_{t})$-progressively measurable processes $u$ such that $\Mean\left[\int_{0}^{T} |u(t)|^{2}dt \right] < \infty$. The elements of $\mathcal{H}_{2}((\mathcal{F}_{t}),\Prb; \Gamma)$ might be referred to as (individual) strategies.

Denote by $\mathcal{R}$ the space of all deterministic relaxed controls on $\Gamma\times [0,T]$, that is,
\[
	\mathcal{R}\doteq \left\{ r : r\text{ positive measure on }\mathcal{B}(\Gamma\times [0,T]): r(\Gamma\times[0,t]) = t\; \forall t\in[0,T] \right\}.
\]	
If $r\in \mathcal{R}$ and $B\in \mathcal{B}(\Gamma)$, then the mapping $[0,T]\ni t\mapsto r(B\times[0,t])$ is absolutely continuous, hence differentiable almost everywhere. Since $\mathcal{B}(\Gamma)$ is countably generated, the time derivative of $r$ exists almost everywhere and is a measurable mapping $\dot{r}_{t}\!: [0,T]\rightarrow \mathcal{P}(\Gamma)$ such that $r(dy,dt) = \dot{r}_{t}(dy)dt$. Denote by $\mathcal{R}_{2}$ the space of deterministic relaxed controls with finite second moments:
\[
	\mathcal{R}_{2}\doteq \left\{ r\in \mathcal{R}:\int_{\Gamma\times [0,T]}|y|^{2}\, r(dy,dt)<\infty \right\}.
\]
By definition, $\mathcal{R}_{2}\subset \mathcal{R}$. The topology of weak convergence of measures turns $\mathcal{R}$ into a Polish space (not compact unless $\Gamma$ is bounded). Equip $\mathcal{R}_{2}$ with the topology of weak convergence of measures plus convergence of second moments, which makes $\mathcal{R}_{2}$ a Polish space, too.

Any $\Gamma$-valued process $v$ defined on some probability space $(\Omega,\mathcal{F},\Prb)$ induces an $\mathcal{R}$-valued random variable $\rho$ according to 
\begin{equation} \label{ExControlRelaxation}
	\rho_{\omega}\bigl(B\times I\bigr)\doteq \int_{I}\delta_{v(t,\omega)}(B)dt,\quad B\in \Borel{\Gamma},\; I\in \Borel{[0,T]},\;\omega \in \Omega.
\end{equation}
If $v$ is such that $\int_{0}^{T} |v(t)|^{2} dt < \infty$ $\Prb$-almost surely, then the induced random variable $\rho$ takes values in $\mathcal{R}_{2}$ $\Prb$-almost surely. If $v$ is progressively measurable with respect to a filtration $(\mathcal{F}_{t})$ in $\mathcal{F}$, then $\rho$ is adapted in the sense that the mapping $t\mapsto \rho(B\times[0,t])$ is $(\mathcal{F}_{t})$-adapted for every $B\in \mathcal{B}(\Gamma)$ \citep[cf.][Section~3.3]{kushner90}. More generally, an $\mathcal{R}$-valued random variable $\rho$ defined on some probability space $(\Omega,\mathcal{F},\Prb)$ is called adapted to a filtration $(\mathcal{F}_{t})$ in $\mathcal{F}$ if the process $t\mapsto \rho(B\times[0,t])$ is $(\mathcal{F}_{t})$-adapted for every $B\in \Borel{\Gamma}$.

Below, we will make use of the following canonical space. Set
\begin{equation*}
	\mathcal{Z}\doteq \mathcal{X}\times \mathcal{R}_{2}\times \mathcal{W},
\end{equation*}
and endow $\mathcal{Z}$ with the product topology, which makes it a Polish space. Let $\mathrm{d}_{\mathcal{R}_{2}}$ be a complete metric compatible with the topology of $\mathcal{R}_{2}$. Set
\[
	\mathrm{d}_{\mathcal{Z}}\left((\phi,r,w),(\tilde{\phi},\tilde{r},\tilde{w})\right)\doteq \|\phi - \tilde{\phi}\|_{\mathcal{X}} + \frac{\mathrm{d}_{\mathcal{R}_{2}}(r,\tilde{r})}{1+\mathrm{d}_{\mathcal{R}_{2}}(r,\tilde{r})} + \frac{\| w - \tilde{w}\|_{\mathcal{W}}}{1+ \| w - \tilde{w}\|_{\mathcal{W}}},
\]
where $(\phi,r,w)$, $(\tilde{\phi},\tilde{r},\tilde{w})$ are elements of $\mathcal{Z}$ written component-wise. This defines a complete metric compatible with the topology of $\mathcal{Z}$. Let $\mathrm{d}_{\prbms[2]{\mathcal{Z}}}$ be the square Wasserstein metric on $\prbms[2]{\mathcal{Z}}$ induced by $\mathrm{d}_{\mathcal{Z}}$. Since $\mathrm{d}_{\mathcal{Z}}$ is bounded with respect to the second and third component of $\mathcal{Z}$, the condition of finite second moment is a restriction only on the first marginal of the probability measures on $\Borel{\mathcal{Z}}$. Let us indicate by $\mathrm{d}_{\prbms{\prbms[2]{\mathcal{Z}}}}$ the bounded Lipschitz metric on $\prbms{\prbms[2]{\mathcal{Z}}}$ induced by $\mathrm{d}_{\prbms[2]{\mathcal{Z}}}$. Denote by $(\hat{X},\hat{\rho},\hat{W})$ the coordinate process on $\mathcal{Z}$:
\begin{align*}
	& \hat{X}(t,(\phi,r,w))\doteq \phi (t),& & \hat{\rho}(t,(\phi,r,w))\doteq r_{|\mathcal{B}(\Gamma\times [0,t])},& & \hat{W}(t,(\phi,r,w))\doteq w(t).&
\end{align*}
Let $(\mathcal{G}_{t})$ be the canonical filtration in $\mathcal{B}(\mathcal{Z})$, that is,
\[
	\mathcal{G}_{t}\doteq \sigma \left( (\hat{X},\hat{\rho},\hat{W})(s) : 0\leq s\leq t\right),\quad t\in [0,T].
\]

Let $b$ denote the drift coefficient and $\sigma$ the dispersion coefficient of the dynamics, and let $f$, $F$ quantify the running costs and terminal costs, respectively; we take
\begin{align*}
	& b\!: [0,T]\times \mathbb{R}^{d}\times \prbms[2]{\mathbb{R}^{d}}\times \Gamma \rightarrow \mathbb{R}^{d},& & \sigma\!: [0,T]\times \mathbb{R}^{d}\times \prbms[2]{\mathbb{R}^{d}} \rightarrow \mathbb{R}^{d\times d_{1}},& \\
	& f\!: [0,T]\times \mathbb{R}^{d}\times \prbms[2]{\mathbb{R}^{d}}\times \Gamma \rightarrow [0,\infty),& & F\!: \mathbb{R}^{d}\times \prbms[2]{\mathbb{R}^{d}} \rightarrow [0,\infty). &
\end{align*}
Notice that the dispersion coefficient $\sigma$ does not depend on the control variable and that the cost coefficients $f$, $F$ are non-negative functions. We make the following assumptions, where $K$, $L$ are some finite positive constants:
\begin{hypenv}

	\item \label{HypMeasCont} Measurability and continuity in time and control: $b$, $\sigma$, $f$, $F$ are Borel measurable and such that, for all $(x,\nu)\in \mathbb{R}^{d}\times \prbms[2]{\mathbb{R}^{d}}$, $b(\cdot,x,\nu,\cdot)$, $\sigma(\cdot,x,\nu)$, $f(\cdot,x,\nu,\cdot)$ are continuous, uniformly over compact subsets of $\mathbb{R}^{d}\times \prbms[2]{\mathbb{R}^{d}}$.
	 
	\item \label{HypLipschitz} Lipschitz continuity of $b$, $\sigma$: for all $x,\tilde{x}\in \mathbb{R}^{d}$, $\nu,\tilde{\nu}\in \prbms[2]{\mathbb{R}^{d}}$,
\begin{multline*}
	\sup_{t\in[0,T]} \sup_{\gamma\in\Gamma} \left\{ \left|b(t,x,\nu,\gamma) - b(t,\tilde{x},\tilde{\nu},\gamma) \right| \vee \left|\sigma(t,x,\nu) - \sigma(t,\tilde{x},\tilde{\nu}) \right| \right\} \\
	\leq L\left( |x-\tilde{x}| + \mathrm{d}_{2}(\nu,\tilde{\nu}) \right).
\end{multline*}

	\item \label{HypGrowth} Sublinear growth of $b$, $\sigma$: for all $x\in \mathbb{R}^{d}$, $\nu\in \prbms[2]{\mathbb{R}^{d}}$, $\gamma\in \Gamma$,
\begin{align*}
	\sup_{t\in[0,T]} \left|b(t,x,\nu,\gamma)\right| &\leq K\left(1 + |x| + |\gamma| + \sqrt{\int |y|^{2} \nu(dy)} \right), \\
	\sup_{t\in[0,T]} \left|\sigma(t,x,\nu) \right| &\leq K\left(1 + |x| + \sqrt{\int |y|^{2} \nu(dy)} \right).
\end{align*}

	\item \label{HypCostLip} Local Lipschitz continuity of $f$, $F$: for all $x,\tilde{x}\in \mathbb{R}^{d}$, $\nu,\tilde{\nu}\in \prbms[2]{\mathbb{R}^{d}}$
\begin{multline*}
	\sup_{t\in[0,T]} \sup_{\gamma\in\Gamma} \left\{ \left|f(t,x,\nu,\gamma) - f(t,\tilde{x},\tilde{\nu},\gamma) \right| + \left|F(x,\nu) - F(\tilde{x},\tilde{\nu}) \right| \right\} \\
	\leq L\left( |x-\tilde{x}| + \mathrm{d}_{2}(\nu,\tilde{\nu}) \right) \left( 1 + |x| + |\tilde{x}| + \sqrt{\int |y|^{2} \nu(dy)} + \sqrt{\int |y|^{2} \tilde{\nu}(dy)} \right).
\end{multline*}
	
	\item \label{HypCostGrowth} Subquadratic growth of $f$, $F$: for all $x\in \mathbb{R}^{d}$, $\nu\in \prbms[2]{\mathbb{R}^{d}}$, $\gamma\in \Gamma$,
\[
	\sup_{t\in[0,T]} \left\{ \left|f(t,x,\nu,\gamma)\right| \vee \left|F(x,\nu) \right| \right\} \leq K \left(1 + |x|^{2} + |\gamma|^{2} + \int |y|^{2} \nu(dy) \right). 
\]	

	\item \label{HypCostCoercivity} Action space and coercivity: $\Gamma \subset \mathbb{R}^{d_{2}}$ is closed, and there exist $c_{0} > 0$ and $\Gamma_{0}\subset \Gamma$ such that $\Gamma_{0}$ is compact and for every $\gamma\in \Gamma\setminus \Gamma_{0}$
\[
	\inf_{(t,x,\nu)\in [0,T]\times \mathbb{R}^{d}\times \prbms[2]{\mathbb{R}^{d}}} f(t,x,\nu,\gamma) \geq c_{0} |\gamma|^{2}.
\]

\end{hypenv}

%-------

\section{$N$-player games} \label{SectPrelimitSystems}

Let $N\in \mathbb{N}$. Let $(\Omega_{N},\mathcal{F}^{N},\Prb_{N})$ be a complete probability space equipped with a filtration $(\mathcal{F}^{N}_{t})$ in $\mathcal{F}^{N}$ that satisfies the usual hypotheses and carrying $N$ independent $d_{1}$-dimensional $(\mathcal{F}^{N}_{t})$-Wiener processes $W^{N}_{1},\ldots,W^{N}_{N}$. For each $i\in \{1,\ldots,N\}$, choose a random variable $\xi^{N}_{i} \in L^{2}(\Omega_{N},\mathcal{F}^{N}_{0},\Prb_{N};\mathbb{R}^{d})$, the initial state of player $i$ in the prelimit game with $N$ players. In addition, we assume that the stochastic basis is rich enough to carry a sequence $(\vartheta^{N}_{i})_{i\in \{1,\ldots,N\}}$ of independent random variables with values in the interval $[0,1]$ such that each $\vartheta^{N}_{i}$ is $\mathcal{F}^{N}_{0}$-measurable and uniformly distributed on $[0,1]$, and $(\vartheta^{N}_{i})_{i\in \{1,\ldots,N\}}$ is independent of the $\sigma$-algebra generated by $\xi^{N}_{1},\ldots,\xi^{N}_{N}$ and the Wiener processes $W^{N}_{1},\ldots,W^{N}_{N}$. The random variables $\vartheta^{N}_{i}$, $i\in \{1,\ldots,N\}$, are a technical device which we may use without loss of generality; see Remark~\ref{RemNashlocalext} below.

A vector of individual strategies, that is, a vector $\boldsymbol{u}= (u_{1},\ldots,u_{N})$ such that $u_{1},\ldots,u_{N}\in \mathcal{H}_{2}((\mathcal{F}^{N}_{t}),\Prb_{N};\Gamma)$, is called a strategy vector. Given a strategy vector $\boldsymbol{u}= (u_{1},\ldots,u_{N})$, consider the system of It{\^o} stochastic integral equations
\begin{equation} \label{EqPrelimitDynamics}
\begin{split}
	X^{N}_{i}(t) &= \xi^{N}_{i} + \int_{0}^{t} b\left(s,X^{N}_{i}(s),\mu^{N}(s),u_{i}(s)\right)ds \\
	&\quad + \int_{0}^{t} \sigma\left(s,X^{N}_{i}(s),\mu^{N}(s)\right)dW^{N}_{i}(s),\quad t\in [0,T],
\end{split}
\end{equation}
$i\in\{1,\ldots,N\}$, where $\mu^{N}(s)$ is the empirical measure of the processes $X^{N}_{1},\ldots,X^{N}_{N}$ at time $s\in [0,T]$, that is,
\[
	\mu^{N}_{\omega}(s)\doteq \frac{1}{N} \sum_{j=1}^{N} \delta_{X^{N}_{j}(s,\omega)},\quad \omega\in \Omega_{N}.
\]
The process $X^{N}_{i}$ describes the evolution of the private state of player $i$ if he/she uses strategy $u_{i}$ while the other players use strategies $u_{j}$, $j\neq i$. Thanks to assumptions \hypref{HypLipschitz} and \hypref{HypGrowth}, the system of equations \eqref{EqPrelimitDynamics} possesses a unique solution in the following sense: given any strategy vector $\boldsymbol{u}= (u_{1},\ldots,u_{N})$, there exists a vector $(X^{N}_{1},\ldots,X^{N}_{N})$ of continuous $\mathbb{R}^{d}$-valued $(\mathcal{F}^{N}_{t})$-adapted processes  such that \eqref{EqPrelimitDynamics} holds $\Prb_{N}$-almost surely, and $(X^{N}_{1},\ldots,X^{N}_{N})$ is unique (up to $\Prb_{N}$-indistinguishability) among all continuous $(\mathcal{F}^{N}_{t})$-adapted solutions.

The following estimates on the controlled state process and the associated empirical measure process will be useful in Section~\ref{SectConvergence}.

\begin{lemma} \label{LemmaGrowthBounds}
There exists a finite constant $C_{T,K}$ depending on $T$, $K$, but not on $N$, such that if $\boldsymbol{u}^{N}= (u^{N}_{1},\ldots,u^{N}_{N})$ is a strategy vector for the $N$-player game and $(X^{N}_{1},\ldots,X^{N}_{N})$ the solution of the system \eqref{EqPrelimitDynamics} under $\boldsymbol{u}^{N}$, then
\begin{multline*}
	\sup_{t\in [0,T]} \Mean_{N}\left[ |X^{N}_{i}(t)|^{2} \right] \\
	\leq C_{T,K} \left(1 + \Mean_{N}\left[ |\xi^{N}_{i}|^{2} \right] + \Mean_{N}\left[ \int_{0}^{T} \left(\mathrm{d}_{2}\left(\mu^{N}(t),\delta_{0}\right)^{2} + |u^{N}_{i}(t)|^{2}\right) dt \right] \right)
\end{multline*}
for every $i\in \{1,\ldots,N\}$, and
\begin{multline*}
	\sup_{t\in [0,T]} \Mean_{N}\left[ \mathrm{d}_{2}\left(\mu^{N}(t),\delta_{0}\right)^{2} \right] \leq \sup_{t\in [0,T]} \Mean_{N}\left[ \frac{1}{N}\sum_{j=1}^{N}|X^{N}_{j}(t)|^{2} \right] \\
	\leq C_{T,K} \left(1+  \frac{1}{N}\sum_{j=1}^{N} \Mean_{N}\left[ |\xi^{N}_{j}|^{2} + \int_{0}^{T} |u^{N}_{j}(t)|^{2}dt \right] \right).
\end{multline*}
\end{lemma}

\begin{proof}
By Jensen's inequality, H{\"o}lder's inequality, It{\^o}'s isometry, assumption~\hypref{HypGrowth}, and the Fubini-Tonelli theorem, we have for every $t\in [0,T]$,
\[
\begin{split}
	\Mean_{N}\left[ |X^{N}_{i}(t)|^{2} \right] &\leq 3\Mean_{N}\left[ |\xi^{N}_{i}|^{2} \right] + 12(T+1)K^{2} \int_{0}^{t} \Mean_{N}\left[ |X^{N}_{i}(s)|^{2} \right] ds\\
	&\hspace{-2ex}+ 12(T+1)K^{2} \Mean_{N}\left[\int_{0}^{T}\left(1 + \mathrm{d}_{2}\left(\mu^{N}(s),\delta_{0}\right)^{2} + |u^{N}_{i}(s)|^{2} \right)ds \right],
\end{split}
\]
and the first estimate follows by Gronwall's lemma.

By definition of the square Wasserstein metric $\mathrm{d}_{2}$, we have for every $t\in [0,T]$, every $\omega\in \Omega_{N}$,
\[
	\mathrm{d}_{2}\left(\mu^{N}_{\omega}(t),\delta_{0}\right)^{2} = \frac{1}{N}\sum_{j=1}^{N}|X^{N}_{j}(t,\omega)|^{2}.
\]
Thus, using again assumption~\hypref{HypGrowth} and the same inequalities as above, we have for every $t\in [0,T]$,
\begin{align*}
	& \Mean_{N}\left[ \frac{1}{N}\sum_{j=1}^{N}|X^{N}_{j}(t)|^{2} \right] \\
%\begin{split}
%	&\leq 3\Mean_{N}\left[ \frac{1}{N}\sum_{j=1}^{N} |\xi^{N}_{j}|^{2} \right] + 12(T+1)K^{2} \Mean_{N}\left[\int_{0}^{T}\left(1 + \frac{1}{N}\sum_{j=1}^{N} |u^{N}_{j}(s)|^{2} \right)ds \right] \\
%	&\quad + 12(T+1)K^{2} \int_{0}^{t} \Mean_{N}\left[\mathrm{d}_{2}\left(\mu^{N}(s),\delta_{0}\right)^{2} + \frac{1}{N}\sum_{j=1}^{N}|X^{N}_{j}(s)|^{2} \right] ds.
%\end{split}\\
\begin{split}
	&\leq 3\Mean_{N}\left[ \frac{1}{N}\sum_{j=1}^{N} |\xi^{N}_{j}|^{2} \right] + 12(T+1)K^{2} \int_{0}^{T} \Mean_{N}\left[ 1 + \frac{1}{N}\sum_{j=1}^{N} |u^{N}_{j}(s)|^{2} \right] ds \\
	&\quad + 24(T+1)K^{2} \int_{0}^{t} \Mean_{N}\left[\frac{1}{N}\sum_{j=1}^{N}|X^{N}_{j}(s)|^{2} \right] ds,
\end{split}
\end{align*}
and we conclude again by Gronwall's lemma. The constant $C_{T,K}$ for both estimates need not be greater than $12(T\vee 1)(T+1)(K\vee 1)^{2} \exp\left(24(T\!+\!1)K^{2}T\right)$.
\end{proof}

\begin{lemma} \label{LemmaGrowthBounds2}
Let $p \geq 2$. Then there exists a finite constant $\tilde{C}_{p,T,K,d}$ depending on $p$, $T$, $K$, $d$, but not on $N$ such that if $\boldsymbol{u}^{N}= (u^{N}_{1},\ldots,u^{N}_{N})$ is a strategy vector for the $N$-player game and $(X^{N}_{1},\ldots,X^{N}_{N})$ the solution of the system \eqref{EqPrelimitDynamics} under $\boldsymbol{u}^{N}$, then 
\begin{multline*}
	\Mean_{N}\left[\sup_{t\in [0,T]} \mathrm{d}_{2}\left(\mu^{N}(t),\delta_{0}\right)^{p} \right] \leq \frac{1}{N}\sum_{i=1}^{N} \Mean_{N}\left[ \|X^{N}_{i}\|_{\mathcal{X}}^{p} \right] \\
	\leq \tilde{C}_{p,T,K,d} \left(1 + \frac{1}{N}\sum_{i=1}^{N} \Mean_{N}\left[ |\xi^{N}_{i}|^{p} + \int_{0}^{T} |u^{N}_{i}(t)|^{p} dt \right] \right).
\end{multline*}
\end{lemma}

\begin{proof}
The inequality
\[
	\Mean_{N}\left[\sup_{t\in [0,T]} \mathrm{d}_{2}\left(\mu^{N}(t),\delta_{0}\right)^{p} \right] \leq \frac{1}{N}\sum_{i=1}^{N} \Mean_{N}\left[ \|X^{N}_{i}\|_{\mathcal{X}}^{p} \right]
\]
follows by \eqref{EqEmpWasserstein} and Jensen's inequality. In verifying the second part of the assertion, we may assume that
\[
	\frac{1}{N}\sum_{i=1}^{N}\Mean_{N}\left[ |\xi^{N}_{i}|^{p} + \int_{0}^{T} |u^{N}_{i}(t)|^{p} dt \right] < \infty.
\]
By Jensen's inequality, H{\"o}lder's inequality, \hypref{HypGrowth}, the Fubini-Tonelli theorem, and the Burkholder-Davis-Gundy inequalities, we have for every $t\in [0,T]$,
\begin{align*}
	& \frac{1}{N}\sum_{i=1}^{N} \Mean_{N}\left[\sup_{s\in [0,t]} |X^{N}_{i}(s)|^{p} \right] \\
%\begin{split}
%	&\leq \frac{3^{p-1}}{N}\sum_{i=1}^{N} \left( \Mean_{N}\left[ |\xi^{N}_{i}|^{p} \right] + T^{p-1} \int_{0}^{t} \Mean_{N}\left[ \left|b\bigl(s,X^{N}_{i}(s),\mu^{N}(s),u^{N}_{i}(s)\bigr)\right|^{p} \right] ds\right.\\
%	&\qquad  \left. + \hat{C}_{p,d}\, T^{p/2-1} \int_{0}^{t} \Mean_{N}\left[ \left|\sigma\bigl(s,X^{N}_{i}(s),\mu^{N}(s)\bigr)\right|^{p} \right] ds \right)
%\end{split}\\
%\begin{split}
%	&\leq 3^{p-1}\Mean_{N}\left[\frac{1}{N}\sum_{i=1}^{N} |\xi^{N}_{i}|^{p} \right] \\
%	&\;+ \frac{12^{p-1} T^{p-1} K^{p}}{N}\sum_{i=1}^{N} \int_{0}^{t} \Mean_{N}\left[ 1 + |X^{N}_{i}(s)|^{p} + \mathrm{d}_{2}\left(\mu^{N}(s),\delta_{0}\right)^{p} + |u^{N}_{i}(s)|^{p} \right] ds \\
%	&\quad+ \frac{9^{p-1} T^{p/2-1} K^{p} \hat{C}_{p,d}}{N}\sum_{i=1}^{N} \int_{0}^{t} \Mean_{N}\left[ 1 + |X^{N}_{i}(s)|^{p} + \mathrm{d}_{2}\left(\mu^{N}(s),\delta_{0}\right)^{p} \right] ds
%\end{split} \\
%\begin{split}
%	&\leq \hat{C}_{p,T,K,d}\left(1 + \frac{1}{N}\sum_{i=1}^{N}\Mean_{N}\left[|\xi^{N}_{i}|^{p} + \int_{0}^{T}  |u^{N}_{i}(s)|^{p} \right] ds \right) \\
%	&\quad + \hat{C}_{p,T,K,d} \int_{0}^{t} \frac{1}{N}\sum_{i=1}^{N} \Mean_{N}\left[ |X^{N}_{i}(s)|^{p} + \mathrm{d}_{2}\left(\mu^{N}(s),\delta_{0}\right)^{p} \right] ds,
%\end{split} \\
\begin{split}
	&\leq \hat{C}_{p,T,K,d} \left(1 + \frac{1}{N}\sum_{i=1}^{N}\Mean_{N}\left[|\xi^{N}_{i}|^{p} + \int_{0}^{T}  |u^{N}_{i}(s)|^{p} ds \right] \right) \\
	&\quad +\, 2\hat{C}_{p,T,K,d} \int_{0}^{t} \frac{1}{N}\sum_{i=1}^{N} \Mean_{N}\left[ \sup_{s\in [0,\tilde{s}]}|X^{N}_{i}(s)|^{p} \right] d\tilde{s},
\end{split}
\end{align*}
where $\hat{C}_{p,T,K,d}\doteq 12^{p-1} (T\vee 1)^{p} (K\vee 1)^{p}(1+\hat{C}_{p,d})$ and $\hat{C}_{p,d}$, which depends only on $p$ and $d$, is the finite ``universal'' constant from the Burkholder-Davis-Gundy inequalities \citep[for instance, Theorem~3.3.28 and Remark~3.3.30 in ][pp.\,166-167]{karatzasshreve91}. The assertion now follows thanks to Gronwall's lemma.
\end{proof}

Player $i$ evaluates a strategy vector $\boldsymbol{u} = (u_{1},\ldots,u_{N})$ according to the cost functional
\[
	J^{N}_{i}\left(\boldsymbol{u}\right)\doteq \Mean_{N}\left[ \int_{0}^{T} f\left(s,X^{N}_{i}(s),\mu^{N}(s),u_{i}(s)\right)ds + F\left(X^{N}_{i}(T),\mu^{N}(T) \right) \right],
\]
where $(X^{N}_{1},\ldots,X^{N}_{N})$ is the solution of the system \eqref{EqPrelimitDynamics} under $\boldsymbol{u}$ and $\mu^{N}$ is the empirical measure process induced by $(X^{N}_{1},\ldots,X^{N}_{N})$.

Given a strategy vector $\boldsymbol{u} = (u_{1},\ldots,u_{N})$ and an individual strategy $v\in \mathcal{H}_{2}((\mathcal{F}^{N}_{t}),\Prb_{N};\Gamma)$, let $[\boldsymbol{u}^{-i},v]\doteq (u_{1},\ldots,u_{i-1},v,u_{i+1},\ldots,u_{N})$ indicate the strategy vector that is obtained from $\boldsymbol{u}$ by replacing $u_{i}$, the strategy of player $i$, with $v$. Let $(\mathcal{F}^{N,i}_{t})$ denote the filtration generated by $\vartheta^{N}_{i}$, $\xi^{N}_{i}$, and the Wiener process $W^{N}_{i}$, that is,
\[
	\mathcal{F}^{N,i}_{t}\doteq \sigma\left(\vartheta^{N}_{i},\xi^{N}_{i},W^{N}_{i}(s) : s\in [0,t]\right),\quad t\in [0,T].
\]
The filtration $(\mathcal{F}^{N,i}_{t})$ represents the local information available to player $i$. Clearly, $(\mathcal{F}^{N,i}_{t}) \subset \mathcal{F}^{N}_{t}$ and $\mathcal{H}_{2}((\mathcal{F}^{N,i}_{t}),\Prb_{N};\Gamma) \subset \mathcal{H}_{2}((\mathcal{F}^{N}_{t}),\Prb_{N};\Gamma)$. We may refer to the elements of $\mathcal{H}_{2}((\mathcal{F}^{N,i}_{t}),\Prb_{N};\Gamma)$ as narrow strategies or narrow individual strategies for player $i$.  

\begin{defn} \label{DefNash}
	Let $\epsilon \geq 0$, $u_{1},\ldots,u_{N}\in \mathcal{H}_{2}((\mathcal{F}^{N}_{t}),\Prb_{N};\Gamma)$. The strategy vector $\boldsymbol{u} \doteq (u_{1},\ldots,u_{N})$ is called a \emph{local $\epsilon$-Nash equilibrium} for the $N$-player game if for every $i\in \{1,\ldots,N\}$, every $v\in \mathcal{H}_{2}((\mathcal{F}^{N,i}_{t}),\Prb_{N}; \Gamma)$,
\begin{equation} \label{EqDefNash}
	J^{N}_{i}(\boldsymbol{u}) \leq J^{N}_{i}\left([\boldsymbol{u}^{-i},v]\right) + \epsilon.
\end{equation}

If inequality \eqref{EqDefNash} holds for all $v\in \mathcal{H}_{2}((\mathcal{F}^{N}_{t}),\Prb_{N}; \Gamma)$, then $\boldsymbol{u}$ is called an \emph{$\epsilon$-Nash equilibrium}.

If $\boldsymbol{u}$ is a (local) $\epsilon$-Nash equilibrium with $\epsilon = 0$, then $\boldsymbol{u}$ is called a \emph{(local) Nash equilibrium}.
\end{defn}

\begin{rem} \label{RemNash}
In Definition~\ref{DefNash}, Nash equilibria are defined with respect to stochastic open-loop strategies. This is the same notion as the one used in the probabilistic approach to mean field games; see \citet{carmonadelarue13}. A Nash equilibrium in stochastic open-loop strategies may be induced by a Markov feedback strategy (or a more general closed-loop strategy); still, it need not correspond to a Nash equilibrium in feedback strategies. Given a vector of feedback strategies, varying the strategy of exactly one player means that the feedback functions defining the strategies of the other players are kept frozen. Since in general the state processes of the other players depend on the state process of the deviating player (namely, through the empirical measure of the system), the strategies of the other players seen as control processes may change when one player deviates. This is in contrast with the stochastic open-loop formulation where the control processes of the other players are frozen when one player varies her/his strategy. Now, suppose we had a Nash equilibrium in Markov feedback strategies for the $N$-player game. If the feedback functions defining that Nash equilibrium depend only on time, the current individual state, and the current empirical measure, and if they are regular in the sense of being Lipschitz continuous, then they will induce an $\epsilon_{N}$-Nash equilibrium in stochastic open-loop strategies with $\epsilon_{N}$ also depending on the Lipschitz constants of the feedback functions. Here, we do not address the question of when Nash equilibria in regular feedback strategies exist nor of how their Lipschitz constants would depend on the number of players $N$. Neither do we address the more general question of convergence of $N$-player Nash equilibria in feedback strategies, regular or not. That difficult problem was posed in \citet{lasrylions06b, lasrylions07} and is beyond the scope of this work.
\end{rem}

\begin{rem} \label{RemNashlocalext}
The random variables $\vartheta^{N}_{i}$ appearing in the definition of the local information filtrations $(\mathcal{F}^{N,i}_{t})$ are a technical device for randomization. They will be used in the sequel only in two places, namely in the proof of Proposition~\ref{PropNashExistence} on existence of local $\epsilon$-Nash equilibria, where they allow to pass from optimal relaxed controls to nearly optimal ordinary controls, and in the proof of Lemma~\ref{LemmaCoupling}, where they serve to generate a coupling of initial conditions. The presence of the random variables $\vartheta^{N}_{i}$ causes no loss of generality in the following sense. Suppose that $\boldsymbol{u} \doteq (u_{1},\ldots,u_{N})$ is a strategy vector adapted to the filtration generated by $\xi^{N}_{1},\ldots,\xi^{N}_{N}$ and the Wiener processes $W^{N}_{1},\ldots,W^{N}_{N}$ such that, for some $\epsilon \geq 0$, every $i\in \{1,\ldots,N\}$, inequality \eqref{EqDefNash} holds for all individual strategies $v$ that are adapted to the filtration generated by $\xi^{N}_{i}$ and the Wiener process $W^{N}_{i}$. Then inequality \eqref{EqDefNash} holds for all $v\in \mathcal{H}_{2}((\mathcal{F}^{N,i}_{t}),\Prb_{N}; \Gamma)$; hence $\boldsymbol{u}$ is a local $\epsilon$-Nash equilibrium. To check this, take conditional expectation with respect to $\vartheta^{N}_{i}$ inside the expectation defining the cost functional $J^{N}_{i}$ and use the independence of $\vartheta^{N}_{i}$ from the $\sigma$-algebra generated by $\xi^{N}_{1},\ldots,\xi^{N}_{N}$ and $W^{N}_{1},\ldots,W^{N}_{N}$. An analogous reasoning applies to the situation of non-local (approximate) Nash equilibria provided the strategy vector $\boldsymbol{u}$ is independent of the family $(\vartheta^{N}_{i})_{i\in \{1,\ldots,N\}}$.
\end{rem}

By Definition~\ref{DefNash}, an $\epsilon$-Nash equilibrium is also a local $\epsilon$-Nash equilibrium. Observe that the individual strategies of a local $\epsilon$-Nash equilibrium are adapted to the full filtration $(\mathcal{F}^{N}_{t})$; only the competitor strategies in the verification of the local equilibrium property have to be narrow strategies, that is, strategies adapted to one of the smaller filtrations $(\mathcal{F}^{N,1}_{t}),\ldots,(\mathcal{F}^{N,N}_{t})$.

If $\xi^{N}_{1},\ldots,\xi^{N}_{N}$ are independent and $\boldsymbol{u} = (u_{1},\ldots,u_{N})$ is a vector of narrow strategies, that is, $u_{i}\in \mathcal{H}_{2}((\mathcal{F}^{N,i}_{t}),\Prb_{N};\Gamma)$ for every $i\in \{1,\ldots,N\}$, then $(\xi^{N}_{1},u_{1},W^{N}_{1}),\ldots,(\xi^{N}_{N},u_{N},W^{N}_{N})$, interpreted as $\mathbb{R}^{d}\times \mathcal{R}_{2}\times \mathcal{W}$-valued random variables, are independent. This allows to deduce existence of local approximate Nash equilibria through Fan's fixed point theorem in a way similar to that for one-shot games \citep[cf.\ Appendix~8.1 in][]{cardaliaguet13}. For simplicity, we give the result for a compact action space, bounded coefficients and in the fully symmetric situation. In the sequel, Proposition~\ref{PropNashExistence} will be used only to provide an example of a situation in which all the hypotheses of our main result can be easily verified. 

\begin{prop} \label{PropNashExistence}
In addition to \hyprefall, assume that $\Gamma$ is compact and that $b$, $\sigma$, $f$, $F$ are bounded. Suppose that $\xi^{N}_{1},\ldots,\xi^{N}_{N}$ are independent and identically distributed. Given any $\varepsilon > 0$, there exist narrow strategies $u^{\epsilon}_{i}\in \mathcal{H}_{2}((\mathcal{F}^{N,i}_{t}),\Prb_{N};\Gamma)$, $i\in \{1,\ldots,N\}$, such that $\boldsymbol{u}^{\epsilon} \doteq (u^{\epsilon}_{1},\ldots,u^{\epsilon}_{N})$ is a local $\epsilon$-Nash equilibrium for the $N$-player game and the random variables $(\xi^{N}_{1},u^{\epsilon}_{1},W^{N}_{1}),\ldots,(\xi^{N}_{N},u^{\epsilon}_{N},W^{N}_{N})$ are independent and identically distributed.
\end{prop}

\begin{proof}
Since $\Gamma$ is compact by hypothesis, we have $\mathcal{R} = \mathcal{R}_{2}$ as topological spaces, and $\prbms{\mathcal{R}}$ is compact.

Let $\mathfrak{m}_{0}$ denote the common distribution of the initial states $\xi^{N}_{1},\ldots,\xi^{N}_{N}$; thus $\mathfrak{m}_{0}\in \prbms[2]{\mathbb{R}^{d}}$. With a slight abuse of notation, let $(\hat{X}(0),\rho,\hat{W})$ denote the restriction to $\mathbb{R}^{d}\times \mathcal{R}\times \mathcal{W}$ of the canonical process on $\mathcal{Z}$. Let $(\tilde{\mathcal{G}}_{t})$ indicate the corresponding canonical filtration, that is, $\tilde{\mathcal{G}}_{t}\doteq \sigma(\hat{X}(0),\rho(s),\hat{W}(s): s\leq t)$, $t\in [0,T]$. Let $\mathcal{Y}$ be the space of all $\nu\in \prbms{\mathbb{R}^{d}\times \mathcal{R}\times \mathcal{W}}$ such that $[\nu]_{1} = \mathfrak{m}_{0}$ and $\hat{W}$ is a $(\tilde{\mathcal{G}}_{t})$-Wiener process under $\nu$ (in particular, $\hat{W}(0) = 0$ $\nu$-almost surely). Then $\mathcal{Y}$ is a non-empty compact convex subset of $\prbms{\mathbb{R}^{d}\times \mathcal{R}\times \mathcal{W}}$, which in turn is contained in a locally convex topological linear space (under the topology of weak convergence of measures).

The proof proceeds in two steps. First, we show that there exists $\nu_{\ast}\in \mathcal{Y}$ such that $\otimes^{N} \nu_{\ast}$ corresponds to a local Nash equilibrium in relaxed controls on the canonical space $\mathcal{Z}^{N}$. In the second step, given any $\epsilon > 0$, we use $\nu_{\ast}$ to construct a local $\epsilon$-Nash equilibrium for the $N$-player game.

\textbf{First step}. Let $\nu,\bar{\nu}\in \mathcal{Y}$. Then there exists a unique $\Psi(\nu;\bar{\nu}) \in \prbms[2]{\mathcal{Z}^{N}}$ such that
\[
	\Psi(\nu;\bar{\nu}) = \Prb\circ \left(\boldsymbol{X},\boldsymbol{\rho},\boldsymbol{W} \right)^{-1},
\]
where $\boldsymbol{W} = (W_{1},\ldots,W_{N})$ is a vector of independent $d_{1}$-dimensional $(\mathcal{F}_{t})$-adapted Wiener processes defined on some stochastic basis $((\Omega,\mathcal{F},\Prb),(\mathcal{F}_{t}))$ satisfying the usual hypotheses and carrying a vector $\boldsymbol{\rho} = (\rho^{1},\ldots,\rho^{N})$ of $(\mathcal{F}_{t})$-adapted $\mathcal{R}$-valued random variables such that
\[
	\Prb\circ \left(\boldsymbol{X}(0),\boldsymbol{\rho},\boldsymbol{W} \right)^{-1} = \nu \otimes^{N-1}\bar{\nu},
\]
and $\boldsymbol{X} = (X_{1},\ldots,X_{N})$ is the vector of continuous $\mathbb{R}^{d}$-valued $(\mathcal{F}_{t})$-adapted processes determined through the system of equations
\begin{equation} \label{EqPrelimitDynamicsRel}
\begin{split}
	X_{i}(t) &= X_{i}(0) + \int_{\Gamma\times [0,t]} b\left(s,X_{i}(s),\frac{1}{N} \sum_{j=1}^{N} \delta_{X^{N}_{j}(s)},\gamma\right) \rho^{i}(d\gamma,ds) \\
	&\quad + \int_{0}^{t} \sigma\left(s,X_{i}(s),\frac{1}{N} \sum_{j=1}^{N} \delta_{X^{N}_{j}(s)}\right)dW_{i}(s),\quad t\in [0,T],
\end{split}
\end{equation}
$i\in\{1,\ldots,N\}$, which is the relaxed version of \eqref{EqPrelimitDynamics}. The mapping
\[
	(\nu,\bar{\nu})\mapsto \Psi(\nu;\bar{\nu})
\]
defines a continuous function $\mathcal{Y}\times \mathcal{Y} \to \prbms[2]{\mathcal{Z}^{N}}$. The continuity of $\Psi$ can be checked by using a martingale problem characterization of solutions to \eqref{EqPrelimitDynamicsRel}; cf.\ \citet{elkarouietalii87, kushner90}, and also Section~\ref{SectLimitSystems} below. Define a function $J\!: \mathcal{Y}\times \mathcal{Y} \rightarrow [0,\infty)$ by
\[
\begin{split}
	&J(\nu;\bar{\nu})\\
	&\doteq \Mean_{\Psi(\nu;\bar{\nu})}\left[ \int_{\Gamma\times [0,T]} f\left(s,\hat{X}_{1}(s),\hat{\mu}(s),\gamma\right)d\hat{\rho}^{1}(d\gamma,ds) + F\left(\hat{X}_{1}(T),\hat{\mu}(T) \right) \right],
\end{split}
\]
where $\hat{\mu}(s)\doteq \frac{1}{N} \sum_{j=1}^{N} \delta_{\hat{X}_{j}(s)}$ and $(\hat{X}_{1},\ldots,\hat{X}_{N})$, $(\hat{\rho}^{1},\ldots,\hat{\rho}^{N})$ are components of the canonical process on $\mathcal{Z}^{N}$ with the obvious interpretation. Thanks to the continuity of $\Psi$ and the boundedness and continuity of $f$, $F$, we have that $J$ is a continuous mapping on $\mathcal{Y}\times \mathcal{Y}$. On the other hand, for any fixed $\bar{\nu}\in \mathcal{Y}$, all $\nu,\tilde{\nu}\in \mathcal{Y}$, all $\lambda\in [0,1]$,
\begin{align*}
	\Psi\left(\lambda\nu + (1-\lambda)\tilde{\nu}; \bar{\nu}\right) &= \lambda\Psi(\nu;\bar{\nu}) + (1-\lambda)\Psi(\tilde{\nu};\bar{\nu}), \\
	J\left(\lambda\nu + (1-\lambda)\tilde{\nu}; \bar{\nu}\right) &= \lambda J(\nu;\bar{\nu}) + (1-\lambda) J(\tilde{\nu};\bar{\nu}).
\end{align*}

Define a function $\chi\!: \mathcal{Y}\rightarrow \Borel{\mathcal{Y}}$ by
\[
	\chi(\bar{\nu})\doteq \left\{\nu \in \mathcal{Y} : J(\nu;\bar{\nu}) = \min_{\tilde{\nu}\in \mathcal{Y}} J(\tilde{\nu};\bar{\nu}) \right\}.
\]
Observe that $\chi(\bar{\nu})$ is non-empty, compact and convex for every $\bar{\nu}\in \mathcal{Y}$. Thus, $\chi$ is well-defined as a mapping from $\mathcal{Y}$ to $\mathcal{K}(\mathcal{Y})$, the set of all non-empty compact convex subsets of $\mathcal{Y}$. Moreover, $\chi$ is upper semicontinuous in the sense that $\nu\in \chi(\bar{\nu})$ whenever $(\nu_{n})\subset \mathcal{Y}$, $(\bar{\nu}_{n})\subset \mathcal{Y}$ are sequences such that $\lim_{n\to\infty} \bar{\nu}_{n} = \bar{\nu}$, $\lim_{n\to\infty} \nu_{n} = \nu$, and $\nu_{n}\in \chi(\bar{\nu}_{n})$ for each $n\in\mathbb{N}$ (recall that $\mathcal{Y}$ is metrizable). We are therefore in the situation of Theorem~1 in \citet{fan52}, which guarantees the existence of a fixed point for $\chi$, that is, there exists $\nu_{\ast}\in \mathcal{Y}$ such that $\nu_{\ast}\in \chi(\nu_{\ast})$.

\textbf{Second step}. Let $\epsilon > 0$, and let $\nu_{\ast}\in \mathcal{Y}$ be such that $\nu_{\ast}\in \chi(\nu_{\ast})$. Let $\mathrm{d}_{\mathcal{Y}}$ be a compatible metric on the compact Polish space $\mathcal{Y}$, and define a corresponding metric on $\mathcal{Y}\times \mathcal{Y}$ by $\mathrm{d}_{\mathcal{Y}\times\mathcal{Y}}((\nu,\bar{\nu}),(\mu,\bar{\mu}))\doteq \mathrm{d}_{\mathcal{Y}}(\nu,\mu) + \mathrm{d}_{\mathcal{Y}}(\bar{\nu},\bar{\mu})$. Choose a stochastic basis $((\Omega,\mathcal{F},\Prb),(\mathcal{F}_{t}))$ satisfying the usual hypotheses and carrying a vector $\boldsymbol{W} = (W_{1},\ldots,W_{N})$ of independent $d_{1}$-dimensional $(\mathcal{F}_{t})$-adapted Wiener processes, a vector $\boldsymbol{\rho} = (\rho^{1},\ldots,\rho^{N})$ of $(\mathcal{F}_{t})$-adapted $\mathcal{R}$-valued random variables as well as a vector $\boldsymbol{\xi} = (\xi_{1},\ldots,\xi_{N})$ of $\mathbb{R}^{d}$-valued $\mathcal{F}_{0}$-measurable random variables such that
\[
	\Prb\circ \left(\boldsymbol{\xi},\boldsymbol{\rho},\boldsymbol{W} \right)^{-1} = \otimes^{N} \nu_{\ast}.
\]
For $i\in \{1,\ldots,N\}$, let $(\mathcal{F}^{\circ,i}_{t})$ be the filtration generated by $\xi_{i}$, $\rho^{i}$, $W_{i}$, that is, $\mathcal{F}^{\circ,i}_{t}\doteq \sigma(\xi_{i},\rho^{i}(s),W_{i}(s) : s\leq t)$, $t\in [0,T]$. By independence and a version of the chattering lemma \citep[for instance, Theorem~3.5.2 in][p.\,59]{kushner90}, for every $\delta > 0$, there exists a vector $\boldsymbol{\rho}^{\delta} = (\rho^{\delta,1},\ldots,\rho^{\delta,N})$ of $\mathcal{R}$-valued random variables such that:
\begin{enumerate}[(i)]
	\item for every $i\in \{1,\ldots,N\}$, $\rho^{\delta,i}$ is the relaxed control induced by a piece-wise constant $(\mathcal{F}^{\circ,i}_{t})$-progressively measurable $\Gamma$-valued process;
	
	\item the random variables $(\xi_{1},\rho^{\delta,1},W_{1}),\ldots,(\xi_{N},\rho^{\delta,N},W_{N})$ are independent and identically distributed;
	
	\item setting $\nu_{\delta}\doteq \Prb\circ (\xi_{1},\rho^{\delta,1},W_{1})^{-1}$, we have $\mathrm{d}_{\mathcal{Y}}(\nu_{\delta},\nu_{\ast}) \leq \delta$.
\end{enumerate}
Since $J$ is continuous on the compact space $\mathcal{Y}\times \mathcal{Y}$, it is uniformly continuous. We can therefore find $\delta = \delta(\epsilon) > 0$ such that
\begin{equation} \label{EqNashCosts}
	\left|J(\nu_{\delta};\nu_{\delta}) - J(\nu_{\ast};\nu_{\ast})\right| + \max_{\nu\in \mathcal{Y}} \left|J(\nu;\nu_{\delta}) - J(\nu;\nu_{\ast})\right| \leq \epsilon.
\end{equation}

The law $\nu_{\delta}$ (with $\delta = \delta(\epsilon)$) and the corresponding product measure can be reproduced on the stochastic basis of the $N$-player game. More precisely, there exists a measurable function $\psi\!: [0,T]\times [0,1]\times \mathbb{R}^{d}\times \mathcal{W} \rightarrow \Gamma$ such that, upon setting
\begin{align*}
	& u_{i}(t,\omega)\doteq \psi\left(t,\vartheta^{N}_{i}(\omega),\xi^{N}_{i}(\omega),W^{N}_{i}(\cdot,\omega) \right),& & (t,\omega)\in [0,T]\times \Omega_{N},&
\end{align*}
the following hold:
\begin{enumerate}[(i)]
	\item $u_{i}\in \mathcal{H}_{2}((\mathcal{F}^{N,i}_{t}),\Prb_{N}; \Gamma)$ for every $i\in \{1,\ldots,N\}$;
	
	\item $(\xi^{N}_{1},u_{1},W^{N}_{1}),\ldots,(\xi^{N}_{N},u_{N},W_{N})$, interpreted as $\mathbb{R}^{d}\times \mathcal{R}\times \mathcal{W}$-valued random variables, are independent and identically distributed;
	
	\item $\Prb_{N}\circ (\xi^{N}_{1},u_{1},W^{N}_{1})^{-1} = \nu_{\delta}$.
\end{enumerate}
The relaxed controls $\rho^{\delta,1},\ldots,\rho^{\delta,1}$ are, in fact, induced by $\Gamma$-valued processes that may be taken to be piece-wise constant in time with respect to a common equidistant grid in $[0,T]$. Existence of a function $\psi$ with the desired properties can therefore be established by iteration along the grid points, repeatedly invoking Theorem~6.10 in \citet[p.\,112]{kallenberg01} on measurable transfers; this procedure also yields progressive measurability of $\psi$.

Set $\boldsymbol{u}\doteq (u_{1},\ldots,u_{N})$ with $u_{i}\in \mathcal{H}_{2}((\mathcal{F}^{N,i}_{t}),\Prb_{N}; \Gamma)$ as above. Then
\[
	J^{N}_{1}(\boldsymbol{u}) = J(\nu_{\delta};\nu_{\delta}).
\]
Let $v\in \mathcal{H}_{2}((\mathcal{F}^{N,1}_{t}),\Prb_{N}; \Gamma)$, and set $\nu\doteq \Prb_{N}\circ (\xi^{N}_{1},v,W^{N}_{1})^{-1}$, where $v$ is identified with its relaxed control. By independence and construction,
\[
	J^{N}_{1}\left([\boldsymbol{u}^{-1},v]\right) = J(\nu;\nu_{\delta}).
\]
Now, thanks to \eqref{EqNashCosts} and the equilibrium property of $\nu_{\ast}$,
\begin{align*}
	& J(\nu;\nu_{\delta}) - J(\nu_{\delta};\nu_{\delta}) \\
	&= J(\nu;\nu_{\delta}) - J(\nu;\nu_{\ast}) + J(\nu_{\ast};\nu_{\ast}) - J(\nu_{\delta};\nu_{\delta}) + J(\nu;\nu_{\ast}) - J(\nu_{\ast};\nu_{\ast}) \\
	&\geq -\epsilon.
\end{align*}
It follows that 
\[
	J^{N}_{1}(\boldsymbol{u}) \leq J^{N}_{1}\left([\boldsymbol{u}^{-1},v]\right) + \epsilon \quad\text{for all }v\in \mathcal{H}_{2}((\mathcal{F}^{N,1}_{t}),\Prb_{N}; \Gamma).
\]
This establishes the local approximate equilibrium property of the strategy vector $\boldsymbol{u}$ with respect to deviations in narrow strategies of player one. By symmetry, the property also holds with respect to deviations of the other players. We conclude that $\boldsymbol{u}$ is a local $\epsilon$-Nash equilibrium.
\end{proof}

%-------

\section{Mean field games} \label{SectLimitSystems}

In order to describe the limit system for the $N$-player games introduced above, consider the stochastic integral equation
\begin{equation} \label{EqLimitDynamics}
\begin{split}
	X(t) &= X(0) + \int_{0}^{t} b\bigl(s,X(s),\mathfrak{m}(s),u(s)\bigr)ds \\
	&\quad + \int_{0}^{t} \sigma\bigl(s,X(s),\mathfrak{m}(s)\bigr)dW(s),\quad t\in [0,T],
\end{split}
\end{equation}
where $\mathfrak{m}\in \mathcal{M}_{2}$ is a flow of probability measures, $W$ a $d_{1}$-dimensional Wiener process defined on some stochastic basis, and
$u$ a $\Gamma$-valued square-integrable adapted process.
%Our assumptions on $b$, $\sigma$ guarantee that existence and uniqueness of solutions hold for Eq.~\eqref{EqLimitDynamics}. Thus, given a stochastic basis $((\Omega,\mathcal{F},\Prb),(\mathcal{F}_{t}))$ rich enough to carry a $d_{1}$-dimensional $(\mathcal{F}_{t})$-Wiener process $W$, a strategy $u\in \mathcal{H}_{2}((\mathcal{F}_{t}),\Prb; \Gamma)$, and a flow of probability measures $\mathfrak{m}\in \mathcal{M}_{2}$ (as well as an additional $\mathcal{F}_{0}$-measurable random variable representing the initial condition if that is not deterministic), there exists a continuous $\mathbb{R}^{d}$-valued $(\mathcal{F}_{t})$-adapted process $X$ such that Eq.~\eqref{EqLimitDynamics} holds $\Prb$-almost surely, and $X$ is unique (up to $\Prb$-indistinguishability) among all continuous $(\mathcal{F}_{t})$-adapted solutions $\tilde{X}$ with $\tilde{X}(0) = X(0)$ $\Prb$-almost surely.

The notion of solution of the mean field game we introduce here makes use of a version of Eq.~\eqref{EqLimitDynamics} involving relaxed controls and varying stochastic bases. Given a flow of measures $\mathfrak{m} \in \mathcal{M}_{2}$, consider the stochastic integral equation
\begin{equation} \label{EqLimitDynamicsRel}
\begin{split}
	X(t) &=  X(0) + \int_{\Gamma\times[0,t]} b\bigl(s,X(s),\mathfrak{m}(s),\gamma\bigr)\rho(d\gamma,ds) \\
	&\quad + \int_{0}^{t} \sigma\bigl(s,X(s),\mathfrak{m}(s))\bigr)dW(s),\quad t\in [0,T].
\end{split}
\end{equation}
A solution of Eq.~\eqref{EqLimitDynamicsRel} with flow of measures $\mathfrak{m}\in \mathcal{M}_{2}$ is a quintuple
$((\Omega,\mathcal{F},\Prb),(\mathcal{F}_{t}),X,\rho,W)$ such that $(\Omega,\mathcal{F},\Prb)$ is a complete probability space, $(\mathcal{F}_{t})$ a filtration in $\mathcal{F}$ satisfying the usual hypotheses, $W$ a $d_{1}$-dimensional $(\mathcal{F}_{t})$-Wiener process, $\rho$ an $\mathcal{R}_{2}$-valued random variable adapted to $(\mathcal{F}_{t})$, and $X$ an $\mathbb{R}^{d}$-valued $(\mathcal{F}_{t})$-adapted continuous process satisfying Eq.~\eqref{EqLimitDynamicsRel} with flow of measures $\mathfrak{m}$ $\Prb$-almost surely. Under our assumptions on $b$ and $\sigma$, existence and uniqueness of solutions hold for Eq.~\eqref{EqLimitDynamicsRel} given any flow of measures $\mathfrak{m} \in \mathcal{M}_{2}$. Moreover, if $((\Omega,\mathcal{F},\Prb),(\mathcal{F}_{t}),X,\rho,W)$ is a solution, then the joint distribution of $(X,\rho,W)$ with respect to $\Prb$ can be identified with a probability measure on $\Borel{\mathcal{Z}}$. Conversely, the set of probability measures $\Theta\in \prbms{\mathcal{Z}}$ that correspond to a solution of Eq.~\eqref{EqLimitDynamicsRel} with respect to some stochastic basis carrying a $d_{1}$-dimensional Wiener process can be characterized through a local martingale problem. To this end, for $f\in \mathbf{C}^{2}(\mathbb{R}^{d}\times \mathbb{R}^{d_{1}})$, $\mathfrak{m}\in \mathcal{M}_{2}$, define the process $M_{f}^{\mathfrak{m}}$ on $(\mathcal{Z},\Borel{\mathcal{Z}})$ by
\begin{equation} \label{ExTestMartingale}
\begin{split}
	M_{f}^{\mathfrak{m}}\bigl(t,(\phi,r,w)\bigr) &\doteq f\bigl(\phi(t),w(t)\bigr) - f\bigl(\phi(0),0\bigr) \\
	&\quad - \int_{\Gamma\times [0,t]} \mathcal{A}^{\mathfrak{m}}_{\gamma,s}(f)\bigl(\phi(s),w(s)\bigr)\, r(d\gamma,ds),\; t\in [0,T],
\end{split}
\end{equation}
where
\begin{equation} \label{ExGenerator}
\begin{split}
	\mathcal{A}^{\mathfrak{m}}_{\gamma,s}(f)(x,y) &\doteq  \sum_{j=1}^{d} b_{j}\bigl(s,x,\mathfrak{m}(s),\gamma\bigr)\frac{\partial f}{\partial x_{j}}(x,y) \\
	&\quad + \frac{1}{2} \sum_{j=1}^{d}\sum_{k=1}^{d} (\sigma\trans{\sigma})_{jk}\bigl(s,x,\mathfrak{m}(s)\bigr) \frac{\partial^{2} f}{\partial x_{j}\partial x_{k}}(x,y)\\
	&\quad + \frac{1}{2} \sum_{l=1}^{d_{1}} \frac{\partial^{2} f}{\partial y_{l}^{2}}(x,y) + \sum_{k=1}^{d}\sum_{l=1}^{d_{1}} \sigma_{kl}\bigl(s,x,\mathfrak{m}(s)\bigr) \frac{\partial^{2} f}{\partial x_{k}\partial y_{l}}(x,y).
\end{split}
\end{equation}
Recall that $(\mathcal{G}_{t})$ denotes the canonical filtration in $\Borel{\mathcal{Z}}$ and $(\hat{X},\hat{\rho},\hat{W})$ the coordinate process on $\mathcal{Z}$. By construction, 
\[
	M_{f}^{\mathfrak{m}}(t) = f\bigl(\hat{X}(t),\hat{W}(t)\bigr) - f\bigl(\hat{X}(0),0\bigr) - \int_{\Gamma\times[0,t]}\! \mathcal{A}^{\mathfrak{m}}_{\gamma,s}(f)\bigl(\hat{X}(s),\hat{W}(s)\bigr)\hat{\rho}(d\gamma,ds),
\]
and $M_{f}^{\mathfrak{m}}$ is $(\mathcal{G}_{t})$-adapted.

\begin{defn} \label{DefSolution}
	A probability measure $\Theta\in \prbms{\mathcal{Z}}$ is called a \emph{solution of Eq.~\eqref{EqLimitDynamicsRel} with flow of measures $\mathfrak{m}$} if the following hold:
\begin{enumerate}[(i)]
	\item $\mathfrak{m}\in \mathcal{M}_{2}$;
	
	\item $\hat{W}(0) = 0$ $\Theta$-almost surely;
	
	\item \label{DefSolutionMart} $M_{f}^{\mathfrak{m}}$ is a local martingale with respect to the filtration $(\mathcal{G}_{t})$ and the probability measure $\Theta$ for every $f$ monomial of first or second order.
\end{enumerate}
\end{defn}

\begin{rem}
The test functions $f$ in \eqref{DefSolutionMart} of Definition~\ref{DefSolution} are the functions $\mathbb{R}^{d}\times \mathbb{R}^{d_{1}} \rightarrow \mathbb{R}$ given by $(x,y)\mapsto x_{j}$, $(x,y)\mapsto y_{l}$, $(x,y)\mapsto x_{j}\cdot x_{k}$, $(x,y)\mapsto y_{l}\cdot y_{\tilde{l}}$, and $(x,y)\mapsto x_{j}\cdot y_{l}$, where $j,k\in \{1,\ldots,d\}$, $l,\tilde{l}\in \{1,\ldots,d_{1}\}$.
\end{rem}

The following lemma justifies the terminology of Definition~\ref{DefSolution}.

\begin{lemma} \label{LemmaMPCharacterization}
Let $\mathfrak{m}\in \mathcal{M}_{2}$. If $((\Omega,\mathcal{F},\Prb),(\mathcal{F}_{t}),X,\rho,W)$ is a solution of Eq.~\eqref{EqLimitDynamicsRel} with flow of measures $\mathfrak{m}$, then $\Theta\doteq \Prb\circ (X,\rho,W)^{-1} \in \prbms{\mathcal{Z}}$ is a solution of Eq.~\eqref{EqLimitDynamicsRel} with flow of measures $\mathfrak{m}$ in the sense of Definition~\ref{DefSolution}.

Conversely, if $\Theta\in \prbms{\mathcal{Z}}$ is a solution of Eq.~\eqref{EqLimitDynamicsRel} with flow of measures $\mathfrak{m}$ in the sense of Definition~\ref{DefSolution}, then the quintuple $((\mathcal{Z},\mathcal{G}^{\Theta},\Theta), (\mathcal{G}^{\Theta}_{t+}), \hat{X}, \hat{\rho}, \hat{W})$ is a solution of Eq.~\eqref{EqLimitDynamicsRel} with flow of measures $\mathfrak{m}$, where $\mathcal{G}^{\Theta}$ is the $\Theta$-completion of $\mathcal{G} \doteq \Borel{\mathcal{Z}}$ and $(\mathcal{G}^{\Theta}_{t+})$ the right-continuous version of the $\Theta$-augmentation of the canonical filtration $(\mathcal{G}_{t})$.
\end{lemma}

\begin{proof}
The first part of the assertion is a consequence of It{\^o}'s formula and the local martingale property of the stochastic integral. The local martingale property of $M^{\mathfrak{m}}_{f}$ clearly holds for any $f\in \mathbf{C}^{2}(\mathbb{R}^{d}\times \mathbb{R}^{d_{1}})$.

The proof of the second part is similar to the proof of Proposition~5.4.6 in \citet[pp.\,315-316]{karatzasshreve91}, though here we do not need to extend the probability space; see Appendix~\ref{AppMPCharacterization} below.
\end{proof}

A particular class of solutions of Eq.~\eqref{EqLimitDynamicsRel} in the sense of Definition~\ref{DefSolution} are those where the flow of measures $\mathfrak{m}\in \mathcal{M}_{2}$ is induced by the probability measure $\Theta\in \prbms{\mathcal{Z}}$ in the sense that $\mathfrak{m}(t)$ coincides with the law of $\hat{X}(t)$ under $\Theta$. We call those solutions McKean-Vlasov solutions:

\begin{defn} \label{DefMKVSolution}
A probability measure $\Theta\in \prbms{\mathcal{Z}}$ is called a \emph{McKean-Vlasov solution of Eq.~\eqref{EqLimitDynamicsRel}} if there exists $\mathfrak{m}\in \mathcal{M}_{2}$ such that
\begin{enumerate}[(i)]
	\item $\Theta$ is a solution of Eq.~\eqref{EqLimitDynamicsRel} with flow of measures $\mathfrak{m}$;

	\item $\Theta\circ (\hat{X}(t))^{-1} = \mathfrak{m}(t)$ for every $t\in [0,T]$.
\end{enumerate}
\end{defn}

\begin{rem} \label{RemMcKeanVlasov}
If $\Theta\in \prbms[2]{\mathcal{Z}}$, then the induced flow of measures is in $\mathcal{M}_{2}$. More precisely, let $\Theta\in \prbms[2]{\mathcal{Z}}$ and set $\mathfrak{m}(t)\doteq \Theta\circ (\hat{X}(t))^{-1}$, $t\in [0,T]$. By definition of $\prbms[2]{\mathcal{Z}}$ and the metric $\mathrm{d}_{\mathcal{Z}}$,
\[
	\Mean_{\Theta}\left[ \|\hat{X}\|_{\mathcal{X}}^{2} \right] = \int_{\mathcal{Z}} \|\phi\|_{\mathcal{X}}^{2} \Theta(d\phi,dr,dw) < \infty.
\]
This implies, in particular, that $\mathfrak{m}(t)\in \prbms[2]{\mathbb{R}^{d}}$ for every $t\in [0,T]$. By construction and definition of the square Wasserstein metric, for all $s,t\in [0,T]$,
\[
	\mathrm{d}_{2}\left(\mathfrak{m}(t),\mathfrak{m}(s)\right)^{2} \leq \Mean_{\Theta}\left[ |\hat{X}(t)-\hat{X}(s)|^{2} \right].
\]
Continuity of the trajectories of $\hat{X}$ and the dominated convergence theorem with $2\|\hat{X}\|_{\mathcal{X}}^{2}$ as dominating $\Theta$-integrable random variable imply that $\mathrm{d}_{2}\left(\mathfrak{m}(t),\mathfrak{m}(s)\right) \to 0$ whenever $|t-s|\to 0$. It follows that $\mathfrak{m}\in \mathcal{M}_{2}$.
\end{rem}

Uniqueness holds not only for solutions of Eq.~\eqref{EqLimitDynamicsRel} with fixed flow of measures $\mathfrak{m}\in \mathcal{M}_{2}$, but also for McKean-Vlasov solutions of Eq.~\eqref{EqLimitDynamicsRel}.

\begin{lemma} \label{LemmaMcKeanVlasovUnique}
Let $\Theta, \tilde{\Theta}\in \prbms[2]{\mathcal{Z}}$. If $\Theta$, $\tilde{\Theta}$ are McKean-Vlasov solutions of Eq.~\eqref{EqLimitDynamicsRel} such that $\Theta\circ (\hat{X}(0),\hat{\rho},\hat{W})^{-1} = \tilde{\Theta}\circ (\hat{X}(0),\hat{\rho},\hat{W})^{-1}$, then $\Theta = \tilde{\Theta}$.
\end{lemma}

\begin{proof}
Let $\Theta, \tilde{\Theta}\in \prbms[2]{\mathcal{Z}}$ be McKean-Vlasov solutions of Eq.~\eqref{EqLimitDynamicsRel} such that $\Theta\circ (\hat{X}(0),\hat{\rho},\hat{W})^{-1} = \tilde{\Theta}\circ (\hat{X}(0),\hat{\rho},\hat{W})^{-1}$. Set
\begin{align*}
	& \mathfrak{m}(t)\doteq \Theta\circ \hat{X}(t)^{-1},& &\tilde{\mathfrak{m}}(t)\doteq \tilde{\Theta}\circ \hat{X}(t)^{-1},& &t\in [0,T].&
\end{align*}
In view of Remark~\ref{RemMcKeanVlasov}, we have $\mathfrak{m}, \tilde{\mathfrak{m}}\in \mathcal{M}_{2}$. Define an extended canonical space $\bar{\mathcal{Z}}$ by
\[
	\bar{\mathcal{Z}}\doteq \mathcal{X}\times \mathcal{X}\times \mathcal{R}_{2}\times \mathcal{W}.
\]
Let $(\bar{\mathcal{G}})_{t\geq 0}$ denote the canonical filtration in $\bar{\mathcal{G}}\doteq \Borel{\bar{\mathcal{Z}}}$, and let $(X,\tilde{X},\hat{\rho},\hat{W})$ be the canonical process. A construction analogous to the one used in the proof of Proposition~1 in \citet{yamadawatanabe71} (also see Section~5.3.D in \citet{karatzasshreve91}) yields a measure $Q\in \prbms{\bar{\mathcal{Z}}}$ such that
\begin{align*}
	& Q\circ (X,\hat{\rho},\hat{W})^{-1} = \Theta, & & Q\circ (\tilde{X},\hat{\rho},\hat{W})^{-1} = \tilde{\Theta}, & & Q\left\{ X(0) = \tilde{X}(0) \right\} = 1.&
\end{align*}
By Lemma~\ref{LemmaMPCharacterization}, $((\bar{\mathcal{Z}},\bar{\mathcal{G}}^{Q},Q), (\bar{\mathcal{G}}^{Q}_{t+}), X, \hat{\rho}, \hat{W})$, $((\bar{\mathcal{Z}},\bar{\mathcal{G}}^{Q},Q), (\bar{\mathcal{G}}^{Q}_{t+}), \tilde{X}, \hat{\rho}, \hat{W})$ are solutions of Eq.~\eqref{EqLimitDynamicsRel} with flow of measures $\mathfrak{m}$ and $\tilde{\mathfrak{m}}$, respectively, where $\bar{\mathcal{G}}^{Q}$ is the $Q$-completion of $\bar{\mathcal{G}}$ and $(\bar{\mathcal{G}}^{Q}_{t+})$ the right-continuous version of the $Q$-augmentation of $(\bar{\mathcal{G}}_{t})$.

By construction and definition of the square Wasserstein distance,
\[
	\mathrm{d}_{2}\left(\mathfrak{m}(t), \tilde{\mathfrak{m}}(t)\right)^{2}\leq \Mean_{Q}\left[ \left| X(t) - \tilde{X}(t) \right|^{2} \right] \quad\text{for all }t\in [0,T].
\]
Using \hypref{HypLipschitz}, H{\"o}lder's inequality, It{\^{o}}'s isometry, Fubini's theorem and the fact that $X(0) = \tilde{X}(0)$ $Q$-almost surely, we find that for every $t\in [0,T]$,
\begin{align*}
	&\Mean_{Q}\left[ \left| X(t) - \tilde{X}(t) \right|^{2} \right]  \\
%	\begin{split}
%	&\leq 4\Mean_{Q}\left[ \int_{\Gamma\times [0,t]} \hat{\rho}(d\gamma,ds)\cdot \int_{\Gamma\times [0,t]} \left| b\bigl(s,X(s),\mathfrak{m}(s),\gamma\bigr) - b\bigl(s,\tilde{X}(s),\tilde{\mathfrak{m}}(s),\gamma\bigr) \right|^{2} d\tilde{\rho
%}(d\gamma,ds) \right] \\
%	&\quad + 4\Mean_{Q}\left[ \int_{0}^{t} \left| \sigma\bigl(s,X(s),\mathfrak{m}(s)\bigr) - \sigma\bigl(s,\tilde{X}(s),\tilde{\mathfrak{m}}(s)\bigr) \right|^{2} ds \right]
%	\end{split}\\
	&\leq 4(T+1)L^{2} \int_{0}^{t} \Mean_{Q}\left[ \left| X(s) - \tilde{X}(s) \right|^{2} + \mathrm{d}_{2}\left(\mathfrak{m}(s), \tilde{\mathfrak{m}}(s)\right)^{2} \right] ds \\
	&\leq 8(T+1)L^{2} \int_{0}^{t} \Mean_{Q}\left[ \left| X(s) - \tilde{X}(s) \right|^{2} \right] ds.
\end{align*}
Gronwall's lemma and the continuity of trajectories imply that $X = \tilde{X}$ $Q$-almost surely and that $\mathfrak{m} = \tilde{\mathfrak{m}}$. It follows that $\Theta = \tilde{\Theta}$.
\end{proof}

Define the costs associated with a flow of measures $\mathfrak{m} \in \mathcal{M}_{2}$, an initial distribution $\nu \in \prbms{\mathbb{R}^{d}}$ and a probability measure $\Theta\in \prbms{\mathcal{Z}}$ by
\begin{equation*}
\begin{split}
	&\hat{J}(\nu,\Theta;\mathfrak{m}) \\
	&\doteq \begin{cases}
	\Mean_{\Theta}\left[ \int_{\Gamma\times [0,T]} f\bigl(s,\hat{X}(s),\mathfrak{m}(s),\gamma\bigr) \hat{\rho}(d\gamma,ds) + F\bigl(\hat{X}(T),\mathfrak{m}(T)\bigr) \right] \\
	\qquad\text{if $\Theta$ is a solution of Eq.~\eqref{EqLimitDynamicsRel} with flow of measures $\mathfrak{m}$}\\
	\qquad\text{and }\Theta\circ \hat{X}(0)^{-1} = \nu, \\
	\infty\quad \text{otherwise.}
	\end{cases}
\end{split}
\end{equation*}
This defines a measurable mapping $\hat{J}\!: \prbms{\mathbb{R}^{d}}\times \prbms{\mathcal{Z}}\times \mathcal{M}_{2}\rightarrow [0,\infty]$. The corresponding value function $\hat{V}\!: \prbms{\mathbb{R}^{d}}\times \mathcal{M}_{2} \rightarrow [0,\infty]$ is given by
\[
	\hat{V}(\nu;\mathfrak{m})\doteq \inf_{\Theta\in \prbms{\mathcal{Z}}} \hat{J}(\nu,\Theta;\mathfrak{m}).
\]

\begin{defn} \label{DefMFGSol}
	A pair $(\Theta,\mathfrak{m})$ is called a \emph{solution of the mean field game} if the following hold:
\begin{enumerate}[(i)]
	\item $\mathfrak{m}\in \mathcal{M}_{2}$, $\Theta\in \prbms{\mathcal{Z}}$, and $\Theta$ is a solution of Eq.~\eqref{EqLimitDynamicsRel} with flow of measures $\mathfrak{m}$;
			
	\item Mean field condition: $\Theta\circ \hat{X}(t)^{-1} = \mathfrak{m}(t)$ for every $t\in [0,T]$;
	
	\item Optimality condition: $\hat{J}(\mathfrak{m}(0),\Theta;\mathfrak{m}) \leq \hat{J}(\mathfrak{m}(0),\tilde{\Theta};\mathfrak{m})$ for every $\tilde{\Theta}\in \prbms{\mathcal{Z}}$.
\end{enumerate}
\end{defn}

%If $(\Theta,\mathfrak{m})$ is a solution of the mean field game, then $\mathfrak{m}(0) = \Theta\circ \hat{X}(0)^{-1}$ is called the \emph{initial distribution} of the solution. If $\mathfrak{m}(0) = \delta_{x}$ for some $x\in \mathbb{R}^{d}$, then $x$ is called the \emph{initial state} of the solution.

In Definition~\ref{DefMFGSol}, there is some redundancy in the choice of the pair $(\Theta,\mathfrak{m})$ as solution of the mean field game in that, thanks to the mean field condition, the flow of measures $\mathfrak{m}$ is completely determined by the probability measure $\Theta$. Consequently, we may call a probability measure $\Theta\in \prbms{\mathcal{Z}}$ a \emph{solution of the mean field game} if the pair $(\Theta,\mathfrak{m})$ is a solution of the mean field game in the sense of Definition~\ref{DefMFGSol} where $\mathfrak{m}$ is the flow of measures induced by $\Theta$, that is, $\mathfrak{m}(t)\doteq \Theta\circ \hat{X}(t)^{-1}$, $t\in [0,T]$.

If $\Theta$ is a solution of the mean field game, then, again thanks to the mean field condition, it is also a McKean-Vlasov solution of Eq.~\eqref{EqLimitDynamicsRel}. In general, however, $\Theta$ is not optimal as a controlled McKean-Vlasov solution. In the optimality condition of Definition~\ref{DefMFGSol}, in fact, the flow of measures is frozen at the flow of measures induced by $\Theta$, while in an optimization problem of McKean-Vlasov type the flow of measures would have to vary with the controlled solution.

\begin{rem}
The use of relaxed controls in Definition~\ref{DefMFGSol} has a twofold motivation. The first is pragmatic and well known \citep[for instance,][]{elkarouietalii87, kushner90}, namely the fact that relaxed controls allow one to embed the space of control processes into a nice topological space (if $\Gamma$ is compact, then $\mathcal{R} = \mathcal{R}_{2}$ is compact; for unbounded $\Gamma$, $\mathcal{R}_{2}$ is still Polish) without changing the minimal costs. In particular, existence of optimal controls is guaranteed in the space of relaxed controls. The second motivation is related to this fact, but more conceptual. The mean field condition in the mean field game is required to hold for the law of the state process under an optimal control only. Thus, existence of optimal controls (for a given flow of measures) is crucial for the existence of solutions to the mean field game. For ordinary optimal control problems, on the other hand, it suffices that the minimal costs be well defined. Still, it is natural to ask for conditions ensuring that a solution of the mean field game can be obtained in ordinary control processes, not just in relaxed controls. Sufficient conditions of this kind have been established in \citet{lacker15}. One simple sufficient condition is that the dynamics be linear and the costs convex in the control.
\end{rem}

The next lemma, the proof of which is based on time discretization and dynamic programming, will be an essential ingredient in the construction of competitor strategies in the proof of Theorem~\ref{ThConnection} below.

\begin{lemma} \label{LemmaNoiseFeedback}
Let $\mathfrak{m}\in \mathcal{M}_{2}$. Given any $\epsilon > 0$, there exists a measurable function $\psi^{\mathfrak{m}}_{\epsilon}: [0,T]\times \mathbb{R}^{d}\times \mathcal{W} \rightarrow \Gamma$ such that the following hold:
\begin{enumerate}[(i)]
	\item $\psi^{\mathfrak{m}}_{\epsilon}$ is progressively measurable in the sense that, for every $t\in [0,T]$, every $x\in \mathbb{R}^{d}$, we have $\psi^{\mathfrak{m}}_{\epsilon}(t,x,w) = \psi^{\mathfrak{m}}_{\epsilon}(t,x,\tilde{w})$ whenever $w(s) = \tilde{w}(s)$ for all $s\in [0,t]$;
	
	\item $\psi^{\mathfrak{m}}_{\epsilon}$ takes values in a finite subset of $\Gamma$;

	\item $\hat{J}(\mathfrak{m}(0),\Theta^{\mathfrak{m}}_{\epsilon};\mathfrak{m}) \leq \hat{V}(\mathfrak{m}(0);\mathfrak{m}) + \epsilon$, where $\Theta^{\mathfrak{m}}_{\epsilon}$ is the unique probability measure in $\prbms[2]{\mathcal{Z}}$ such that $\Theta^{\mathfrak{m}}_{\epsilon}$ is a solution of Eq.~\eqref{EqLimitDynamicsRel} with flow of measures $\mathfrak{m}$, $\Theta^{\mathfrak{m}}_{\epsilon}\circ (\hat{X}(0))^{-1} = \mathfrak{m}(0)$, and
\[
	\hat{\rho}(d\gamma,dt) = \delta_{\psi^{\mathfrak{m}}_{\epsilon}\left(t,\hat{X}(0),\hat{W}\right)}(d\gamma)\,dt\quad \Theta^{\mathfrak{m}}_{\epsilon}\text{-almost surely.}
\]
\end{enumerate}
\end{lemma}

\begin{proof}
Fix $\mathfrak{m}\in \mathcal{M}_{2}$, and set, for $(t,x,\gamma) \in [0,T]\times \mathbb{R}^{d}\times \Gamma$,
\begin{align*}
	& b_{\mathfrak{m}}(t,x,\gamma)\doteq b\bigl(t,x,\mathfrak{m}(t),\gamma\bigr),& & \sigma_{\mathfrak{m}}(t,x)\doteq \sigma\bigl(t,x,\mathfrak{m}(t)\bigr),& \\
	& f_{\mathfrak{m}}(t,x,\gamma)\doteq f\bigl(t,x,\mathfrak{m}(t),\gamma\bigr),& & F_{\mathfrak{m}}(x)\doteq F\bigl(x,\mathfrak{m}(T)\bigr).&
\end{align*}
Thanks to assumptions \hypref{HypMeasCont}, \hypref{HypLipschitz}, \hypref{HypCostLip}, and the continuity of $\mathfrak{m}$, we have that $b_{\mathfrak{m}}$, $\sigma_{\mathfrak{m}}$, $f_{\mathfrak{m}}$ are continuous in the time and control variable, uniformly over compact subsets of $\mathbb{R}^{d}$, $b_{\mathfrak{m}}$, $\sigma_{\mathfrak{m}}$ are globally Lipschitz continuous in the state variable, uniformly in the other variables, and $f_{\mathfrak{m}}$, $F_{\mathfrak{m}}$ are locally Lipschitz continuous in the state variable, uniformly in the other variables, with local Lipschitz constants that grow sublinearly in the state variable. 

The function $\psi^{\mathfrak{m}}_{\epsilon}$ will be constructed based on the principle of dynamic programming applied in discrete time. To this end, we first introduce an original control problem corresponding to the minimal costs $\hat{V}(\cdot,\mathfrak{m})$, then we build a sequence of approximating optimal control problems by successively restricting the set of admissible strategies. The proof proceeds in six steps.

\textbf{First step}. Let $\mathcal{U}$ be the set of all quadruples $((\Omega,\mathcal{F},\Prb),(\mathcal{F}_{t}),\rho,W)$ such that the pair $((\Omega,\mathcal{F},\Prb),(\mathcal{F}_{t}))$ forms a stochastic basis satisfying the usual hypotheses, $W$ is a $d_{1}$-dimensional $(\mathcal{F}_{t})$-Wiener process, and $\rho$ is an $(\mathcal{F}_{t})$-adapted $\mathcal{R}_{2}$-valued random variable such that $\Mean\left[ \int_{\Gamma\times[0,T]} |\gamma|^{2} \rho(d\gamma,ds) \right] < \infty$. For simplicity, we may write $\rho\in \mathcal{U}$ instead of $((\Omega,\mathcal{F},\Prb),(\mathcal{F}_{t}),\rho,W)\in \mathcal{U}$. Given any $\rho\in \mathcal{U}$, $(t_{0},x)\in [0,T]\times \mathbb{R}^{d}$, the stochastic integral equation
\begin{equation} \label{EqLimitDynamicsRelm}
\begin{split}
	X(t) &=  x + \int_{\Gamma\times[0,t]} b_{\mathfrak{m}}\bigl(t_{0}+s,X(s),\gamma\bigr)\rho(d\gamma,ds) \\
	&\quad + \int_{0}^{t} \sigma_{\mathfrak{m}}\bigl(t_{0}+s,X(s)\bigr)dW(s),\quad t\in [0,T-t_{0}],
\end{split}
\end{equation}
has a unique solution $X = X^{t_{0},x,\rho}$, that is, $X$ is the unique (up to indistinguishability with respect to $\Prb$) $\mathbb{R}^{d}$-valued $(\mathcal{F}_{t})$-adapted continuous process that satisfies \eqref{EqLimitDynamicsRelm} with $\Prb$-probability one. Although the solution $X$ of Eq.~\eqref{EqLimitDynamicsRelm} starts in $x$ at time zero, it corresponds to the solution of Eq.~\eqref{EqLimitDynamicsRel} starting in $x$ at time $t_{0}$. Define the costs associated with strategy $\rho$ and initial condition $(t_{0},x)\in [0,T]\times \mathbb{R}^{d}$ by
\[
	J_{\mathfrak{m}}(t_{0},x,\rho)\doteq \Mean\left[ \int_{\Gamma\times [0,T-t_{0}]}\! f_{\mathfrak{m}}\bigl(t_{0}+s,X(s),\gamma\bigr) \rho(d\gamma,ds) + F_{\mathfrak{m}}\bigl(X(T-t_{0})\bigr) \right],
\]
where $X = X^{t_{0},x,\rho}$. The corresponding value function $V_{\mathfrak{m}}$ is given by
\[
	V_{\mathfrak{m}}(t,x)\doteq \inf_{\rho\in \mathcal{U}} J_{\mathfrak{m}}(t,x,\rho),
\]
which is well-defined as a measurable function $[0,T]\times \mathbb{R}^{d} \rightarrow [0,\infty)$. Actually, $V_{\mathfrak{m}}$ is continuous. For $x\in \mathbb{R}^{d}$, $\rho\in \mathcal{U}$, set
\[
	\Theta^{x,\rho}\doteq \Prb\circ (X^{0,x,\rho},\rho,W)^{-1}.
\]
Then $\Theta^{x,\rho}$ is a solution of Eq.~\eqref{EqLimitDynamicsRel} with flow of measures $\mathfrak{m}$ and
\[
	J_{\mathfrak{m}}(0,x,\rho) = \hat{J}(\delta_{x},\Theta^{x,\rho};\mathfrak{m}).
\]
Conversely, in view of Lemma~\ref{LemmaMPCharacterization} and thanks to Assumption~\hypref{HypCostCoercivity}, any $\Theta\in \prbms{\mathcal{Z}}$ with $\hat{J}(\delta_{x},\Theta;\mathfrak{m}) < \infty$ induces a strategy $\rho\in \mathcal{U}$ such that $\Theta^{x,\rho} = \Theta$. It follows that $V_{\mathfrak{m}}(0,x) = \hat{V}(\delta_{x};\mathfrak{m})$ for every $x\in \mathbb{R}^{d}$ and, by conditioning on the initial state at time zero, 
\[
	\int_{\mathbb{R}^{d}} V_{\mathfrak{m}}(0,x)\, \mathfrak{m}(0)(dx) = \hat{V}(\mathfrak{m}(0);\mathfrak{m}).
\]

\textbf{Second step}. The function $V_{\mathfrak{m}}(0,\cdot)$ is locally Lipschitz continuous. To be more precise, choose $c_{0} > 0$, $\Gamma_{0} \subset \Gamma$ according to \hypref{HypCostCoercivity}, and let $r_{0} > 0$ be such that $\Gamma_{0} \subset \left\{ \gamma\in \mathbb{R}^{d_{2}} : |\gamma| \leq r_{0} \right\}$. We are going to show that there exists a constant $C_{1}\in (0,\infty)$ depending only on $K$, $L$, $T$, $\mathfrak{m}$, $r_{0}$, and $c_{0}$ such that
\begin{equation} \label{EqValueFnctLocalLip}
	\left| V_{\mathfrak{m}}(0,x) - V_{\mathfrak{m}}(0,\tilde{x})\right| \leq C_{1}  \left(1 + R\right) \left|x - \tilde{x}\right| \text{ whenever } |x|\vee|\tilde{x}| \leq R.
\end{equation}

To establish \eqref{EqValueFnctLocalLip}, set, for $\epsilon > 0$, $R > 0$,
\[
	\mathcal{U}_{\epsilon,R}\doteq \left\{ \rho\in \mathcal{U} : J_{\mathfrak{m}}(0,x;\rho) \leq V_{\mathfrak{m}}(0,x) + \epsilon \text{ for some }x \text{ with } |x|\leq R \right\}.
\]
Then for all $x, \tilde{x}\in \mathbb{R}^{d}$ with $|x|\vee |\tilde{x}| \leq R$,
\[
	\left| V_{\mathfrak{m}}(0,x) - V_{\mathfrak{m}}(0,\tilde{x})\right| \leq \inf_{\epsilon > 0} \sup_{\rho\in \mathcal{U}_{\epsilon, R}} \left| J_{\mathfrak{m}}(0,x;\rho) - J_{\mathfrak{m}}(0,\tilde{x}; \rho)\right|.
\]

Let $x, \tilde{x}\in \mathbb{R}^{d}$, $\rho \in \mathcal{U}$, and let $X$, $\tilde{X}$ be the solutions of \eqref{EqLimitDynamicsRelm} under $\rho$ with initial state $x$ and $\tilde{x}$, respectively. Using H{\"o}lder's inequality, Jensen's inequality, It{\^o}'s isometry, Fubini's theorem, assumption \hypref{HypLipschitz}, and Gronwall's lemma, we find that there exists a constant $C_{L,T}$ depending only on $L$, $T$ such that
\[
	\sup_{t\in [0,T]} \Mean\left[ \left| X(t) - \tilde{X}(t) \right|^{2} \right] \leq C_{L,T} \left|x - \tilde{x}\right|.
\]
Reusing the same tools but with assumption \hypref{HypGrowth} in place of \hypref{HypLipschitz} (also cf.\ Lemma~\ref{LemmaGrowthBounds}), we find that that there exists a constant $C_{K,T,\mathfrak{m}}$ depending only on $K$, $T$, and on $\mathfrak{m}$ (through $\sup_{t\in [0,T]} \int |y|^{2} \mathfrak{m}(t)(dy)$, which is finite since $\mathfrak{m}$ is continuous in time) such that
\[
	\sup_{t\in [0,T]} \Mean\left[ \left| X(t)\right|^{2} \right] \leq C_{K,T,\mathfrak{m}} \left(1 + |x|^{2} + \Mean\left[ \int_{\Gamma\times [0,T]} |\gamma|^{2} \rho(d\gamma,dt) \right] \right).
\]
Thanks to the above estimates and assumption \hypref{HypCostLip}, we have that there exist a constant $C_{L,T,\mathfrak{m}}$ depending only on $L$, $T$, and $\mathfrak{m}$, and a constant $C_{K,L,T,\mathfrak{m}}$ depending only on $K$, $L$, $T$, and $\mathfrak{m}$ such that
\begin{align*}
	&\left| J_{\mathfrak{m}}(0,x;\rho) - J_{\mathfrak{m}}(0,\tilde{x}; \rho)\right| \\
	&\leq C_{L,T,\mathfrak{m}} \left( 1 + \sup_{t\in [0,T]} \sqrt{\Mean\left[\left|X(t)\right|^{2} \right]} + \sup_{t\in [0,T]} \sqrt{\Mean\left[\left|\tilde{X}(t)\right|^{2} \right]}\right)\cdot \left|x - \tilde{x}\right| \\
	&\leq C_{K,L,T,\mathfrak{m}} \left( 1 + |x|\vee |\tilde{x}| + \sqrt{\Mean\left[ \int_{\Gamma\times [0,T]} |\gamma|^{2} \rho(d\gamma,dt) \right]} \right)\cdot \left|x - \tilde{x}\right|.
\end{align*}
It follows that for all $x, \tilde{x}\in \mathbb{R}^{d}$ with $|x|\vee |\tilde{x}| \leq R$,
\[
\begin{split}
	&\left| V_{\mathfrak{m}}(0,x) - V_{\mathfrak{m}}(0,\tilde{x})\right| \\
	&\leq C_{K,L,T,\mathfrak{m}}\cdot \inf_{\epsilon > 0} \left(1 + R + \sup_{\rho\in \mathcal{U}_{\epsilon, R}} \sqrt{\Mean\left[ \int_{\Gamma\times [0,T]} |\gamma|^{2} \rho(d\gamma,dt) \right]} \right) \cdot |x - \tilde{x}|.
\end{split}
\]

By the same estimates as above, but using \hypref{HypCostGrowth} instead of \hypref{HypCostLip}, we find that there exists a constant $\tilde{C}_{K,T,\mathfrak{m}}$ depending only on $K$, $T$, $\mathfrak{m}$ such that, for all $x\in \mathbb{R}^{d}$, all $\rho\in \mathcal{U}$,
\[
	J_{\mathfrak{m}}(0,x;\rho) \leq \tilde{C}_{K,T,\mathfrak{m}} \left( 1 + |x|^{2} + \Mean\left[ \int_{\Gamma\times [0,T]} |\gamma|^{2} \rho(d\gamma,dt) \right] \right).
\]
This implies that there exists a constant $C_{K,T,\mathfrak{m},\Gamma}$ depending only on $K$, $T$, $\mathfrak{m}$, and on $\Gamma$ (through $\min_{\gamma\in \Gamma} |\gamma|^{2}$) such that, for all $x\in \mathbb{R}^{d}$,
\[
	V_{\mathfrak{m}}(0,x) \leq C_{K,T,\mathfrak{m},\Gamma} \left( 1 + |x|^{2} \right).
\]
Let $\rho\in \mathcal{U}_{\epsilon,R}$ for some $\epsilon > 0$. Choose $x\in \mathbb{R}^{d}$ with $|x|\leq R$ such that $J_{\mathfrak{m}}(0,x;\rho) \leq V_{\mathfrak{m}}(0,x) + \epsilon$ (possible by definition of $\mathcal{U}_{\epsilon,R}$). By the coercivity assumption \hypref{HypCostCoercivity},
\begin{align*}
	J_{\mathfrak{m}}(0,x;\rho) &\geq c_{0} \Mean\left[ \int_{(\Gamma\setminus \Gamma_{0})\times [0,T]} |\gamma|^{2} \rho(d\gamma,dt) \right], \\
\intertext{hence}
	c_{0} \Mean\left[ \int_{(\Gamma\setminus \Gamma_{0})\times [0,T]} |\gamma|^{2} \rho(d\gamma,dt) \right] &\leq C_{K,T,\mathfrak{m},\Gamma} \left( 1 + R^{2} \right) + \epsilon.
\end{align*}
By construction,
\[
	\Mean\left[ \int_{\Gamma\times [0,T]} |\gamma|^{2} \rho(d\gamma,dt) \right] \leq T\cdot r_{0}^{2} + \Mean\left[ \int_{(\Gamma\setminus \Gamma_{0})\times [0,T]} |\gamma|^{2} \rho(d\gamma,dt) \right].
\]
It follows that there exists a constant $C_{K,T,\mathfrak{m},c_{0},r_{0}}$ depending only on $K$, $T$, $\mathfrak{m}$, $c_{0}$, and on $r_{0}$ (clearly, $\min_{\gamma\in \Gamma} |\gamma|^{2}\leq r_{0}^{2}$) such that
\[
	\sup_{\rho\in \mathcal{U}_{\epsilon,R}} \sqrt{\Mean\left[ \int_{\Gamma\times [0,T]} |\gamma|^{2} \rho(d\gamma,dt) \right]} \leq C_{K,T,\mathfrak{m},c_{0},r_{0}} \left(1 + R + \sqrt{\epsilon}\right).
\]
This establishes \eqref{EqValueFnctLocalLip}.

\textbf{Third Step}. For $M\in \mathbb{N}$, set $\Gamma_{M}\doteq \left\{\gamma \in \Gamma: |\gamma|\leq M\right\}$. For $M$ big enough, say $M \geq M_{0}$, $\Gamma_{M}$ is non-empty. Choose $\gamma_{0}\in \Gamma_{M_{0}}$, and set $\Gamma_{M}\doteq \{\gamma_{0}\}$ if $M < M_{0}$. Then, for every $M\in \mathbb{N}$, $\Gamma_{M}$ is compact (and non-empty) and $\Gamma_{M} \subset \Gamma_{M+1}$. Set
\[
	\mathcal{U}_{M}\doteq \left\{\rho \in \mathcal{U}: \rho(\Gamma_{M}\times [0,T]) = T\; \Prb\text{-almost surely}\right\},
\]
and let $V_{\mathfrak{m},M}$ be the value function defined with respect to $\mathcal{U}_{M}$ instead of $\mathcal{U}$. We claim that
\begin{equation} \label{EqValueFnctCompactConv}
	V_{\mathfrak{m},M}(0,\cdot) \stackrel{M\to\infty}{\searrow} V_{\mathfrak{m}}(0,\cdot)\text{ uniformly over compact subsets of }\mathbb{R}^{d}.
\end{equation}
Notice that, by construction, $V_{\mathfrak{m},M}(0,\cdot) \geq V_{\mathfrak{m},M+1}(0,\cdot) \geq V_{\mathfrak{m}}(0,\cdot)$ for every $M\in \mathbb{N}$. By Step~2, we know that $V_{\mathfrak{m}}(0,\cdot)$ is locally Lipschitz. Repeating the arguments of Step~2 (notice that $\mathcal{U}_{M} \subset \mathcal{U}$ by definition), we find that inequality \eqref{EqValueFnctLocalLip} also holds for $V_{\mathfrak{m},M}(0,\cdot)$ in place of $V_{\mathfrak{m}}(0,\cdot)$ and that the constant $C_{1}$ can be chosen independently of $M\in \mathbb{N}$. To establish \eqref{EqValueFnctCompactConv}, it is therefore enough to check that point-wise convergence holds. Fix $x\in \mathbb{R}^{d}$. It suffices to show that given $\rho\in \mathcal{U}$ there exits a sequence $(\rho^{(M)})\subset \mathcal{U}$ such that $\rho^{(M)}\in \mathcal{U}_{M}$ for every $M$ and $J_{\mathfrak{m}}(0,x;\rho^{(M)}) \to J_{\mathfrak{m}}(0,x;\rho)$ as $M\to \infty$.

Let $\rho\in \mathcal{U}$. For $M\in \mathbb{N}$, let $\rho^{(M)}\in \mathcal{U}_{M}$ be such that for every $B\in \Borel{\Gamma}$, every $I\in \Borel{[0,T]}$,
\[
	\rho^{(M)}(B\times I) = \rho((B\cap \Gamma_{M})\times I) + \rho((\Gamma\setminus \Gamma_{M}) \times I)\cdot \delta_{\gamma_{0}}(B).
\]
This determines a unique strategy $\rho^{(M)}\in \mathcal{U}_{M}$. Clearly, $\rho^{(M)}$ comes with the same stochastic basis as $\rho$. If $(\dot{\rho}_{t})$ is a version of the time derivative process associated with $\rho$ (thus, $\rho(d\gamma,dt) = \dot{\rho}_{t}(d\gamma)dt$), then a version of the time derivative process of $\rho^{(M)}$ is given by
\[
	\dot{\rho}^{(M)}_{t}(d\gamma) = \mathbf{1}_{\Gamma_{M}}(\gamma)\cdot \rho_{t}(d\gamma) + \rho_{t}(\Gamma\setminus \Gamma_{M})\cdot \delta_{\gamma_{0}}(d\gamma).
\]
Let $X$, $X^{(M)}$ be the solutions of \eqref{EqLimitDynamicsRelm} under $\rho$ and $\rho^{(M)}$, respectively. Thanks to H{\"o}lder's inequality, Jensen's inequality, It{\^o}'s isometry, Fubini's theorem, and assumption \hypref{HypLipschitz}, there exists a constant $C_{L,T}$ depending only on $L$, $T$ such that, for every $t\in [0,T]$,
\begin{multline*}
	\Mean\left[ \left| X(t) - X^{(M)}(t) \right|^{2} \right] \leq C_{L,T} \int_{0}^{t} \Mean\left[ \left| X(s) - X^{(M)}(s) \right|^{2} \right]ds \\
	+ C_{L,T} \Mean\left[ \left| \int_{\Gamma\times [0,t]} b_{\mathfrak{m}}\bigl(s,X(s),\gamma\bigr)\left(\rho^{(M)} - \rho \right)(d\gamma,ds) \right|^{2} \right].
\end{multline*}
Using the definition of $\rho^{(M)}$, H{\"o}lder's inequality and assumption \hypref{HypGrowth}, we find that, for some constant $C_{K,T,\mathfrak{m}}$ depending only on $K$, $T$ and $\mathfrak{m}$, 
\begin{align*}
	&\Mean\left[ \left| \int_{\Gamma\times [0,t]} b_{\mathfrak{m}}\bigl(s,X(s),\gamma\bigr)\left(\rho^{(M)} - \rho \right)(d\gamma,ds) \right|^{2} \right] \\
	\begin{split}
	&\leq 2T \Mean\left[ \int_{0}^{T} \int_{\Gamma\setminus \Gamma_{M}} \left| b_{\mathfrak{m}}\bigl(s,X(s),\gamma\bigr)\right|^{2}\dot{\rho}_{s}(d\gamma)ds \right] \\
	&\quad + 2\Mean\left[ \rho\left((\Gamma\setminus \Gamma_{M})\times [0,T] \right)\cdot  \int_{0}^{T} \left| b_{\mathfrak{m}}\bigl(s,X(s),\gamma_{0}\bigr)\right|^{2}ds \right]
	\end{split} \\
	\begin{split}
	&\leq C_{K,T,\mathfrak{m}}\Mean\left[ \rho\bigl((\Gamma\setminus \Gamma_{M})\times [0,T] \bigr)\cdot \left(1 + \sup_{r\in [0,T]} |X(r)|^{2}\right) \right] \\
	&\quad + C_{K,T,\mathfrak{m}}\Mean\left[ \int_{\Gamma\times [0,T]} \mathbf{1}_{\Gamma\setminus \Gamma_{M}}(\gamma)\cdot |\gamma|^{2} \rho(d\gamma,ds) \right].
	\end{split}
\end{align*}
By \hypref{HypGrowth} and the usual estimates, including Gronwall's lemma, we have $\Mean\left[\sup_{r\in [0,T]} |X(r)|^{2} \right] < \infty$. Since $\rho_{\omega}$ is a measure with total mass $T$ for every $\omega\in \Omega$, we have $\rho\bigl((\Gamma\setminus \Gamma_{M})\times [0,T] \bigr) \to 0$ as $M\to \infty$ $\Prb$-almost surely. This implies, by dominated convergence,
\[
	\Mean\left[ \rho\bigl((\Gamma\setminus \Gamma_{M})\times [0,T] \bigr)\cdot \left(1 + \sup_{r\in [0,T]} |X(r)|^{2}\right) \right] \stackrel{M\to\infty}{\longrightarrow} 0.
\]
On the other hand, $\Mean\left[ \int_{\Gamma\times [0,T]} |\gamma|^{2} \rho(d\gamma,ds) \right] < \infty$ by definition of $\mathcal{U}$. This means that
\[
	\Mean\left[ \int_{\Gamma\times [0,T]} \mathbf{1}_{\Gamma\setminus \Gamma_{M}}(\gamma)\cdot |\gamma|^{2} \rho(d\gamma,ds) \right] \stackrel{M\to\infty}{\longrightarrow} 0.
\]
An application of Gronwall's lemma now yields
\[
	\Mean\left[ \left| X(t) - X^{(M)}(t) \right|^{2} \right] \stackrel{M\to\infty}{\longrightarrow} 0.
\]
This convergence together with assumption \hypref{HypCostGrowth} (and an estimate completely analogous to the one above) implies that $J_{\mathfrak{m}}(0,x;\rho^{(M)}) \to J_{\mathfrak{m}}(0,x;\rho)$ as $M\to \infty$.

\textbf{Fourth Step}. Choose a family $(\Gamma_{M,k})_{M,k\in\mathbb{N}}$ of finite subsets of $\Gamma$ such that $\Gamma_{M,k}\subset \Gamma_{M,k+1}\subset \Gamma_{M}$, $\Gamma_{M,k}\subset \Gamma_{M+1,k}$, and $\min_{\tilde{\gamma} \in \Gamma_{M,k}} |\gamma - \tilde{\gamma}| \leq 1/k$ for any $\gamma\in \Gamma_{M}$. Let $\mathcal{U}_{M,k}$ be the set of all $\rho\in \mathcal{U}$ such that $\rho$ is the $\mathcal{R}_{2}$-valued random variable induced by a $\Gamma_{M,k}$-valued adapted process that is piece-wise constant in time with respect to the equidistant grid of step size $T\cdot 2^{-k}$. Thus, $((\Omega,\mathcal{F},\Prb),(\mathcal{F}_{t}),\rho,W)\in \mathcal{U}_{M,k}$ if and only if $\rho_{\omega}(d\gamma,dt) = \delta_{u(t,\omega)}(d\gamma)dt$ for $\Prb$-almost every $\omega\in \Omega$, where $u$ is a $\Gamma_{M,k}$-valued $(\mathcal{F}_{t})$-progressively measurable process with c{\`a}dl{\`a}g trajectories that are piece-wise constant over the grid $\{0,T\cdot 2^{-k}, 2T\cdot 2^{-k}, 3T\cdot 2^{-k}, \ldots,T\}$. Let $V_{\mathfrak{m},M,k}$ be the value function defined with respect to $\mathcal{U}_{M,k}$. Then, thanks to the continuity in time and control of the coefficients according to \hypref{HypMeasCont}, a version of the chattering lemma \citep[for instance, Theorem~3.5.2 in][p.\,59]{kushner90}, and the local Lipschitz continuity of $V_{\mathfrak{m},M,k}(0,\cdot)$, which holds uniformly in $k$ and $M$ (one repeats the arguments of Step~2), we find that
\[
	V_{\mathfrak{m},M,k}(0,\cdot) \stackrel{k\to\infty}{\searrow} V_{\mathfrak{m},M}(0,\cdot) \text{ uniformly over compact subsets of }\mathbb{R}^{d}.
\]
By \eqref{EqValueFnctCompactConv} and since $\mathcal{U}_{M,k}\subset \mathcal{U}_{M,k+1}\subset \mathcal{U}_{M}$ and $\mathcal{U}_{M,k}\subset \mathcal{U}_{M+1,k}$, it follows that
\begin{equation} \label{EqApproxValueFnctConv}
	V_{\mathfrak{m},M,M}(0,\cdot) \stackrel{M\to\infty}{\searrow} V_{\mathfrak{m}}(0,\cdot)\text{ uniformly over compact subsets of }\mathbb{R}^{d}.
\end{equation}

\textbf{Fifth step}. The value function $V_{\mathfrak{m},M,k}$ coincides with the value function of a discrete-time optimal control problem defined as follows. Set $h \doteq T\cdot 2^{-k}$. Thanks to Theorem~1 in \citet{kallenberg96} and because $\Gamma_{M,k}$ is finite, we find a measurable and universally predictable function
\[
	\Phi_{\mathfrak{m},M,k}\!: \mathbb{N}_{0}\times \mathbb{R}^{d}\times \Gamma_{M,k}\times \mathbf{C}([0,h],\mathbb{R}^{d_{1}})\rightarrow \mathbb{R}^{d}
\]
such that $\Phi_{\mathfrak{m},M,k}(j,x,\gamma,W) = X((j+1)h)$ $\Prb$-almost surely whenever $X$ is the unique strong solution to
\[
\begin{split}
	X(t) &= x + \int_{0}^{t} b_{\mathfrak{m}}\bigl(j\cdot h + s,X(s),\gamma\bigr)ds \\
	&\quad + \int_{0}^{t} \sigma_{\mathfrak{m}}\bigl(j\cdot h + s,X(s)\bigr)dW(s), \quad t\in [0,h],
\end{split}
\]
where $W$ is a $d_{1}$-dimensional standard Wiener process defined on some stochastic basis $((\Omega,\mathcal{F},\Prb),(\mathcal{F}_{t}))$. The function $\Phi_{\mathfrak{m},M,k}$ is the system function of the control problem in the sense of \citet{bertsekasshreve96}. Let $\bar{\mathcal{U}}_{M,k}$ denote the set of discrete-time Markov feedback strategies with values in $\Gamma_{M,k}$, that is, the set of all Borel measurable functions $v: \mathbb{N}_{0}\times \mathbb{R}^{d} \rightarrow \Gamma_{M,k}$. To describe the path-wise evolution of the system, choose a complete probability space $(\Omega_{\circ},\mathcal{F}^{\circ},\Prb_{\circ})$ rich enough to carry a $d_{1}$-dimensional standard Wiener process $W_{\circ}$. For $j\in \mathbb{N}_{0}$, set $\zeta_{j}\doteq (W(jh+s) - W(jh))_{s\in [0,h]}$, which defines a $\mathbf{C}([0,h],\mathbb{R}^{d_{1}})$-valued random variable. Given any Markov feedback strategy $v \in \bar{\mathcal{U}}_{M,k}$ and initial condition $(j,x) \in \{0,\ldots,2^{k}\}\times \mathbb{R}^{d}$, the corresponding state sequence is defined recursively, for each $\omega \in \Omega_{\circ}$, by
\begin{equation} \label{EqLimitDynamicsDiscrete}
\begin{split}
	X_{0}(\omega)&\doteq x, \\
	X_{l+1}(\omega)&\doteq \Phi_{\mathfrak{m},M,k}\left(j+l,X_{l}(\omega),v(j+l,X_{l}(\omega)),\zeta_{l}(\omega)\right),
\end{split}
\end{equation}
$l \in \{0,\ldots,2^{k}-j-1\}$. The associated costs are given by
\[
	\bar{J}_{\mathfrak{m},M,k}(j,x,v)\doteq \Mean_{\circ}\left[\sum_{l=0}^{2^{k}-j-1} f_{\mathfrak{m}}\bigl((j+l)h,X_{l},v(j+l,X_{l})\bigr)\cdot h + F_{\mathfrak{m}}(X_{k-j})\right],
\]
where $(X_{l})$ is the state sequence generated according to \eqref{EqLimitDynamicsDiscrete} with feedback strategy $v$ and initial condition $(j,x)$. Let $\bar{V}_{\mathfrak{m},M,k}$ be the value function of the control problem just defined:
\[
	\bar{V}_{\mathfrak{m},M,k}(j,x) \doteq \inf_{v\in \bar{\mathcal{U}}_{M,k}} \bar{J}_{\mathfrak{m},M,k}(j,x,v),\quad (j,x)\in \{0,\ldots,2^{k}\}\times \mathbb{R}^{d}.
\]

By Proposition~8.6 in \citet[pp.\,209-210]{bertsekasshreve96}, the principle of dynamic programming applies to $\bar{V}_{\mathfrak{m},M,k}$. This has two consequences. First, notice that any feedback strategy $v\in \bar{\mathcal{U}}_{M,k}$ induces, for any given initial condition $(j,x)\in \{0,\ldots,2^{k}\}\times \mathbb{R}^{d}$, a relaxed control variable $\rho\in \mathcal{U}_{M,k}$ such that
\[
	\bar{J}_{\mathfrak{m},M,k}(j,x,v) = J_{\mathfrak{m}}(jh,x,\rho).
\]
This implies $\bar{V}_{\mathfrak{m},M,k}(j,x) \geq V_{\mathfrak{m},M,k}(jh,x)$ for all $(j,x)\in \{0,\ldots,2^{k}\}\times \mathbb{R}^{d}$. Since $\bar{V}_{\mathfrak{m},M,k}(2^{k},\cdot) = F_{\mathfrak{m}}(\cdot) = V_{\mathfrak{m},M,k}(2^{k}h,\cdot)$, it follows by dynamic programming for $\bar{V}_{\mathfrak{m},M,k}$ and backward induction that
\[
	\bar{V}_{\mathfrak{m},M,k}(j,x) = V_{\mathfrak{m},M,k}(jh,x)\quad  \text{for all } (j,x)\in \{0,\ldots,2^{k}\}\times \mathbb{R}^{d}.
\]
As a second consequence of the principle of dynamic programming, there exists an optimal Markov feedback strategy. More precisely, we can choose $v_{\ast} \in \bar{\mathcal{U}}_{M,k}$ such that, for every $(j,x)\in \{0,\ldots,2^{k}\}\times \mathbb{R}^{d}$,
\[
\begin{split}
	v_{\ast}(j,x)\in &\argmin_{\gamma\in \Gamma_{M,k}}\Bigl\{f_{\mathfrak{m}}(jh,x,\gamma)\cdot h \\
	&\quad + \int_{\mathbf{C}([0,h],\mathbb{R}^{d_{1}})} \bar{V}_{\mathfrak{m},M,k}\bigl(j+1,\Phi_{\mathfrak{m},M,k}(j,x,\gamma,y)\bigr)\, \eta_{h}(dy) \Bigr\},
\end{split}
\]
where $\eta_{h}$ is standard Wiener measure on $\Borel{\mathbf{C}([0,h],\mathbb{R}^{d_{1}})}$. Then
\[
	\bar{J}_{\mathfrak{m},M,k}(j,x,v_{\ast}) = \bar{V}_{\mathfrak{m},M,k}(j,x)\quad \text{for all } (j,x)\in \{0,\ldots,2^{k}\}\times \mathbb{R}^{d}.
\]

\textbf{Sixth step}. Define a function $\psi^{\mathfrak{m}}_{M,k}: [0,T]\times \mathbb{R}^{d}\times \mathcal{W} \rightarrow \Gamma_{M,k}$ as follows. Let $x\in \mathbb{R}^{d}$, $w\in \mathcal{W}$. In analogy with \eqref{EqLimitDynamicsDiscrete}, recursively define a sequence $(x_{j})_{j\in \{0,\ldots,2^{k}\}}$ by
\begin{align*}
	& x_{0}\doteq x,& & x_{j+1}\doteq \Phi_{\mathfrak{m},M,k}\left(j,x_{j},v_{\ast}(j,x_{j}),(w(jh+s)-w(jh))_{s\in [0,h]}\right).&
\end{align*}
For $j\in \{0,\ldots,2^{k}-1\}$, $s\in [0,h)$, set
\[
	\psi^{\mathfrak{m}}_{M,k}(jh+s,x,w)\doteq v_{\ast}(j,x_{j}),
\]
and set $\psi^{\mathfrak{m}}_{M,k}(T,x,w)\doteq v_{\ast}(2^{k},x_{k})$. By construction, $\psi^{\mathfrak{m}}_{M,k}$ is progressively measurable with values in a finite set. Let $((\Omega,\mathcal{F},\Prb),(\mathcal{F}_{t}))$ be a stochastic basis rich enough to carry a $d_{1}$-dimensional $(\mathcal{F}_{t})$-Wiener process $W$ and an $\mathbb{R}^{d}$-valued $\mathcal{F}_{0}$-measurable random variable $\xi$ such that $\Prb\circ \xi^{-1} = \mathfrak{m}(0)$. For every $x\in \mathbb{R}^{d}$, the process $\psi^{\mathfrak{m}}_{M,k}(t,x,W)$ induces a relaxed control random variable $\rho$ such that $((\Omega,\mathcal{F},\Prb),(\mathcal{F}_{t}),\rho,W)\in \mathcal{U}_{M,k}$ and $J_{\mathfrak{m}}(0,x,\rho) = V_{\mathfrak{m},M,k}(0,x)$. Let $\rho^{M,k}$ be the relaxed control random variable in $\mathcal{U}_{M,k}$ induced by the process $\psi^{\mathfrak{m}}_{M,k}(t,\xi,W)$. Let $X_{M,k}$ be the unique continuous $(\mathcal{F}_{t})$-adapted process such that $X_{M,k}(0) = \xi$ and $((\Omega,\mathcal{F},\Prb),(\mathcal{F}_{t}),X_{M,k},\rho^{M,k},W)$ is a solution of Eq.~\eqref{EqLimitDynamicsRel} with flow of measures $\mathfrak{m}$. Set
\[
	\Theta^{\mathfrak{m}}_{M,k}\doteq \Prb\circ (X_{M,k},\rho^{M,k},W)^{-1}.
\]
Then $\Theta^{\mathfrak{m}}_{M,k}\in \prbms[2]{\mathcal{Z}}$ and $\Theta^{\mathfrak{m}}_{M,k}$ is a solution of Eq.~\eqref{EqLimitDynamicsRel} with flow of measures $\mathfrak{m}$ such that $\Theta^{\mathfrak{m}}_{M,k}\circ (\hat{X}(0))^{-1} = \mathfrak{m}(0)$, $\hat{\rho}(d\gamma,dt) = \delta_{\psi^{\mathfrak{m}}_{M,k}\left(t,\hat{X}(0),\hat{W}\right)}(d\gamma)dt$ with probability one under $\Theta^{\mathfrak{m}}_{M,k}$, and
\[
	\hat{J}(\mathfrak{m}(0),\Theta^{\mathfrak{m}}_{M,k};\mathfrak{m}) = \int_{\mathbb{R}^{d}} V_{\mathfrak{m},M,k}(0,x) \mathfrak{m}(0)(dx) < \infty.
\]
By \eqref{EqApproxValueFnctConv} and dominated convergence, it follows that
\[
	\hat{J}(\mathfrak{m}(0),\Theta^{\mathfrak{m}}_{M,M};\mathfrak{m}) \stackrel{M\to\infty}{\searrow} \hat{V}(\mathfrak{m}(0);\mathfrak{m}).
\]
Hence, given any $\epsilon > 0$, there exists $M(\epsilon)\in \mathbb{N}$ such that, for all $M \geq M(\epsilon)$, $\hat{J}(\mathfrak{m}(0),\Theta^{\mathfrak{m}}_{M,M};\mathfrak{m}) \leq \hat{V}(\mathfrak{m}(0);\mathfrak{m}) + \epsilon$.
\end{proof}

\begin{rem}
	The conditions of Lemma~\ref{LemmaNoiseFeedback} do not determine $\psi^{\mathfrak{m}}_{\epsilon}$ in a unique way. On the other hand, once $\psi^{\mathfrak{m}}_{\epsilon}$ has been constructed, the probability measure $\Theta^{\mathfrak{m}}_{\epsilon}$ is uniquely determined as the law of the solution of Eq.~\eqref{EqLimitDynamicsRel} with flow of measures $\mathfrak{m}$, initial distribution $\mathfrak{m}(0)$ and control process $u$ given by $u(t)\doteq \psi^{\mathfrak{m}}_{\epsilon}(t,X(0),W)$, $t\in [0,T]$, where $W$ is the driving Wiener process and $u$ is identified with its relaxed control random variable. Notice that $u$ is square-integrable since $\psi^{\mathfrak{m}}_{\epsilon}$ takes values in a finite subset of $\Gamma$. 
\end{rem}

%-----

\section{Convergence of Nash equilibria} \label{SectConvergence}

For $N\in \mathbb{N}$, let $u^{N}_{1},\ldots,u^{N}_{N}\in \mathcal{H}_{2}((\mathcal{F}^{N}_{t}),\Prb_{N};\Gamma)$ be individual strategies for the $N$-player game, and let $\boldsymbol{u}^{N} \doteq (u^{N}_{1},\ldots,u^{N}_{N})$ be the corresponding strategy vector. Let $Q^{N}$ be the normalized occupation measure associated with $\boldsymbol{u}^{N}$. More precisely, $Q^{N}$ is the $\prbms[2]{\mathcal{Z}}$-valued random variable determined by setting, for $B\in \Borel{\mathcal{X}}$, $R\in \Borel{\mathcal{R}_{2}}$, $D\in \Borel{\mathcal{W}}$,
\begin{equation} \label{ExOccupationMeasure}
	Q^{N}_{\omega}(B\times R\times D) \doteq \frac{1}{N} \sum_{i=1}^{N} \delta_{X^{N}_{i}(\cdot,\omega)}(B)\cdot \delta_{\rho^{N,i}_{\omega}}(R)\cdot \delta_{W^{N}_{i}(\cdot,\omega)}(D),\; \omega\in \Omega_{N},
\end{equation}
where $(X^{N}_{1},\ldots,X^{N}_{N})$ is the solution of the system of equations \eqref{EqPrelimitDynamics} under strategy vector $\boldsymbol{u}^{N}$, and $\rho^{N,i}$ is the relaxed control associated with individual strategy $u^{N}_{i}$, $i\in \{1,\ldots,N\}$.

Convergence results will be obtained under the hypothesis that
\begin{equation} \label{CondTightness}
\tag{T}
	\exists\,\delta_{0} > 0:\; \sup_{N\in\mathbb{N}} \Mean_{N}\left[ \frac{1}{N} \sum_{i=1}^{N}\left( |\xi^{N}_{i}|^{2+\delta_{0}} + \int_{0}^{T} |u^{N}_{i}(t)|^{2+\delta_{0}}dt \right) \right] < \infty.
\end{equation}
Whenever \eqref{CondTightness} holds, we will---as we may---suppose that $\delta_{0} \in (0,1\wedge T]$.

\begin{rem}
Condition~\eqref{CondTightness} is automatically satisfied if the action space $\Gamma$ is compact and the initial states, i.e., the random variables $\xi^{N}_{i}$, $N\in \mathbb{N}$, $i\in \{1,\ldots,N\}$, are uniformly bounded.
\end{rem}

\begin{lemma} \label{LemmaTightness}
If condition~\eqref{CondTightness} holds, then the family $(\Prb_{N}\circ (Q^{N})^{-1})_{N\in\mathbb{N}}$ is pre-compact in $\prbms{\prbms[2]{\mathcal{Z}}}$.
\end{lemma}

\begin{proof}
We verify that condition~\eqref{CondTightness} implies the pre-compactness of the family $(\Prb_{N}\circ (Q^{N})^{-1})_{N\in\mathbb{N}}$ by using a suitable tightness function on $\prbms[2]{\mathcal{Z}}$. For a function $\psi$ on $[0,T]$ with values in $\mathbb{R}^{d}$ or $\mathbb{R}^{d_{1}}$, let $\mathbf{w}_{\psi}(\cdot,T)$ denote the modulus of continuity of $\psi$ on $[0,T]$, that is, the function
\[
	[0,\infty)\ni h \mapsto \mathbf{w}_{\psi}(h,T)\doteq \sup_{t,s\in [0,T]: |t-s|\leq h} |\psi(t)-\psi(s)| \in [0,\infty].
\]
If $\psi$ is continuous, then the modulus of continuity of $\psi$ takes values in $[0,\infty)$. Clearly, $\mathbf{w}_{\psi}(h,T) = \mathbf{w}_{\psi}(T,T)$ whenever $h > T$. Choose $\delta_{0} > 0$ according to condition~\eqref{CondTightness}, and set $\alpha\doteq \frac{\delta_{0}}{2(8+\delta_{0})}$. Define the function $g\!: \prbms[2]{\mathcal{Z}} \rightarrow [0,\infty]$ by
\begin{equation} \label{ExTFunctMeasures}
\begin{split}
	g(\Theta)&\doteq \int_{\mathcal{Z}}\left( \|\phi\|_{\mathcal{X}}^{2+\delta_{0}} + |w(0)|  + \int_{\Gamma\times [0,T]} |\gamma|^{2+\delta_{0}}\, r(d\gamma,dt) \right. \\
	&\qquad+ \left. \sup_{h\in (0,1]}\left\{ h^{-\alpha} \left(\mathbf{w}_{\phi}(h,T) + \mathbf{w}_{w}(h,T)\right) \right\} \right) \Theta(d\phi,dr,dw).
\end{split}
\end{equation}
Then $g$ is a tightness function on $\prbms[2]{\mathcal{Z}}$; see Appendix~\ref{AppTightnessFnctMeasures}. It is therefore enough to check that condition~\eqref{CondTightness} entails $\sup_{N\in \mathbb{N}} \Mean_{N}\left[ g(Q^{N}) \right] < \infty$. By definition of $Q^{N}$ and $g$,
\[
\begin{split}
	\Mean_{N}\left[ g(Q^{N}) \right] &= \frac{1}{N}\sum_{i=1}^{N} \Mean_{N}\left[ \|X^{N}_{i}\|_{\mathcal{X}}^{2+\delta_{0}} + \int_{0}^{T} |u^{N}_{i}(t)|^{2+\delta_{0}}dt \right]\\
	&\quad + \frac{1}{N}\sum_{i=1}^{N} \Mean_{N}\left[\sup_{h\in (0,1]}\left\{ h^{-\alpha} \left(\mathbf{w}_{X^{N}_{i}}(h,T) + \mathbf{w}_{W^{N}_{i}}(h,T)\right) \right\} \right].
\end{split}
\]
By Lemma~\ref{LemmaGrowthBounds2} and condition~\eqref{CondTightness},
\[
	\sup_{N\in \mathbb{N}} \left\{ \frac{1}{N}\sum_{i=1}^{N} \Mean_{N}\left[ \|X^{N}_{i}\|_{\mathcal{X}}^{2+\delta_{0}} + \int_{0}^{T} |u^{N}_{i}(t)|^{2+\delta_{0}}dt \right] \right\}< \infty.
\]
As to the terms involving the moduli of continuity, set $p\doteq 2 + \delta_{0}/2$; then, by monotonicity of $h\mapsto h^{-\alpha}$ and Markov's inequality (as well as Jensen's inequality),
\begin{align*}
	& \frac{1}{N}\sum_{i=1}^{N} \Mean_{N}\left[\sup_{h\in (0,1]}\left\{ h^{-\alpha} \left(\mathbf{w}_{X^{N}_{i}}(h,T) + \mathbf{w}_{W^{N}_{i}}(h,T)\right) \right\} \right] \\
	&\leq \frac{1}{N}\sum_{i=1}^{N} \Mean_{N}\left[\sup_{k\in \mathbb{N} : k \geq 1/T}\left\{ (k+1)^{\alpha} \left(\mathbf{w}_{X^{N}_{i}}\left(\tfrac{1}{k},T\right) + \mathbf{w}_{W^{N}_{i}}\left(\tfrac{1}{k},T\right) \right) \right\} \right] \\
%	&= \frac{1}{N}\sum_{i=1}^{N} \int_{0}^{\infty} \Prb_{N}\left( \sup_{k\in \mathbb{N}} (k+1)^{\alpha} \left(\mathbf{w}_{X^{N}_{i}}\left(\tfrac{1}{k},T\right) + \mathbf{w}_{W^{N}_{i}}\left(\tfrac{1}{k},T\right) \right) \geq M \right)dM \\
	&\leq 1 + \frac{1}{N}\sum_{i=1}^{N} \int_{1}^{\infty} \sum_{k=1}^{\infty} \Prb_{N}\left( \mathbf{w}_{X^{N}_{i}}\left(\tfrac{1}{k},T\right) + \mathbf{w}_{W^{N}_{i}}\left(\tfrac{1}{k},T\right) \geq \frac{M}{(k+1)^{\alpha}} \right)dM \\
	&\leq 1 + \sum_{k=1}^{\infty} (k+1)^{\alpha\cdot p}\left( \frac{1}{N}\sum_{i=1}^{N}\Mean_{N}\left[ \mathbf{w}_{X^{N}_{i}}\left(\tfrac{1}{k},T\right)^{p} + \mathbf{w}_{W^{N}_{i}}\left(\tfrac{1}{k},T\right)^{p} \right] \right) \frac{2^{p-1}}{p-1},
\end{align*}
where we have used that $\int_{1}^{\infty} M^{-p}\,dM = 1/(p-1)$ since $p > 1$. To find an upper bound for the above sums that does not depend on $N$, we employ estimates on the moments of the modulus of continuity of It{\^o} processes; cf.\ \citet{fischernappo10} and the references therein. Since $W^{N}_{1},\ldots,W^{N}_{N}$ are standard $d_{1}$-dimensional Wiener processes, we have by Lemma~3 of that paper and H{\"o}lder's inequality that there exists a finite constant $\bar{C}_{p,d_{1}}$ depending only on $p$ and $d_{1}$ such that, for every $i\in \{1,\ldots,N\}$, every $k\in \mathbb{N}$ with $k \geq 1/T$,
\[
	\Mean_{N}\left[ \mathbf{w}_{W^{N}_{i}}\left(\tfrac{1}{k},T\right)^{p} \right] \leq \bar{C}_{p,d_{1}} \left( \frac{\log(2T k)}{k}\right)^{p/2}.
\]
Recall that $p = 2 + \delta_{0}/2$. By Theorem~1 in \citet{fischernappo10}, there exists a finite constant $\bar{C}_{\delta_{0},d,d_{1}}$ depending only on $\delta_{0}$ (through $p = 2 + \delta_{0}/2$ and $\delta_{0}/2 = 2 + \delta_{0} - p$), $d$, and $d_{1}$ such that, for every $k\in \mathbb{N}$ with $k \geq 1/T$,
\begin{align*}
	&\frac{1}{N}\sum_{i=1}^{N} \Mean_{N}\left[ \mathbf{w}_{X^{N}_{i}}\left(\tfrac{1}{k},T\right)^{2+\delta_{0}/2} \right] \\
	\begin{split}
	&\leq \bar{C}_{\delta_{0},d,d_{1}} \left( \frac{\log(2T k)}{k}\right)^{1+\delta_{0}/4}  \\
	&\quad \cdot\left(\frac{1}{N}\sum_{i=1}^{N} \Mean_{N}\left[\sup_{s,t\in [0,T]:s<t} \left(\frac{\int_{s}^{t}\left|b\bigl(\tilde{s},X^{N}_{i}(\tilde{s}),\mu^{N}(\tilde{s}),u^{N}_{i}(\tilde{s})\bigr)\right|d\tilde{s}}{\sqrt{|t-s|}}\right)^{2+\delta_{0}/2}\right] \right.\\
	&\qquad \left. + \frac{1}{N}\sum_{i=1}^{N} \Mean_{N}\left[\sup_{s\in [0,T]} \left|\sigma\bigl(s,X^{N}_{i}(s),\mu^{N}(s)\bigr)\right|^{2+\delta_{0}} \right] + 1\right).
	\end{split}
\end{align*}
Thanks to assumption~\hypref{HypGrowth}, Lemma~\ref{LemmaGrowthBounds2} and condition~\eqref{CondTightness}, we have
\[
	\sup_{N\in\mathbb{N}} \left\{ \frac{1}{N}\sum_{i=1}^{N} \Mean_{N}\left[\sup_{s\in [0,T]} \left|\sigma\bigl(s,X^{N}_{i}(s),\mu^{N}(s)\bigr)\right|^{2+\delta_{0}} \right] \right\} < \infty.
\]
On the other hand, by H{\"o}lder's inequality,
\begin{align*}
	&\frac{1}{N}\sum_{i=1}^{N} \Mean_{N}\left[\sup_{s,t\in [0,T]:s<t} \left(\frac{\int_{s}^{t}\left|b\bigl(\tilde{s},X^{N}_{i}(\tilde{s}),\mu^{N}(\tilde{s}),u^{N}_{i}(\tilde{s})\bigr)\right|d\tilde{s}}{\sqrt{|t-s|}}\right)^{2+\delta_{0}/2}\right] \\
%	&\leq \frac{1}{N}\sum_{i=1}^{N} \Mean_{N}\left[ \left(\int_{0}^{T}\left|b\bigl(\tilde{s},X^{N}_{i}(\tilde{s}),\mu^{N}(\tilde{s}),u^{N}_{i}(\tilde{s})\bigr)\right|^{2}d\tilde{s}\right)^{1+\delta_{0}/4}\right] \\
	&\leq T^{\delta_{0}/4}\cdot \frac{1}{N}\sum_{i=1}^{N} \Mean_{N}\left[ \int_{0}^{T}\left|b\bigl(\tilde{s},X^{N}_{i}(\tilde{s}),\mu^{N}(\tilde{s}),u^{N}_{i}(\tilde{s})\bigr)\right|^{2+\delta_{0}/2} d\tilde{s} \right],
\end{align*}
and, thanks to assumption~\hypref{HypGrowth}, Lemma~\ref{LemmaGrowthBounds} and condition~\eqref{CondTightness},
\[
	\sup_{N\in \mathbb{N}}\left\{ \frac{1}{N}\sum_{i=1}^{N} \Mean_{N}\left[ \int_{0}^{T}\left|b\bigl(\tilde{s},X^{N}_{i}(\tilde{s}),\mu^{N}(\tilde{s}),u^{N}_{i}(\tilde{s})\bigr)\right|^{2+\delta_{0}/2} d\tilde{s} \right] \right\} < \infty.
\]
Recall that $\alpha = \frac{\delta_{0}}{2(8+\delta_{0})}$ and $p = 2 + \delta_{0}/2$. It follows that, for some finite constant $\bar{C}_{K,T,\delta_{0},d,d_{1}}$ not depending on $N$,
\begin{multline*}
	\sup_{N\in \mathbb{N}} \left\{ \frac{1}{N}\sum_{i=1}^{N} \Mean_{N}\left[\sup_{h\in (0,1]}\left\{ h^{-\alpha} \left(\mathbf{w}_{X^{N}_{i}}(h,T) + \mathbf{w}_{W^{N}_{i}}(h,T)\right) \right\} \right] \right\}\\
	\leq \bar{C}_{K,T,\delta_{0},d,d_{1}}\left(1 + \sum_{k=1}^{\infty} (k+1)^{\alpha\cdot p}\left( \frac{\log(2T k)}{k}\right)^{p/2} \right),
\end{multline*}
where the infinite sum on the right-hand side above has a finite limit since $p/2 - \alpha\cdot p = (8+2\delta_{0})/(8+\delta_{0}) > 1$.
\end{proof}

\bigskip
Below, we will use the symbol $\mathbb{I}$ to indicate the index set of a (convergent) subsequence; thus $\mathbb{I}$ is a subset of $\mathbb{N}$ with the natural ordering and $\#\mathbb{I} = \infty$.

\begin{lemma} \label{LemmaCoupling}
Suppose that $(\Prb_{n}\circ \xi^{n}_{i^{n}_{\ast}})_{n\in\mathbb{I}}$ converges in $\prbms[2]{\mathbb{R}^{d}}$ to some $\bar{\nu}\in \prbms[2]{\mathbb{R}^{d}}$, where, for each $n\in \mathbb{I}$, $i^{n}_{\ast} \in \{1,\ldots,n\}$. Then there exists a sequence $(\bar{\xi}^{n})_{n\in\mathbb{I}}$ of $\mathbb{R}^{d}$-valued random variables such that the following hold:
\begin{enumerate}[(i)] 
	\item for every $n\in \mathbb{I}$, $\bar{\xi}^{n}$ is defined on $(\Omega_{n},\mathcal{F}^{n})$, measurable with respect to $\sigma(\xi^{n}_{i^{n}_{\ast}},\vartheta^{n}_{i^{n}_{\ast}}) \subset \mathcal{F}^{n}_{0}$, and such that $\Prb_{n}\circ (\bar{\xi}^{n})^{-1} = \bar{\nu}$;
	
	\item $\Mean_{n} \left[ |\xi^{n}_{i^{n}_{\ast}} - \bar{\xi}^{n}|^{2}\right] \to 0$ as $n\to \infty$.
\end{enumerate}
\end{lemma}

\begin{proof}
Set $\nu_{n}\doteq \Prb_{n}\circ (\xi^{n}_{i^{n}_{\ast}})^{-1}$. By hypothesis,
\[
	\mathrm{d}_{2}\left(\nu_{n},\bar{\nu}\right) \stackrel{n\to\infty}{\longrightarrow} 0.
\]
Let $n\in \mathbb{I}$. By definition of the square Wasserstein metric,
\[
	\mathrm{d}_{2}\left(\nu_{n},\bar{\nu}\right)^{2} = \inf_{\alpha\in \prbms{\mathbb{R}^{d}\times \mathbb{R}^{d}}: [\alpha]_{1} = \nu_{n}\text{ and } [\alpha]_{2} = \bar{\nu}} \int_{\mathbb{R}^{d}\times \mathbb{R}^{d}} |x-\tilde{x}|^{2}\, \alpha(dx,d\tilde{x}).
\]
The infimum in the above equation is attained; see, for instance, Theorem~1.3 (Kantorovich's theorem) in \citet[pp.\,19-20]{villani03}. Thus, there exists $\alpha^{n}_{\ast}\in \prbms{\mathbb{R}^{d}\times \mathbb{R}^{d}}$ such that $[\alpha^{n}_{\ast}]_{1} = \nu_{n}$, $[\alpha^{n}_{\ast}]_{2} = \bar{\nu}$ and
\[
	\mathrm{d}_{2}\left(\nu_{n},\bar{\nu}\right)^{2} = \int_{\mathbb{R}^{d}\times \mathbb{R}^{d}} |x-\tilde{x}|^{2}\, \alpha^{n}_{\ast}(dx,d\tilde{x}).
\]
Recall that $\vartheta^{n}_{1},\ldots,\vartheta^{n}_{n}$ are independent $\mathcal{F}^{n}_{0}$-measurable random variables which are uniformly distributed on $[0,1]$ and independent of the $\sigma$-algebra generated by $\xi^{n}_{1},\ldots,\xi^{n}_{n}$, $W^{n}_{1},\ldots,W^{n}_{n}$. By Theorem~6.10 in \citet[p.\,112]{kallenberg01} on measurable transfers, there exists a measurable function $\phi_{n}\!: \mathbb{R}^{d}\times [0,1] \rightarrow \mathbb{R}^{d}$ such that
\[
	\Prb_{n}\circ \left(\xi^{n}_{i^{n}_{\ast}}, \phi_{n}(\xi^{n}_{i^{n}_{\ast}},\vartheta^{n}_{i^{n}_{\ast}}) \right)^{-1} = \alpha^{n}_{\ast}.
\] 
Set $\bar{\xi}^{n} \doteq \phi_{n}(\xi^{n}_{i^{n}_{\ast}},\vartheta^{n}_{i^{n}_{\ast}})$. Then $\bar{\xi}^{n}$ is $\sigma(\xi^{n}_{i^{n}_{\ast}},\vartheta^{n}_{i^{n}_{\ast}})$-measurable, $\Prb_{n}\circ \left(\bar{\xi}^{n}\right)^{-1} = \bar{\nu}$, and 
\[
	\Mean_{n} \left[ |\xi^{n}_{i^{n}_{\ast}} - \bar{\xi}^{n}|^{2}\right] = \mathrm{d}_{2}\left(\nu_{n},\bar{\nu}\right)^{2},
\]
which tends to zero as $n\to \infty$.
\end{proof}

\begin{lemma} \label{LemmaConvergence}
Grant condition~\eqref{CondTightness}. Let $(Q^{n})_{n\in\mathbb{I}}$ be a subsequence that converges in distribution to some $\prbms[2]{\mathcal{Z}}$-valued random variable $Q$ defined on some probability space $(\Omega,\mathcal{F},\Prb)$. Set
\[
	\mu_{\omega}(t)\doteq Q_{\omega}\circ \hat{X}(t)^{-1},\quad t\in [0,T],\; \omega\in \Omega.
\]
Then for $\Prb$-almost every $\omega\in \Omega$, $\mu_{\omega}\in \mathcal{M}_{2}$ and $Q_{\omega}$ is a solution of Eq.~\eqref{EqLimitDynamicsRel} with flow of measures $\mu_{\omega}$. Moreover,
\begin{equation*}
	\liminf_{\mathbb{I}\ni n\to\infty} \frac{1}{n} \sum_{i=1}^{n} J^{n}_{i}(\boldsymbol{u}^{n}) \geq \int_{\Omega} \hat{J}\bigl(\mu_{\omega}(0),Q_{\omega},\mu_{\omega}\bigr) \Prb(d\omega).
\end{equation*}
\end{lemma}

\begin{proof}
By Lemma~\ref{LemmaTightness}, $(\Prb_{N}\circ (Q^{N})^{-1})_{N\in\mathbb{N}}$ is pre-compact in $\prbms{\prbms[2]{\mathcal{Z}}}$. Let $(Q^{n})_{n\in\mathbb{I}}$ be a subsequence that converges in distribution to some $\prbms[2]{\mathcal{Z}}$-valued random variable $Q$, defined on some probability space $(\Omega,\mathcal{F},\Prb)$. Set $\mu_{\omega}(t)\doteq Q_{\omega}\circ \hat{X}(t)^{-1}$, $t\in [0,T]$, $\omega\in \Omega$. Since $Q_{\omega}\in \prbms[2]{\mathcal{Z}}$ for every $\omega\in \Omega$, we have $\mu_{\omega}\in \mathcal{M}_{2}$ for every $\omega\in \Omega$; cf.\ Remark~\ref{RemMcKeanVlasov} above. By construction, $\hat{W}(0) = 0$ $Q^{n}_{\omega}$-almost surely for $\Prb_{n}$-almost every $\omega\in \Omega_{n}$. Convergence in distribution implies $\hat{W}(0) = 0$ $Q_{\omega}$-almost surely for $\Prb$-almost every $\omega\in \Omega$.

In order to verify that $Q_{\omega}$ is a solution of Eq.~\eqref{EqLimitDynamicsRel} with flow of measures $\mu_{\omega}$ for $\Prb$-almost every $\omega\in \Omega$, it suffices to check that condition~\eqref{DefSolutionMart} of Definition~\ref{DefSolution} holds. The proof of this fact is analogous to the proof of Lemma~5.2 in \citet{budhirajaetal12}. Since the situation here is somewhat different, we give details in Appendix~\ref{AppLimitPoints} below.

The asymptotic lower bound for the average costs is a consequence of a version of Fatou's lemma \citep[cf.\ Theorem~A.3.12][p.\,307]{dupuisellis97} since, for every $n\in \mathbb{I}$,
\[
\begin{split}
\frac{1}{n} \sum_{i=1}^{n} J^{n}_{i}(\boldsymbol{u}^{n}) = \int_{\Omega_{n}} \int_{\mathcal{Z}} \left( \int_{\Gamma\times [0,T]} f\bigl(t,\phi(t),Q^{n}_{\omega}\circ \hat{X}(t)^{-1},\gamma\bigr)\, r(d\gamma,dt)\right. \\
 \left. F\bigl(T,\phi(T),Q^{n}_{\omega}\circ \hat{X}(T)^{-1}\bigr) \right) Q^{n}_{\omega}(d\phi,dr,dw)\, \Prb_{n}(d\omega)
\end{split}
\]
and $Q^{n}_{\omega}\circ \hat{X}(t)^{-1} \to \mu(t)$ in distribution as $n\to \infty$.
\end{proof}

\begin{rem}
Lemma~\ref{LemmaConvergence} shows that, under condition~\eqref{CondTightness}, all limit points of the normalized occupation measures $(Q^{N})_{N\in\mathbb{N}}$ are concentrated on those random variables that, with probability one, take values in the set of McKean-Vlasov solutions of Eq.~\eqref{EqLimitDynamicsRel}. The mean field condition of Definition~\ref{DefMFGSol} is therefore always satisfied.
\end{rem}

In addition to \eqref{CondTightness}, we will need the following weak symmetry condition on the costs:
\begin{align} \label{CondCostSymmetry}
\tag{S}
\begin{split}
	&\text{$\exists$ a sequence of indices }(i_{\ast}^{N})_{N\in\mathbb{N}} \text{ with } i_{\ast}^{N}\in \{1,\ldots,N\} \text{ such that}\\
	& \sup_{N\in \mathbb{N}} J^{N}_{i_{\ast}^{N}}(\boldsymbol{u}^{N}) < \infty \text{ and }  \limsup_{N\to\infty} \frac{1}{N}\sum_{i=1}^{N} J^{N}_{i}(\boldsymbol{u}^{N}) \leq \limsup_{N\to\infty} J^{N}_{i_{\ast}^{N}}(\boldsymbol{u}^{N}).
\end{split}
\end{align}

\begin{rem} \label{RemCostSymmetry}
Condition~\eqref{CondCostSymmetry} is automatically satisfied if the cost coefficients $f$, $F$ are bounded functions. If $f$, $F$ are unbounded and the costs associated with $\boldsymbol{u}^{N}$ are symmetric in the sense that, for every $N$, every $i\in \{2,\ldots,N\}$, $J^{N}_{1}(\boldsymbol{u}^{N}) = J^{N}_{i}(\boldsymbol{u}^{N})$, then, thanks to assumption~\hypref{HypCostGrowth} and Lemma~\ref{LemmaGrowthBounds}, condition~\eqref{CondCostSymmetry} follows from condition~\eqref{CondTightness}.
\end{rem}

\begin{thrm} \label{ThConnection}
Let $(\epsilon_{N})_{N\in\mathbb{N}} \subset [0,\infty)$ be a sequence converging to zero. Suppose that $(\boldsymbol{\xi}^{N})_{N\in\mathbb{N}}$ and $(\boldsymbol{u}^{N})_{N\in\mathbb{N}}$ are such that \eqref{CondTightness} and \eqref{CondCostSymmetry} hold and, for each $N\in \mathbb{N}$, $\boldsymbol{\xi}^{N} = (\xi^{N}_{1},\ldots,\xi^{N}_{N})$ is exchangeable and $\boldsymbol{u}^{N}$ is a local $\epsilon_{N}$-Nash equilibrium for the $N$-player game. Let $(Q^{n})_{n\in\mathbb{I}}$ be a subsequence that converges in distribution to some $\prbms[2]{\mathcal{Z}}$-valued random variable $Q$ defined on some probability space $(\Omega,\mathcal{F},\Prb)$. If there is $\mathfrak{m} \in \mathcal{M}_{2}$ such that, for $\Prb$-almost every $\omega\in \Omega$,
\[
	Q_{\omega}\circ \hat{X}(t)^{-1} = \mathfrak{m}(t),\quad t\in [0,T],
\]
then $(Q_{\omega},\mathfrak{m})$ is a solution of the mean field game for $\Prb$-almost every $\omega\in \Omega$.
\end{thrm}

We postpone the proof of Theorem~\ref{ThConnection} to the end of this section. The crucial hypothesis in Theorem~\ref{ThConnection} is the almost sure non-randomness of the flow of measures induced by a limit random variable $Q$. Thus, under the rather general conditions \eqref{CondTightness} and \eqref{CondCostSymmetry}, we prove convergence to solutions of a mean field game for subsequences with limit random variable $Q$ such that $\Prb\circ (Q\circ (\hat{X}(t))^{-1}_{t\in[0,T]})^{-1} = \delta_{\mathfrak{m}}$ for some $\mathfrak{m}\in \mathcal{M}_{2}$. This condition is reminiscent of the characterization of propagation of chaos in the Tanaka-Sznitman theorem. The non-randomness of the induced flow of measures is implied by the non-randomness of the joint law of initial condition, relaxed control and noise process, that is, by the condition $\Prb\circ (Q\circ (\hat{X}(0),\hat{\rho},\hat{W})^{-1})^{-1} = \delta_{\nu}$ for some $\nu \in \prbms{\mathbb{R}^{d}\times \mathcal{R}_{2}\times \mathcal{W}}$. This condition, in turn, is satisfied if the initial states and individual strategies of each $N$-player game are independent and identically distributed, where the marginal distributions are allowed to vary with $N$.

\begin{crll} \label{CorConnection}
Let $(\epsilon_{N})_{N\in\mathbb{N}} \subset [0,\infty)$ be a sequence converging to zero. Suppose that $(\boldsymbol{\xi}^{N})_{N\in\mathbb{N}}$ and $(\boldsymbol{u}^{N})_{N\in\mathbb{N}}$ are such that \eqref{CondTightness} holds and, for each $N\in \mathbb{N}$, $\boldsymbol{u}^{N}$ is a local $\epsilon_{N}$-Nash equilibrium for the $N$-player game and the random variables $(\xi^{N}_{1},u^{N}_{1},W^{N}_{1}),\ldots,(\xi^{N}_{N},u^{N}_{N},W^{N}_{N})$ are independent and identically distributed. Let $(Q^{n})_{n\in\mathbb{I}}$ be a subsequence that converges in distribution to some $\prbms[2]{\mathcal{Z}}$-valued random variable $Q$ defined on some probability space $(\Omega,\mathcal{F},\Prb)$. Then $Q_{\omega}$ is a solution of the mean field game for $\Prb$-almost every $\omega\in \Omega$.
\end{crll}

\begin{proof}
By distributional symmetry of the vectors of initial states and individual strategies, the costs are symmetric and condition~\eqref{CondTightness} entails condition~\eqref{CondCostSymmetry}; cf.\ Remark~\ref{RemCostSymmetry} above.

Let $\mathcal{T} \subset \mathbf{C}_{b}(\mathbb{R}^{d}\times \mathcal{R}_{2}\times \mathcal{W})$ be a countable and measure determining set of functions. %\citep[for instance, Theorem~II.6.6 in][p.\,47]{parthasarathy67}.
Let $(Q^{n})_{n\in\mathbb{I}}$ be a convergent subsequence with limit random variable $Q$ on $(\Omega,\mathcal{F},\Prb)$. Let $\Psi\in \mathcal{T}$, and set
\begin{align*}
	m_{\Psi}&\doteq \Mean_{\Prb}\left[ \Mean_{Q}\left[ \Psi\bigl(\hat{X}(0),\hat{\rho},\hat{W}\bigr) \right] \right], &&\\
	v_{\Psi}&\doteq \Mean_{\Prb}\left[ \left(\Mean_{Q}\left[ \Psi\bigl(\hat{X}(0),\hat{\rho},\hat{W}\bigr) \right] - m_{\Psi} \right)^{2}\right], &&\\
	m^{n}_{\Psi}&\doteq \Mean_{n}\left[ \Mean_{Q^{n}}\left[ \Psi\bigl(\hat{X}(0),\hat{\rho},\hat{W}\bigr) \right] \right], &n\in \mathbb{I}.&
\end{align*}
The mapping $\Theta \mapsto \int \Psi d\Theta$ is continuous on $\prbms[2]{\mathcal{Z}}$. By convergence of $(Q^{n})$ to $Q$ and the continuous mapping theorem,
\begin{align*}
	v_{\Psi} &= \lim_{n\to\infty} \Mean_{n}\left[ \left( \Mean_{Q^{n}}\left[ \Psi\bigl(\hat{X}(0),\hat{\rho},\hat{W}\bigr) \right] -  m^{n}_{\Psi} \right)^{2} \right] \\
	&= \lim_{n\to\infty} \Mean_{n}\left[\left( \frac{1}{n}\sum_{i=1}^{n} \Psi\bigl(\xi^{n}_{i},\rho^{n,i},W^{n}_{i}\bigr) - m^{n}_{\Psi} \right)^{2} \right],
\end{align*}
where $\rho^{n,i}$ is the relaxed control random variable induced by $u^{n}_{i}$. As a consequence of the i.i.d.\ hypothesis, the random variables $\Psi(\xi^{n}_{i},\rho^{n,i},W^{n}_{i})$, $i\in \{1,\ldots,n\}$, are independent and identically distributed with common mean equal to $m^{n}_{\Psi}$. Since $\Psi$ is bounded, it follows that $v_{\Psi} = 0$. This implies
\[
	\Mean_{Q}\left[ \Psi\bigl(\hat{X}(0),\hat{\rho},\hat{W}\bigr) \right] = m_{\Psi}\quad \Prb\text{-almost surely}.	
\]
Since $\mathcal{T}$ is countable, we have with $\Prb$-probability one
\[
	\Mean_{Q}\left[ \Psi\bigl(\hat{X}(0),\hat{\rho},\hat{W}\bigr) \right] = m_{\Psi}\quad \text{for all } \Psi\in \mathcal{T}.	
\]
Since $\mathcal{T}$ is also measure determining, it follows that there exists a measure $\nu\in \prbms{\mathbb{R}^{d}\times \mathcal{R}_{2}\times \mathcal{W}}$ such that, for $\Prb$-almost every $\omega\in \Omega$,
\[
	Q_{\omega}\circ \bigl(\hat{X}(0),\hat{\rho},\hat{W}\bigr)^{-1} = \nu.
\]
On the other hand, we know by Lemma~\ref{LemmaConvergence} that $Q_{\omega}\in \prbms[2]{\mathcal{Z}}$ is a McKean-Vlasov solution of Eq.~\eqref{EqLimitDynamicsRel} for $\Prb$-almost every $\omega\in \Omega$. Uniqueness of such solutions according to Lemma~\ref{LemmaMcKeanVlasovUnique} yields the existence of a measure $\Theta \in \prbms[2]{\mathcal{Z}}$ such that $Q_{\omega} = \Theta$ for $\Prb$-almost every $\omega\in \Omega$. Let $\mathfrak{m}\in \mathcal{M}_{2}$ be the flow of measures induced by $\Theta$. Then, for $\Prb$-almost every $\omega\in \Omega$,
\[
	Q_{\omega}\circ \hat{X}(t)^{-1} = \mathfrak{m}(t),\quad t\in [0,T].
\]
The assertion is now a consequence of Theorem~\ref{ThConnection}.
\end{proof}

Existence of local approximate Nash equilibria as required in Corollary~\ref{CorConnection} is guaranteed, in particular, under the hypotheses of Proposition~\ref{PropNashExistence} above (compact action space, bounded coefficients). Suppose that $(\boldsymbol{\xi}^{N})$ is such that, for each $N\in \mathbb{N}$, $\boldsymbol{\xi}^{N}$ is a vector of independent and identically distributed random variables with common marginal $\mathfrak{m}^{N}_{0}\in \prbms[2]{\mathbb{R}^{d}}$ and that, for some $\delta_{0} > 0$, $\sup_{N\in \mathbb{N}} \int |x|^{2+\delta_{0}} \mathfrak{m}^{N}_{0}(dx) < \infty$. Then, by Proposition~\ref{PropNashExistence}, there exists a corresponding sequence $(\boldsymbol{u}^{N})$ of local approximate Nash equilibria such that the hypotheses of Corollary~\ref{CorConnection} are satisfied. In addition to the desired limit relation, we thus obtain a proof of existence of solutions for the mean field game. Note that existence of solutions is just a by-product of our analysis; analogous existence results can in fact be obtained by directly working with the mean field game; see \citet{lacker15}. The proof there is based, as in Proposition~\ref{PropNashExistence} here, on relaxed controls and a version of Fan's fixed point theorem.

\begin{proof}[Proof of Theorem~\ref{ThConnection}]
By hypothesis, $Q\circ \hat{X}(\cdot)^{-1} = \mathfrak{m}(\cdot)$ $\Prb$-almost surely for some deterministic $\mathfrak{m} \in \mathcal{M}_{2}$. In view of Lemma~\ref{LemmaConvergence}, it is enough to show that the pair $(Q_{\omega},\mathfrak{m})$ satisfies the optimality condition of Definition~\ref{DefMFGSol} with $\Prb$-probability one. This is equivalent to showing that $\hat{J}(\mathfrak{m}(0),Q_{\omega};\mathfrak{m}) = \hat{V}(\mathfrak{m}(0);\mathfrak{m})$ for $\Prb$-almost all $\omega\in \Omega$.

Let $\epsilon > 0$. Choose a function $\psi^{\mathfrak{m}}_{\epsilon}: [0,T]\times \mathbb{R}^{d}\times \mathcal{W} \rightarrow \Gamma$ and a probability measure $\Theta^{\mathfrak{m}}_{\epsilon}\in \prbms[2]{\mathcal{Z}}$ according to Lemma~\ref{LemmaNoiseFeedback}. Choose a sequence of indices $(i_{\ast}^{n})_{n\in\mathbb{I}}$ according to condition~\eqref{CondCostSymmetry}. We will, as we may, assume that $i_{\ast}^{n} = 1$ for every $n\in \mathbb{I}$; otherwise renumber the components of the $n$-player games.

The proof proceeds in five steps. First, we construct a coupling for the initial conditions. In the second step, based on that coupling and the feedback function $\psi^{\mathfrak{m}}_{\epsilon}$, we define a competitor strategy $\boldsymbol{\tilde{u}}^{n}$ that differs from $\boldsymbol{u}^{n}$ only in component one ($= i_{\ast}^{n}$). As verified in step three, the associated normalized occupation measures have the same limit $Q$ as the sequence $(Q^{n})$. This is used in the fourth step to show that $\limsup_{n\to\infty} J^{n}_{1}(\boldsymbol{\tilde{u}}^{n}) \leq \hat{V}(\mathfrak{m}(0);\mathfrak{m}) + \epsilon$. Thanks to this upper limit, the local approximate Nash equilibrium property of $\boldsymbol{u}^{n}$ together with condition~\eqref{CondCostSymmetry}, and the asymptotic lower bound on the average costs from Lemma~\ref{LemmaConvergence}, we establish optimality in the fifth and last step.

\textbf{First step}. By hypothesis, $(\Prb_{n}\circ (Q^{n})^{-1})_{n\in \mathbb{I}}$ converges to $\Prb\circ Q^{-1}$ in $\prbms{\prbms[2]{\mathcal{Z}}}$. By the choice of the metric on $\mathcal{Z}$, the continuity of the map $\mathcal{Z}\ni (\phi,r,w)\mapsto \phi(0)\in \mathbb{R}^{d}$, and the mapping theorem \citep[for instance, Theorem~5.1 in][p.\,30]{billingsley68}, we have that
\[
	\prbms[2]{\mathcal{Z}}\ni \Theta \mapsto \Theta\circ (\hat{X}(0))^{-1} \in \prbms[2]{\mathbb{R}^{d}}
\]
is continuous. This implies, again by the continuous mapping theorem, that
\[
	\Prb_{n}\circ \left(Q^{n}\circ(\hat{X}(0))^{-1}\right)^{-1} \stackrel{n\to\infty}{\longrightarrow} \Prb\circ \left(Q\circ(\hat{X}(0))^{-1}\right)^{-1} \text{ in } \prbms{\prbms[2]{\mathbb{R}^{d}}}.
\]
By construction and hypothesis, respectively,
\begin{align*}
	& Q^{n}\circ(\hat{X}(0))^{-1} = \frac{1}{n}\sum_{i=1}^{n} \delta_{\xi^{n}_{i}}, &\text{while}& & \Prb\circ \left(Q\circ(\hat{X}(0))^{-1}\right)^{-1} = \delta_{\mathfrak{m}(0)}.&
\end{align*}
It follows that $(\frac{1}{n}\sum_{i=1}^{n} \delta_{\xi^{n}_{i}})_{n\in \mathbb{I}}$ converges to $\mathfrak{m}(0)$ in distribution as $\prbms[2]{\mathbb{R}^{d}}$-valued random variables, where $\mathfrak{m}(0)$ is deterministic. This convergence implies, in particular, that
\[
	\Mean_{n}\left[\frac{1}{n}\sum_{i=1}^{n} |\xi^{n}_{i}|^{2} \right] \stackrel{n\to\infty}{\longrightarrow} \int_{\mathbb{R}^{d}} |x|^{2}\, \mathfrak{m}(0)(dx).
\]   
By hypothesis, $\boldsymbol{\xi}^{n} = (\xi^{n}_{1},\ldots,\xi^{n}_{n})$ is exchangeable for every $n\in \mathbb{I}$. Convergence of the associated empirical measures, by the Tanaka-Sznitman theorem \citep[for instance, Theorem~3.2 in][p.\,27]{gottlieb98}, implies that
\[
	\Prb_{n}\circ (\xi^{n}_{1})^{-1} \stackrel{n\to\infty}{\longrightarrow} \mathfrak{m}(0) \text{ in } \prbms{\mathbb{R}^{d}}.
\]
Actually, we have convergence in $\prbms[2]{\mathbb{R}^{d}}$ since, by exchangeability,
\[
	\Mean_{n}\left[|\xi^{n}_{1}|^{2}\right] = \Mean_{n}\left[\frac{1}{n}\sum_{i=1}^{n} |\xi^{n}_{i}|^{2} \right] \text{ for every }n\in \mathbb{I},
\]
and the expectations on the right-hand side above converge to the second moment of $\mathfrak{m}(0)$. We are therefore in the situation of Lemma~\ref{LemmaCoupling}, and we apply that result with the choice $i^{n}_{\ast} = 1$ to obtain a sequence $(\bar{\xi}^{n})_{n\in\mathbb{I}}$ of $\mathbb{R}^{d}$-valued random variables such that $\bar{\xi}^{n}$ is $\sigma(\xi^{n}_{i^{n}_{\ast}},\vartheta^{n}_{i^{n}_{\ast}})$-measurable, $\Prb_{n}\circ (\bar{\xi}^{n})^{-1} = \mathfrak{m}(0)$ and $\Mean_{n}\left[|\xi^{n}_{1} - \bar{\xi}^{n}|^{2}\right] \to 0$ as $n\to \infty$.

\textbf{Second step}. Define a strategy vector $\boldsymbol{\tilde{u}}^{n} = (\tilde{u}^{n}_{1},\ldots,\tilde{u}^{n}_{n})$ by setting, for $(t,\omega)\in [0,T]\times \Omega_{n}$,
\[
	\tilde{u}^{n}_{i}(t,\omega)\doteq \begin{cases}
	\psi^{\mathfrak{m}}_{\epsilon}\left(t,\bar{\xi}^{n}(\omega),W^{n}_{1}(\cdot,\omega)\right) &\text{if } i=1,\\
	u^{n}_{i}(t,\omega) &\text{if } i\in \{2,\ldots,n\}.
	\end{cases}
\]
Notice that $\boldsymbol{\tilde{u}}^{n}$ is indeed a strategy vector for the game with $n$ players. Moreover, $\tilde{u}^{n}_{i} = u^{n}_{i}$ for $i\in \{2,\ldots,n\}$, while $\tilde{u}^{n}_{1}\in \mathcal{H}_{2}((\mathcal{F}^{n,1}_{t}),\Prb_{n};\Gamma)$. Let $\tilde{\rho}^{n,i}$ be the relaxed control induced by $\tilde{u}^{n}_{i}$, $i\in \{1,\ldots,n\}$. Clearly, $\tilde{\rho}^{n,i} = \rho^{n,i}$ for $i\geq 2$. On the other hand, by construction and since $\bar{\xi}^{n}$ and $W^{n}_{1}$ are independent,
\[
	\Prb_{n}\circ \left(\bar{\xi}^{n},\tilde{\rho}^{n,1},W^{n}_{1}\right)^{-1} = \Theta^{\mathfrak{m}}_{\epsilon}\circ \bigl(\hat{X}(0),\hat{\rho},\hat{W}\bigr)^{-1} \quad\text{for every }n\in \mathbb{I}.
\]
The law of $\tilde{u}^{n}_{1}$, in particular, does not change with $n$. It follows that
\[
	\sup_{n\in\mathbb{I}} \Mean_{n}\left[\int_{0}^{T} |\tilde{u}^{n}_{1}(t)|^{2}\, dt \right] < \infty.
\]
The coercivity assumption \hypref{HypCostCoercivity} implies that there exists $C > 0$ such that for every $n\in \mathbb{I}$,
\[
	\Mean_{n}\left[ \int_{0}^{T} |u^{n}_{1}(t)|^{2}\,dt \right] \leq C\left(1 + J_{1}^{n}(\boldsymbol{u}^{n}) \right).
\]
By choice of the index $i_{\ast}^{n} = 1$ according to \eqref{CondCostSymmetry}, we have $\sup_{n\in \mathbb{N}} J_{1}^{n}(\boldsymbol{u}^{n}) < \infty$. Since $\Mean_{n}\left[|\xi^{n}_{1}|^{2} \right] = \frac{1}{n}\sum_{i=1}^{n} \Mean_{n}\left[|\xi^{n}_{i}|^{2} \right]$ by exchangeability, it follows that
\begin{equation} \label{EqConnectionEnergy}
	\sup_{n\in \mathbb{I}} \Mean_{n}\left[|\xi^{n}_{1}|^{2} + \int_{0}^{T} \left( |u^{n}_{1}(t)|^{2} + |\tilde{u}^{n}_{1}(t)|^{2}\right) dt \right] < \infty.
\end{equation}

\textbf{Third step}. Let $(\tilde{X}^{n}_{1},\ldots,\tilde{X}^{n}_{n})$ be the solution of the system of equations \eqref{EqPrelimitDynamics} under strategy vector $\boldsymbol{\tilde{u}}^{n}$, and let $\tilde{\mu}^{N}$ denote the empirical measure process associated with $(\tilde{X}^{n}_{1},\ldots,\tilde{X}^{n}_{n})$. Let $\tilde{Q}^{n}$ be the normalized occupation measure associated with $\boldsymbol{\tilde{u}}^{n}$, i.e., the $\prbms[2]{\mathcal{Z}}$-valued random variable determined by
\[
	\tilde{Q}^{n}_{\omega}(B\times R\times D) \doteq \frac{1}{n} \sum_{i=1}^{n} \delta_{\tilde{X}^{n}_{i}(\cdot,\omega)}(B)\cdot \delta_{\tilde{\rho}^{n,i}_{\omega}}(R)\cdot \delta_{W^{n}_{i}(\cdot,\omega)}(D),\quad \omega\in \Omega_{n},
\]
$B\in \Borel{\mathcal{X}}$, $R\in \Borel{\mathcal{R}_{2}}$, $D\in \Borel{\mathcal{W}}$. We are going to show that
\begin{equation} \label{EqConnectionOMConv}
	\tilde{Q}^{n} \stackrel{n\to\infty}{\longrightarrow} Q \text{ in distribution as $\prbms[2]{\mathcal{Z}}$-valued random variables.}
\end{equation}
Since $Q^{n}\to Q$ in distribution, it suffices to show that
\[
	\mathrm{d}_{\prbms{\prbms[2]{\mathcal{Z}}}}\left(\Prb_{n}\circ(\tilde{Q}^{n})^{-1},\Prb_{n}\circ (Q^{n})^{-1}\right) \stackrel{n\to\infty}{\longrightarrow} 0.
\]
Let $n\in \mathbb{I}$. By construction, definition of the bounded Lipschitz metric, inequality \eqref{EqEmpWasserstein}, and H{\"o}lder's inequality,
\begin{align*}
	&\mathrm{d}_{\prbms{\prbms[2]{\mathcal{Z}}}}\left(\Prb_{n}\circ(\tilde{Q}^{n})^{-1},\Prb_{n}\circ (Q^{n})^{-1}\right) \\
	&= \sup_{G \in \mathbf{C}(\prbms[2]{\mathcal{Z}}):\, \|G\|_{\mathrm{bLip}} \leq 1 } \Mean_{n}\left[ G\bigl(Q^{n}\bigr) - G\bigl(\tilde{Q}^{n}\bigr) \right] \\
	&\leq \Mean_{n}\left[ \mathrm{d}_{\prbms[2]{\mathcal{Z}}}\bigl(Q^{n},\tilde{Q}^{n}\bigr) \right] \\
	&\leq \sqrt{ \Mean_{n}\left[\frac{1}{n}\sum_{i=1}^{n} \mathrm{d}_{\mathcal{Z}}\left(\bigl(X^{n}_{i},\rho^{n,i},W^{n}_{i}\bigr), \bigl(\tilde{X}^{n}_{i},\tilde{\rho}^{n,i},W^{n}_{i}\bigr)\right)^{2} \right]} \\
	&\leq \frac{1}{\sqrt{n}} + \sqrt{ \Mean_{n}\left[\frac{1}{n}\sum_{i=1}^{n} \sup_{t\in [0,T]} \bigl|X^{n}_{i}(t) - \tilde{X}^{n}_{i}(t)\bigr|^{2} \right]},
\end{align*}
where the last inequality follows by definition of $\mathrm{d}_{\mathcal{Z}}$ and from the fact that $\rho^{n,i} = \tilde{\rho}^{n,i}$ for $i\in \{2,\ldots,n\}$. Using assumption~\hypref{HypLipschitz}, H{\"o}lder's inequality, Doob's maximal inequality, It{\^{o}}'s isometry, inequality \eqref{EqEmpWasserstein}, and Fubini's theorem, we find that for $i\in \{2,\ldots,n\}$, every $t\in [0,T]$,
\begin{align*}
	& \Mean_{n}\left[ \sup_{s\in [0,t]} \bigl|X^{n}_{i}(s) - \tilde{X}^{n}_{i}(s)\bigr|^{2} \right] \\
	&\leq 4(T+4)L^{2} \Mean_{n}\left[ \int_{0}^{t} \bigl|X^{n}_{i}(s) - \tilde{X}^{n}_{i}(s)\bigr|^{2}\,ds + \int_{0}^{t} \mathrm{d}_{2}\left(\mu^{N}(s),\tilde{\mu}^{N}(s)\right)^{2} ds \right] \\
	&\leq 4(T+4)L^{2} \int_{0}^{t} \Mean_{n}\left[ \bigl|X^{n}_{i}(s) - \tilde{X}^{n}_{i}(s)\bigr|^{2} + \frac{1}{n}\sum_{k=1}^{n} \bigl|X^{n}_{k}(s) - \tilde{X}^{n}_{k}(s)\bigr|^{2} \right] ds.
\end{align*}
Similarly, but also using assumption~\hypref{HypGrowth},
\[
\begin{split}
	&\Mean_{n}\left[ \sup_{s\in [0,t]} \bigl|X^{n}_{1}(s) - \tilde{X}^{n}_{1}(s)\bigr|^{2} \right] \leq C_{n} \\
	&\qquad + 8(T+2)L^{2} \int_{0}^{t} \Mean_{n}\left[ \bigl|X^{n}_{1}(s) - \tilde{X}^{n}_{1}(s)\bigr|^{2} + \frac{1}{n}\sum_{k=1}^{n} \bigl|X^{n}_{k}(s) - \tilde{X}^{n}_{k}(s)\bigr|^{2} \right]ds,
\end{split}
\]
where $C_{n}$ is equal to
\[
	80 T K^{2} \int_{0}^{T} \Mean_{n}\left[ 1 + |X^{n}_{1}(s)|^{2} + |u^{n}_{1}(s)|^{2} + |\tilde{u}^{n}_{1}(s)|^{2} + \frac{1}{n}\sum_{k=1}^{n} |X^{n}_{k}(s)|^{2} \right] ds.
\]
It follows that, for every $t\in [0,T]$,
\begin{multline*}
	\frac{1}{n}\sum_{i=1}^{n} \Mean_{n} \left[ \sup_{s\in [0,t]} \bigl|X^{n}_{i}(s) - \tilde{X}^{n}_{i}(s)\bigr|^{2} \right] \\
	\leq \frac{C_{n}}{n} + 8(T+4)L^{2} \int_{0}^{t} \Mean_{n}\left[\frac{1}{n}\sum_{i=1}^{n} \sup_{\tilde{s}\in [0,s]}\bigl|X^{n}_{i}(\tilde{s}) - \tilde{X}^{n}_{i}(\tilde{s})\bigr|^{2} \right]ds.
\end{multline*}
Therefore, by Gronwall's lemma,
\[
	\Mean_{n} \left[ \frac{1}{n}\sum_{i=1}^{n} \sup_{t\in [0,T]} \bigl|X^{n}_{i}(t) - \tilde{X}^{n}_{i}(t)\bigr|^{2} \right] \leq \frac{C_{n}}{n} \exp\left(8T(T+4)L^{2}\right).
\]
To complete the proof of \eqref{EqConnectionOMConv}, one checks that $\sup_{n\in \mathbb{I}} C_{n} < \infty$. But this is a consequence of \eqref{EqConnectionEnergy}, condition~\eqref{CondTightness}, and Lemma~\ref{LemmaGrowthBounds}. 

\textbf{Fourth step}. We are going to show that
\begin{equation} \label{EqConnectionLimsup}
	\limsup_{n\to\infty} J^{n}_{1}(\boldsymbol{\tilde{u}}^{n}) \leq \hat{J}\left(\mathfrak{m}(0),\Theta^{\mathfrak{m}}_{\epsilon};\mathfrak{m}\right).
\end{equation}
Let $n\in \mathbb{I}$. Recall that $\tilde{X}^{n}_{1}$ solves the equation
\[
\begin{split}
	\tilde{X}^{n}_{1}(t) &= \xi^{n}_{1} + \int_{0}^{t} b\left(s,\tilde{X}^{n}_{1}(s),\tilde{\mu}^{n}(s),\tilde{u}^{n}_{1}(s)\right)ds \\
	&\quad + \int_{0}^{t} \sigma\left(s,\tilde{X}^{n}_{1}(s),\tilde{\mu}^{n}(s)\right)dW^{n}_{1}(s),\quad t\in [0,T].
\end{split}
\]
Let $\bar{X}^{n}_{1}$ be the unique solution to
\[
\begin{split}
	\bar{X}^{n}_{1}(t) &= \bar{\xi}^{n} + \int_{0}^{t} b\left(s,\bar{X}^{n}_{1}(s),\mathfrak{m}(s),\tilde{u}^{n}_{1}(s)\right)ds \\
	&\quad + \int_{0}^{t} \sigma\left(s,\bar{X}^{n}_{1}(s),\mathfrak{m}(s)\right)dW^{n}_{1}(s),\quad t\in [0,T].
\end{split}
\]
Then, by uniqueness in law and construction, for every $n\in \mathbb{I}$,
\begin{multline*}
	\hat{J}\left(\mathfrak{m}(0),\Theta^{\mathfrak{m}}_{\epsilon};\mathfrak{m}\right) \\
	= \Mean_{n}\left[ \int_{0}^{T} f\bigl(t,\bar{X}^{n}_{1}(t),\mathfrak{m}(t),\tilde{u}^{n}_{1}(t)\bigr)dt + F\bigl(\bar{X}^{n}_{1}(T),\mathfrak{m}(T)\bigr)\right].
\end{multline*}
Using assumption~\hypref{HypLipschitz}, H{\"o}lder's inequality, It{\^{o}}'s isometry, and Fubini's theorem, we find that for every $t\in [0,T]$,
\begin{align*}
	&\Mean_{n}\left[ \left| \tilde{X}^{n}_{1}(t) - \bar{X}^{n}_{1}(t) \right|^{2} \right] \\
	\begin{split}
	&\leq 3\Mean_{n}\left[ \left|\xi^{n}_{1} - \bar{\xi}^{n} \right|^{2} \right] + 6(T+1)L^{2} \Mean_{n}\left[ \int_{0}^{T} \mathrm{d}_{2}\left(\tilde{\mu}^{n}(s), \mathfrak{m}(s)\right)^{2}\,ds \right] \\
	&\quad + 6(T+1)L^{2} \int_{0}^{t} \Mean_{n}\left[ \left| \tilde{X}^{n}_{1}(s) - \bar{X}^{n}_{1}(s) \right|^{2} \right] ds.
\end{split}
\end{align*}
The limit relation \eqref{EqConnectionOMConv} implies that $(\tilde{\mu}^{n}(0))_{n\in\mathbb{I}}$ converges to $\mathfrak{m}(0)$ in distribution as $\prbms[2]{\mathbb{R}^{d}}$-valued random variables and that
\[
	\sup_{t\in [0,T]} \Mean_{n}\left[ \mathrm{d}_{2}\left(\tilde{\mu}^{n}(t), \mathfrak{m}(t)\right)^{2} \right] \stackrel{n\to\infty}{\longrightarrow} 0.
\]
By choice of the random variables $\bar{\xi}^{n}$ according to Lemma~\ref{LemmaCoupling},
\[
	\Mean_{n}\left[ \left|\xi^{n}_{1} - \bar{\xi}^{n} \right|^{2} \right] \stackrel{n\to\infty}{\longrightarrow} 0.
\]
Therefore, by Gronwall's lemma,
\[
	\sup_{t\in [0,T]} \Mean_{n}\left[ \left| \tilde{X}^{n}_{1}(t) - \bar{X}^{n}_{1}(t) \right|^{2} \right] \stackrel{n\to\infty}{\longrightarrow} 0.
\]
Thanks to assumption~\hypref{HypCostLip} and H{\"o}lder's inequality,
\begin{align*}
	& \left| J^{n}_{1}(\boldsymbol{\tilde{u}}^{n}) - \hat{J}\left(\mathfrak{m}(0),\Theta^{\mathfrak{m}}_{\epsilon};\mathfrak{m}\right) \right| \\
\begin{split}
	&\leq \Mean_{n}\left[ \int_{0}^{T} \left| f\bigl(t,\tilde{X}^{n}_{1}(t),\tilde{\mu}^{n}(t),\tilde{u}^{n}_{1}(t)\bigr) - f\bigl(t,\bar{X}^{n}_{1}(t),\mathfrak{m}(t),\tilde{u}^{n}_{1}(t)\bigr)\right| dt \right] \\
	&\quad + \Mean_{n}\left[ \left| F\bigl(\tilde{X}^{n}_{1}(T),\tilde{\mu}^{n}(T)\bigr) - F\bigl(\bar{X}^{n}_{1}(T),\mathfrak{m}(T)\bigr) \right| \right],
\end{split} \\
%\begin{split}
%	&\leq \sqrt{10} L \Mean_{n}\left[\int_{0}^{T} \left( |\tilde{X}^{n}_{1}(t) - \bar{X}^{n}_{1}(t)|^{2} + \mathrm{d}_{2}\left(\tilde{\mu}^{n}(t), \mathfrak{m}(t)\right)^{2} \right)dt\right]^{1/2} \\
%	&\quad \cdot \Mean_{n}\left[\int_{0}^{T}\left(1 + |\tilde{X}^{n}_{1}(t)|^{2} + |\bar{X}^{n}_{1}(t)|^{2} + \mathrm{d}_{2}\left(\tilde{\mu}^{n}(t), \delta_{0}\right)^{5} + \mathrm{d}_{2}\left(\mathfrak{m}(t),\delta_{0}\right)^2 \right)dt \right]^{1/2} \\
%	&\quad + \sqrt{10} L \Mean_{n}\left[|\tilde{X}^{n}_{1}(T) - \bar{X}^{n}_{1}(T)|^{2} + \mathrm{d}_{2}\left(\tilde{\mu}^{n}(T), \mathfrak{m}(T)\right)^{2}\right]^{1/2} \\
%	&\quad \cdot \Mean_{n}\left[1 + |\tilde{X}^{n}_{1}(T)|^{2} + |\bar{X}^{n}_{1}(T)|^{2} + \mathrm{d}_{2}\left(\tilde{\mu}^{n}(T), \delta_{0}\right)^{2} + \mathrm{d}_{2}\left(\mathfrak{m}(T),\delta_{0}\right)^2 \right]^{1/2}
%\end{split} \\
\begin{split}
	&\leq \sqrt{10} L(1+\sqrt{T}) \sup_{t\in [0,T]}\Mean_{n}\left[|\tilde{X}^{n}_{1}(t) - \bar{X}^{n}_{1}(t)|^{2} + \mathrm{d}_{2}\left(\tilde{\mu}^{n}(t), \mathfrak{m}(t)\right)^{2} \right]^{1/2} \\
	&\quad \cdot \sup_{t\in [0,T]}\Mean_{n}\left[1 + |\tilde{X}^{n}_{1}(t)|^{2} + |\bar{X}^{n}_{1}(t)|^{2} + \mathrm{d}_{2}\left(\tilde{\mu}^{n}(t), \delta_{0}\right)^{2} + \mathrm{d}_{2}\left(\mathfrak{m}(t),\delta_{0}\right)^2 \right]^{1/2}.
\end{split}
\end{align*}
By \eqref{EqConnectionEnergy} together with Lemma~\ref{LemmaGrowthBounds} and an analogous estimate applied to $\bar{X}^{n}_{1}$, and since $\sup_{t\in [0,T]}\mathrm{d}_{2}\left(\mathfrak{m}(t),\delta_{0}\right)^2 < \infty$ by continuity, we have
\[
	\sup_{n\in \mathbb{I}} \sup_{t\in [0,T]}\Mean_{n}\left[|\tilde{X}^{n}_{1}(t)|^{2} + |\bar{X}^{n}_{1}(t)|^{2} + \mathrm{d}_{2}\left(\tilde{\mu}^{n}(t), \delta_{0}\right)^{2} + \mathrm{d}_{2}\left(\mathfrak{m}(t),\delta_{0}\right)^2 \right] < \infty.
\]
It follows that $J^{n}_{1}(\boldsymbol{\tilde{u}}^{n}) \to \hat{J}\left(\mathfrak{m}(0),\Theta^{\mathfrak{m}}_{\epsilon};\mathfrak{m}\right)$ as $n\to\infty$, which establishes \eqref{EqConnectionLimsup}.

\textbf{Fifth step}. The limit relation \eqref{EqConnectionLimsup} and the choice of $\Theta^{\mathfrak{m}}_{\epsilon}$ imply that
\[
	\limsup_{j\to\infty} J^{N_{j}}_{1}\left(\boldsymbol{\tilde{u}}^{N_{j}}\right) \leq \hat{V}(\mathfrak{m}(0);\mathfrak{m}) + \epsilon.
\]
By hypothesis, $\boldsymbol{u}^{n}$ is a local $\epsilon_{n}$-Nash equilibrium. By construction, $\boldsymbol{\tilde{u}}^{n}$ differs from $\boldsymbol{u}^{n}$ only in component number one ($= i^{n}_{\ast}$), and $\tilde{u}^{n}_{1}$ is $(\mathcal{F}^{n,1}_{t})$-adapted. Therefore,
\[
	J^{n}_{1}\left(\boldsymbol{u}^{n}\right) \leq J^{n}_{1}\left(\boldsymbol{\tilde{u}}^{n}\right) + \epsilon_{n}.
\]
By choice of the index $1 = i^{n}_{\ast}$ according to \eqref{CondCostSymmetry} and since $\epsilon_{n}\to 0$ by hypothesis,
\[
	\limsup_{n\to\infty} \frac{1}{n} \sum_{i=1}^{n} J^{n}_{i}\left(\boldsymbol{u}^{n}\right) \leq \limsup_{n\to\infty} J^{n}_{1}\left(\boldsymbol{u}^{n}\right) \leq \limsup_{n\to\infty} J^{n}_{1}\left(\boldsymbol{\tilde{u}}^{n}\right).
\]
It follows that
\[
	\limsup_{n\to\infty} \frac{1}{n} \sum_{i=1}^{n} J^{n}_{i}\left(\boldsymbol{u}^{n}\right) \leq \hat{V}(\mathfrak{m}(0);\mathfrak{m}) + \epsilon.
\]
On the other hand, thanks to the second part of Lemma~\ref{LemmaConvergence}, 
\[
	\liminf_{n\to\infty} \frac{1}{n} \sum_{i=1}^{n} J^{n}_{i}\left(\boldsymbol{u}^{n}\right) \geq \int_{\Omega} \hat{J}\bigl(\mathfrak{m}(0),Q_{\omega},\mathfrak{m}\bigr) \Prb(d\omega).
\]
It follows that 
\[
	\int_{\Omega} \hat{J}\bigl(\mathfrak{m}(0),Q_{\omega},\mathfrak{m}\bigr) \Prb(d\omega) \leq \hat{V}(\mathfrak{m}(0);\mathfrak{m}) + \epsilon.
\]
Since $\epsilon > 0$ was arbitrary and $\hat{J}(\mathfrak{m}(0),Q_{\omega},\mathfrak{m}) \geq \hat{V}(\mathfrak{m}(0);\mathfrak{m})$ for every $\omega\in \Omega$ by definition of $\hat{V}$, we conclude that
\[
	\hat{J}\bigl(\mathfrak{m}(0),Q_{\omega},\mathfrak{m}\bigr) = \hat{V}(\mathfrak{m}(0);\mathfrak{m})\quad \text{for $\Prb$-almost all }\omega\in \Omega.
\]
\end{proof}

\begin{rem} \label{RemRandomFlow}
	The proof of Theorem~\ref{ThConnection} gives some insight into why the assumption that the limit flow of measures $\mathfrak{m}$ is deterministic cannot simply be dropped. In the second step of the proof, we define a competitor strategy $\tilde{u}^{n}_{1}$ for the deviating player (player one after relabeling) in terms of the noise feedback function $\psi^{\mathfrak{m}}_{\epsilon}$. In general, for any $t\in [0,T]$, $\psi^{\mathfrak{m}}_{\epsilon}(t,\cdot,\cdot)$ depends on $\mathfrak{m}$ through its values for all times, not only through its values up to time $t$. Therefore, if $\mathfrak{m}$ were random, even taking for granted the measurable dependence of $\psi^{\mathfrak{m}}_{\epsilon}$ on $\mathfrak{m}$, we might end up with a non-adapted competitor strategy. Indeed, the natural choice for $\tilde{u}^{n}_{1}$, namely $\tilde{u}^{n}_{1}(t,\omega)\doteq \psi^{\mu^{n}_{\omega}(\cdot)}_{\epsilon}\left(t,\bar{\xi}^{n}(\omega),W^{n}_{1}(\cdot,\omega)\right)$, would in general yield a $\Gamma$-valued process that would not be an admissible strategy for player one in the $n$-player game.  
\end{rem}

%-------

\begin{appendix}

\section*{Appendix}

\section{Proof of Lemma~\ref{LemmaMPCharacterization}, second part}
\label{AppMPCharacterization}

Let $\Theta\in \prbms{\mathcal{Z}}$ be a solution of Eq.~\eqref{EqLimitDynamicsRel} with flow of measures $\mathfrak{m}$ in the sense of Definition~\ref{DefSolution}. Using the local martingale property of $M^{\mathfrak{m}}_{f}$ for $f$ a monomial of first or second order as in the proof of Proposition~5.4.6 in \citet[pp.\,315-316]{karatzasshreve91}, we find that, under $\Theta$ and with respect to the filtration $(\mathcal{G}_{t})$:
\begin{itemize}
	\item $\hat{W}$ is a $d_{1}$-dimensional vector of continuous local martingales with $\hat{W}(0) = 0$ and quadratic covariations
	\[
		\bigl\langle \hat{W}_{l}, \hat{W}_{\tilde{l}} \bigr\rangle(t) = t\cdot \delta_{l,\tilde{l}},\quad l,\tilde{l}\in \{1,\ldots,d_{1}\};
	\] 
	
	\item $\bar{X}\doteq \hat{X} - \hat{X}(0) - \int_{\Gamma\times[0,\cdot]} b\bigl(s,\hat{X}(s),\mathfrak{m}(s),\gamma\bigr)\hat{\rho}(d\gamma,ds)$ is a $d$-dimensional vector of continuous local martingales with quadratic covariations
	\[
		\bigl\langle \bar{X}_{j}, \bar{X}_{k} \bigr\rangle(t) = \int_{0}^{t} (\sigma\trans{\sigma})_{jk}\bigl(s,\hat{X}(s),\mathfrak{m}(s)\bigr)ds,\quad j,k\in \{1,\ldots,d\};
	\]
	
	\item $\hat{W}$, $\bar{X}$ have quadratic covariations
	\[
		\bigl\langle \bar{X}_{k}, \hat{W}_{l} \bigr\rangle(t) = \int_{0}^{t} \sigma_{kl}\bigl(s,\hat{X}(s),\mathfrak{m}(s)\bigr)ds,
	\]
	where $k\in \{1,\ldots,d\}$, $l\in \{1,\ldots,d_{1}\}$.
\end{itemize}
The local martingale property also holds with respect to the filtration $(\mathcal{G}^{\Theta}_{t+})$; see the solution to Problem~5.4.13 in \citet[pp.\,318-319,\,392]{karatzasshreve91} and Remark~4.2 in \citet{budhirajaetal12}. By L{\'e}vy's characterization of Brownian motion \citep[for instance, Theorem~3.3.16 in][p.\,157]{karatzasshreve91}, we see that $\hat{W}$ is a standard Wiener process with respect to $(\mathcal{G}^{\Theta}_{t+})$. As a consequence, the process
\[
	Y(t)\doteq \int_{0}^{t} \sigma(s,\hat{X}(s),\mathfrak{m}(s)\bigr)d\hat{W}(s),\quad t\in [0,T],
\]
is well defined and a $d$-dimensional vector of continuous local martingales
(under $\Theta$ with respect to $(\mathcal{G}^{\Theta}_{t+})$) with quadratic covariations
\begin{align*}
	\bigl\langle Y_{j}, Y_{k} \bigr\rangle(t) &= \int_{0}^{t} (\sigma\trans{\sigma})_{jk}\bigl(s,\hat{X}(s),\mathfrak{m}(s)\bigr)ds, &j,k\in \{1,\ldots,d\},& \\
	\bigl\langle Y_{j}, \hat{W}_{l} \bigr\rangle(t) &= \int_{0}^{t}\sigma_{jl}\bigl(s,\hat{X}(s),\mathfrak{m}(s)\bigr)ds, &j\in \{1,\ldots,d\},\; l\in \{1,\ldots,d_{1}\}.&
\end{align*}
The quadratic covariations between the components of the vectors of continuous local martingales $\bar{X}$, $Y$ are given by \citep[cf.\ Proposition~3.2.24 in][p.\,147]{karatzasshreve91} 
\begin{align*}
	\bigl\langle Y_{j}, \bar{X}_{k}\bigr\rangle(t) &= \sum_{l=1}^{d_{1}} \int_{0}^{t} \sigma_{jl}\bigl(s,\hat{X}(s),\mathfrak{m}(s)\bigr) d\bigl\langle \bar{X}_{k}, \hat{W}_{l} \bigr\rangle(s) \\
%	&= \sum_{l=1}^{d_{1}} \int_{0}^{t} \sigma_{jl}\bigl(s,\hat{X}(s),\mathfrak{m}(s)\bigr)\cdot \sigma_{kl}\bigl(s,\hat{X}(s),\mathfrak{m}(s)\bigr) ds,\\
	&= \int_{0}^{t} (\sigma\trans{\sigma})_{jk}\bigl(s,\hat{X}(s),\mathfrak{m}(s)\bigr)ds, \quad j,k\in \{1,\ldots,d\}.
\end{align*}
It follows that $\bar{X} - Y$ is a $d$-dimensional vector of continuous local martingales with $\bar{X}(0) = 0 = Y(0)$ and quadratic covariations
\[
	\bigl\langle \bar{X}_{j} - Y_{j}, \bar{X}_{k} - Y_{k} \bigr\rangle = \bigl\langle \bar{X}_{j}, \bar{X}_{k}\bigr\rangle - \bigl\langle Y_{j}, \bar{X}_{k}\bigr\rangle - \bigl\langle \bar{X}_{j}, Y_{k} \bigr\rangle + \bigl\langle Y_{j}, Y_{k} \bigr\rangle \equiv 0.
\]
This implies \citep[cf.\ Problem~1.5.12 in][p.\,35]{karatzasshreve91} that $\bar{X} = Y$ $\Theta$-almost surely, which establishes the solution property.

%-----

\section{Tightness functions}
\label{AppTightnessFncts}

Let $\mathcal{S}$ be a Polish space. A function $g\!: \mathcal{S} \rightarrow [0,\infty]$ is called a \emph{tightness function} on $\mathcal{S}$ if it is measurable and its sublevel sets $\{s\in \mathcal{S}: g(s) \leq c\}$ are pre-compact in $\mathcal{S}$ for all $c\in [0,\infty)$. If $g$ is a tightness function on $\mathcal{S}$, then the function $\prbms{\mathcal{S}} \ni \Theta \mapsto \int_{\mathcal{S}} g(s)\,\Theta(ds) \in [0,\infty]$ is a tightness function on $\prbms{\mathcal{S}}$; see, for instance, Theorem~A.3.17 in \citet[p.\,309]{dupuisellis97}.

\subsection{A tightness function on $\mathcal{R}_{2}$}
\label{AppTightnessFnctControls}

Let $\delta_{0} > 0$. Define a function $\tilde{g}\!: \mathcal{R}_{2}\rightarrow [0,\infty]$ by
\[
	\tilde{g}(r)\doteq \int_{\Gamma\times [0,T]} |\gamma|^{2+\delta_{0}} r(d\gamma,dt).
\]
We check that $\tilde{g}$ is a tightness function on $\mathcal{R}_{2}$. By construction, $\tilde{g}$ is measurable. For $c\in [0,\infty)$, set
\[
	A_{c}\doteq \left\{ r\in \mathcal{R}_{2} : \tilde{g}(r)\leq c \right\}.
\]
Fix $c\in [0,\infty)$. Then we have to show that $A_{c}$ is pre-compact in $\mathcal{R}_{2}$. This is equivalent to showing that
\begin{enumerate}[a)]
	\item $A_{c}$ is pre-compact in $\mathcal{R}$,
	\item if $(r_{n})_{n\in \mathbb{N}} \subset A_{c}$ is such that $r_{n} \to r$ in $\mathcal{R}$ for some $r\in \mathcal{R}$, then $r\in \mathcal{R}_{2}$ and $\int_{\Gamma\times [0,T]} |\gamma|^{2} r_{n}(d\gamma,dt) \to \int_{\Gamma\times [0,T]} |\gamma|^{2} r(d\gamma,dt)$ as $n\to\infty$.
\end{enumerate}
Pre-compactness of $A_{c}$ in $\mathcal{R}$ is equivalent to tightness of $A_{c}$. This holds since, for every $M > 0$, the set $\{\gamma\in \Gamma : |\gamma| \leq M\}$ is compact (by assumption~\hypref{HypCostCoercivity}, $\Gamma$ is closed) and, by Markov's inequality,
\[
	\sup_{r\in A_{c}} r\left\{ (\gamma,t)\in \Gamma\times [0,T] : |\gamma| > M \right\} \leq \frac{1}{M^{2+\delta_{0}}}\cdot \sup_{r\in A_{c}} \tilde{g}(r) \leq \frac{c}{M^{2+\delta_{0}}},
\] 
which tends to zero as $M \to\infty$.

As to the convergence of moments, let $(r_{n})_{n\in \mathbb{N}} \subset A_{c}$ be such that $r_{n} \to r$ in $\mathcal{R}$ for some $r\in \mathcal{R}$. Then, by Fatou's lemma and H{\"o}lder's inequality,
\[
	\liminf_{n\to\infty} \int_{\Gamma\times [0,T]} |\gamma|^{2} r_{n}(d\gamma,dt) \geq \int_{\Gamma\times [0,T]} |\gamma|^{2} r(d\gamma,dt),
\]
hence $r\in A_{c}\subset \mathcal{R}_{2}$. By convergence in $\mathcal{R}$, we have, for every $M > 0$,
\[
	\lim_{n\to\infty} \int_{\Gamma\times [0,T]} |\gamma|^{2}\wedge M\; r_{n}(d\gamma,dt) = \int_{\Gamma\times [0,T]} |\gamma|^{2}\wedge M\; r(d\gamma,dt).
\]
On the other hand, again by H{\"o}lder's and Markov's inequality, for every $n\in \mathbb{N}$, every $M > 0$,
\begin{align*}
	& \int_{\Gamma\times [0,T]} |\gamma|^{2}\cdot \mathbf{1}_{[M,\infty)}\left(|\gamma|^{2}\right) r_{n}(d\gamma,dt) \\
	&\leq \left(\int_{\Gamma\times [0,T]} |\gamma|^{2+\delta_{0}} r_{n}(d\gamma,dt)\right)^{\frac{1}{2+\delta_{0}}} \cdot r_{n}\left\{ (\gamma,t)\in \Gamma\times [0,T] : |\gamma|^{2} > M \right\}^{\frac{1+\delta_{0}}{2+\delta_{0}}} \\
	&\leq c^{\frac{1}{2+\delta_{0}}} \cdot c^{\frac{1+\delta_{0}}{2+\delta_{0}}}\cdot M^{-(1+\delta_{0}/2)}.
\end{align*}
It follows that
\[
	\sup_{n\in \mathbb{N}} \int_{\Gamma\times [0,T]} |\gamma|^{2}\cdot \mathbf{1}_{[M,\infty)}\left(|\gamma|^{2}\right) r_{n}(d\gamma,dt) \stackrel{M\to\infty}{\longrightarrow} 0,
\]
hence $\lim_{n\to\infty} \int_{\Gamma\times [0,T]} |\gamma|^{2} r_{n}(d\gamma,dt) = \int_{\Gamma\times [0,T]} |\gamma|^{2} r(d\gamma,dt)$.

\subsection{A tightness function on $\prbms[2]{\mathcal{Z}$}}
\label{AppTightnessFnctMeasures}

We check that the function $g$ defined by \eqref{ExTFunctMeasures} is a tightness function on $\prbms[2]{\mathcal{Z}}$. By construction, $g$ is measurable (by continuity, the suprema appearing inside the second integral and in the definition of the modulus of continuity can be restricted to countable index sets). Thus we have to show that, given any $c\in [0,\infty)$, the set
\[
	A(c)\doteq \left\{ \Theta\in \prbms[2]{\mathcal{Z}}: g(\Theta) \leq c \right\}
\]
is pre-compact in $\prbms[2]{\mathcal{Z}}$. Fix $c\in [0,\infty)$. The pre-compactness of $A(c)$ in $\prbms[2]{\mathcal{Z}}$ is equivalent to the following two conditions:
\begin{enumerate}[a)]
	\item \label{FnctGTightness} $A(c)$ is tight in $\prbms{\mathcal{Z}}$;
	\item \label{FnctGMoments} if $(\Theta^{n})_{n\in\mathbb{N}} \subset A(c)$ is such that $\Theta^{n}$ converges to $\bar{\Theta}$ in $\prbms{\mathcal{Z}}$ for some $\bar{\Theta}\in \prbms{\mathcal{Z}}$, then $\bar{\Theta}\in \prbms[2]{\mathcal{Z}}$ and $\int_{\mathcal{Z}} \mathrm{d}_{\mathcal{Z}}(s,s_{0})^{2}\, \Theta^{n}(ds) \to \int_{\mathcal{Z}} \mathrm{d}_{\mathcal{Z}}(s,s_{0})^{2}\,\bar{\Theta}(ds)$, where $s_{0}$ is some arbitrarily fixed element of $\mathcal{Z}$. 
\end{enumerate}

To verify \ref{FnctGTightness}), it is enough to check tightness of marginals, that is, to verify that $A_{\mathcal{X}}(c)\doteq \{ [\Theta]_{\mathcal{X}} : \Theta\in A_{c}\}$ is tight in $\prbms{\mathcal{X}}$, $A_{\mathcal{R}_{2}}(c)\doteq \{ [\Theta]_{\mathcal{R}_{2}} : \Theta\in A_{c}\}$ is tight in $\prbms{\mathcal{R}_{2}}$, and $A_{\mathcal{W}}(c)\doteq \{ [\Theta]_{\mathcal{W}} : \Theta\in A_{c}\}$ is tight in $\prbms{\mathcal{W}}$, where $[\Theta]_{\mathcal{X}}$, $[\Theta]_{\mathcal{R}_{1}}$, $[\Theta]_{\mathcal{W}}$ denote the marginal distributions of $\Theta$ on $\mathcal{X}$, $\mathcal{R}_{2}$, and $\mathcal{W}$, respectively. Thanks to Markov's inequality and the Ascoli-Arzel{\`a} criterion \citep[for instance, Theorem~8.2 in][p.\,55]{billingsley68}, $A_{\mathcal{X}}(c)$, $A_{\mathcal{W}}(c)$ are tight in $\prbms{\mathcal{X}}$ and $\prbms{\mathcal{W}}$, respectively. The tightness of $A_{\mathcal{R}_{2}}(c)$ in $\prbms{\mathcal{R}_{2}}$ follows from the fact that the mapping
\[
	 \mathcal{R}_{2} \ni r\mapsto \int_{\Gamma\times [0,T]} |\gamma|^{2+\delta_{0}}\, r(d\gamma,dt) \in [0,\infty]
\]
is a tightness function on $\mathcal{R}_{2}$; see Appendix~\ref{AppTightnessFnctControls}.

In order to check \ref{FnctGMoments}), let $(\Theta^{n})_{n\in\mathbb{N}} \subset A(c)$ be such that $\Theta^{n}$ converges to $\bar{\Theta}$ in $\prbms{\mathcal{Z}}$ for some $\bar{\Theta}\in \prbms{\mathcal{Z}}$. By a version of Fatou's lemma \citep[cf.\ Theorem~A.3.12][p.\,307]{dupuisellis97},
\[
	\liminf_{n\to\infty} \int_{\mathcal{Z}} \|\phi\|_{\mathcal{X}}^{2+\delta_{0}}\, \Theta^{n}(d\phi,dr,dw) \geq \int_{\mathcal{Z}} \|\phi\|_{\mathcal{X}}^{2+\delta_{0}}\, \bar{\Theta}(d\phi,dr,dw).
\]
By definition of $\mathrm{d}_{\mathcal{Z}}$ and of $g$, and thanks to H{\"o}lder's inequality, it follows that $\Theta\in \prbms[2]{\mathcal{Z}}$. By convergence of $(\Theta^{n})_{n\in\mathbb{N}}$ to $\bar{\Theta}$ in $\prbms{\mathcal{Z}}$, we have for every $M > 0$,
\[
	\lim_{n\to\infty} \int_{\mathcal{Z}} M\wedge \|\phi\|_{\mathcal{X}}^{2}\, \Theta^{n}(d\phi,dr,dw) = \int_{\mathcal{Z}} M\wedge \|\phi\|_{\mathcal{X}}^{2}\, \bar{\Theta}(d\phi,dr,dw).
\]
It suffices to show that (recall the notation for the marginal distributions)
\[
	\limsup_{M\to\infty}\; \sup_{n\in \mathbb{N}} \int_{\mathcal{X}} \mathbf{1}_{\{\|\phi\|_{\mathcal{X}}^{2} \geq M\}}\cdot \|\phi\|_{\mathcal{X}}^{2}\, [\Theta^{n}]_{\mathcal{X}}(d\phi) = 0.
\]
But this is true by H{\"o}lder's inequality, the Markov inequality and the fact that $\sup_{n\in\mathbb{N}} g(\Theta^{n}) \leq c < \infty$ by hypothesis since
\begin{align*}
	&\sup_{n\in \mathbb{N}} \int_{\mathcal{X}} \mathbf{1}_{\{\|\phi\|_{\mathcal{X}}^{2} \geq M\}}\cdot \|\phi\|_{\mathcal{X}}^{2}\, [\Theta^{n}]_{\mathcal{X}}(d\phi) \\
	&\leq \sup_{n\in \mathbb{N}}\left\{ [\Theta^{n}]_{\mathcal{X}}\left(\left\{ \|\phi\|_{\mathcal{X}}^{2} \geq M \right\}\right)^{\frac{\delta_{0}}{2+\delta_{0}}} \cdot \left(\int_{\mathcal{X}} \|\phi\|_{\mathcal{X}}^{2+\delta_{0}}\, [\Theta^{n}]_{\mathcal{X}}(d\phi) \right)^{\frac{2}{2+\delta_{0}}} \right\} \\
	&\leq M^{-\frac{\delta_{0}}{2+\delta_{0}}}\cdot c^{\frac{2\delta_{0}}{(2+\delta_{0})^{2}}}\cdot c^{\frac{2}{2+\delta_{0}}},
\end{align*}
which tends to zero as $M\to \infty$.

%-----

\section{Proof of Lemma~\ref{LemmaConvergence}: local martingale property}
\label{AppLimitPoints}

We have to show that, for $\Prb$-almost every $\omega\in \Omega$, any $f\!: \mathbb{R}^{d}\times \mathbb{R}^{d_{1}}\rightarrow \mathbb{R}$ monomial of first or second order, $M^{\mu_{\omega}}_{f}$ is a $(\mathcal{G}_{t})$-local martingale under $Q_{\omega}$; cf.\ \eqref{DefSolutionMart} in Definition~\ref{DefSolution}. Recall that $\mu_{\omega}$ is the flow of measures in $\mathcal{M}_{2}$ induced by $Q_{\omega}$, that is, $\mu_{\omega}(t) = Q_{\omega}\circ (\hat{X}(t))^{-1}$, $t\in [0,T]$. If $\Theta\in \prbms[2]{\mathcal{Z}}$, then the flow of measures induced by $\Theta$ is in $\mathcal{M}_{2}$; cf.\ Remark~\ref{RemMcKeanVlasov} above. Thus, we may write $M^{\Theta}_{f}$ meaning the process $M^{\mathfrak{m}}_{f}$ with $\mathfrak{m}$ the flow of measures in $\mathcal{M}_{2}$ given by $\mathfrak{m}(t) \doteq \Theta\circ (\hat{X}(t))^{-1}$, $t\in [0,T]$.

We closely follow the proof of Lemma~5.2 in \citet{budhirajaetal12}. The canonical space $\mathcal{Z}$ there is slightly bigger than our $\mathcal{Z}$ here (relaxed controls in $\mathcal{R}_{1}$ instead of $\mathcal{R}_{2}$), but this causes no problems since the smaller space gives $L^{2}$-integrability of controls (instead of $L^{1}$) and we have the corresponding distributional convergence of $Q^{n}$ to $Q$ as $\prbms[2]{\mathcal{Z}}$-valued random variables; cf.\ Lemma~\ref{LemmaTightness} above.

In verifying the local martingale property of $M^{\mu_{\omega}}_{f}$, we will work with randomized stopping times. This will ensure almost sure continuity of certain mappings even if the diffusion coefficient $\sigma\trans{\sigma}$ is degenerate. The randomized stopping times live on an extension $(\hat{\mathcal{Z}},\mathcal{B}(\hat{\mathcal{Z}}))$ of the measurable space $(\mathcal{Z},\mathcal{B}(\mathcal{Z}))$ and are adapted to a canonical filtration $(\hat{\mathcal{G}}_{t})$ in $\mathcal{B}(\hat{\mathcal{Z}})$ given by
\begin{align*}
	& \hat{\mathcal{Z}}\doteq \mathcal{Z}\times[0,1], & & \hat{\mathcal{G}}_{t}\doteq \mathcal{G}_{t} \times \mathcal{B}([0,1]),\quad t \in [0,T]. &
\end{align*}
Any random object defined on $(\mathcal{Z},\mathcal{B}(\mathcal{Z}))$ also lives on $(\hat{\mathcal{Z}},\mathcal{B}(\hat{\mathcal{Z}}))$, and no notational distinction will be made. Let $\lambda$ denote the uniform distribution on $\mathcal{B}([0,1])$. Any probability measure $\Theta $ on $\mathcal{B}(\mathcal{Z})$ induces a probability measure on $\mathcal{B}(\hat{\mathcal{Z}})$ given by $\Theta \otimes \lambda$. For $k\in \mathbb{N}$, define a stopping time $\tau _{k}$ on $(\hat{\mathcal{Z}},\mathcal{B}(\hat{\mathcal{Z}}))$ with respect to the filtration $(\hat{\mathcal{G}}_{t})$ by setting, for $((\phi,r,w),a)\in \mathcal{Z}\times [0,1]$,
\begin{align*}
	\tau_{k}((\phi,r,w),a)&\doteq \inf \left\{ t\in [0,T]: v\bigl((\phi,r,w),t\bigr) \geq k+a\right\},\\
\intertext{where}
	v\bigl((\phi,r,w),t\bigr)&\doteq \int_{\Gamma\times [0,t]} |y|^{2}\; r(dy,ds) +\sup_{s\in [0,t]}|\phi(s)|+\sup_{s\in [0,t]}|w(s)|.
\end{align*}
Then, given any $\Theta\in \prbms{\mathcal{Z}}$, $\tau_{k}\nearrow T$ as $k\to \infty$ and the mapping
\begin{equation*}
	\mathcal{Z}\times [0,1]\ni ((\phi,r,w),a)\mapsto \tau_{k}((\phi,r,w),a)\in [0,T]
\end{equation*}
is continuous with probability one under $\Theta\otimes \lambda$.

Notice that if $M_{f}^{\Theta}$ is a local martingale with respect to $(\hat{\mathcal{G}}_{t})$ under $\Theta \otimes \lambda$ with localizing sequence of stopping times $(\tau_{k})_{k\in\mathbb{N}}$, then $M_{f}^{\Theta}$ is also a local martingale with respect to $(\mathcal{G}_{t})$ under $\Theta$ with localizing sequence of stopping times $(\tau_{k}(.,0))_{k\in\mathbb{N}}$; see the appendix in \citet{budhirajaetal12}. Thus, it suffices to prove the martingale property of $M_{f}^{\Theta}$ up till time $\tau_{k}$ with respect to the filtration $(\hat{\mathcal{G}}_{t})$ and the probability measure $\Theta\otimes \lambda$.

Clearly, the process $M_{f}^{\Theta}(\cdot \wedge \tau_{k})$ is a $(\hat{\mathcal{G}}_{t})$-martingale under $\Theta\otimes \lambda$ if and only if
\begin{equation} \label{EqAppConvMartingale}
	\Mean_{\Theta\otimes \lambda}\left[ \Psi \cdot \left(M_{f}^{\Theta}(t_{1}\wedge \tau_{k}) - M_{f}^{\Theta }(t_{0}\wedge \tau_{k})\right)\right] = 0
\end{equation}
for all $t_{0},t_{1}\in \lbrack 0,T]$ with $t_{0}\leq t_{1}$, and $\hat{\mathcal{G}}_{t_{0}}$-measurable $\Psi \in \mathbf{C}_{b}(\hat{\mathcal{Z}})$. To verify the martingale property of $M_{f}^{\Theta}(.\wedge \tau _{k})$, it is enough to check that \eqref{EqAppConvMartingale} holds for any countable collection of times $t_{0}$, $t_{1}$ which is dense in $[0,T]$ and any countable collection of functions $\Psi \in \mathbf{C}_{b}(\hat{\mathcal{Z}})$ that generates the (countably many) $\sigma$-algebras $\hat{\mathcal{G}}_{t_{0}}$. Recall that the collection of test functions $f$ for which a martingale property must be verified consists of just monomials of degree one or two, and hence is finite. Thus, we can choose a countable collection $\mathcal{T}\subset \mathbb{N}\times [0,T]^{2}\times \mathbf{C}_{b}(\hat{\mathcal{Z}})\times \mathbf{C}^{2}(\mathbb{R}^{d}\!\times\!\mathbb{R}^{d_{1}})$ of test parameters such that whenever $\Theta\in \prbms[2]{\mathcal{Z}}$ satisfies $\eqref{EqAppConvMartingale}$ for all $(k,t_{0},t_{1},\Psi,f)\in \mathcal{T}$, then $M^{\Theta}_{f}$ is a $(\mathcal{G}_{t})$-local martingale under $\Theta$.

Let $(k,t_{0},t_{1},\Psi ,f)\in \mathcal{T}$. Define a mapping $\Phi = \Phi_{(k,t_{0},t_{1},\Psi,f)}\!: \prbms[2]{\mathcal{Z}} \rightarrow \mathbb{R}$ by
\begin{equation*}
	\Phi(\Theta)\doteq \Mean_{\Theta \otimes \lambda}\left[ \Psi \cdot \bigl(M_{f}^{\Theta}(t_{1}\wedge \tau_{k}) - M_{f}^{\Theta}(t_{0}\wedge \tau_{k})\bigr)\right].
\end{equation*}
We claim that $\Phi$ is continuous on $\prbms[2]{\mathcal{Z}}$. To check this, take $\Theta \in \prbms[2]{\mathcal{Z}}$ and any sequence $(\Theta _{l})_{l\in \mathbb{N}}\subset \prbms[2]{\mathcal{Z}}$ that converges to $\Theta$ in $\prbms[2]{\mathcal{Z}}$. Let $\mathfrak{m}_{l}$, $l\in \mathbb{N}$, $\mathfrak{m}$ be the induced flows of measures in $\mathcal{M}_{2}$, that is, $\mathfrak{m}_{l}(t)\doteq \Theta_{l}\circ (\hat{X}(t))^{-1}$, $\mathfrak{m}(t)\doteq \Theta\circ (\hat{X}(t))^{-1}$, $t\in [0,T]$. Recall the definition of $M_{f}^{\Theta} = M_{f}^{\mathfrak{m}}$ in \eqref{ExTestMartingale} and \eqref{ExGenerator} above. By Assumption~\hypref{HypGrowth} and definition of the stopping time $\tau _{k}$, the integrand in \eqref{ExTestMartingale} is bounded. By continuity of $b$, $\sigma$ according to Assumption~\hypref{HypLipschitz}, the almost sure continuity of $\tau_{k}$ under $\Theta\otimes\lambda$, the extended mapping theorem \citep[Theorem~5.5 in][p.\,34]{billingsley68} applied to the relaxed controls in \eqref{ExTestMartingale} (plus convergence of first moments by choice of the topology on $\mathcal{R}_{2}$), and the fact that $\Psi\in \mathbf{C}_{b}(\hat{\mathcal{Z}})$, it follows that the mapping
\[
	\hat{\mathcal{Z}}\ni \hat{z} \mapsto \Psi(\hat{z}) \cdot \left( M_{f}^{\mathfrak{m}}(t_{1}\wedge \tau_{k}(\hat{z}),\hat{z}) - M_{f}^{\mathfrak{m}}(t_{0}\wedge \tau_{k}(\hat{z}),\hat{z}) \right) \in \mathbb{R}
\]
is bounded and $\Theta\otimes\lambda$-almost surely continuous. By weak convergence and the mapping theorem \citep[Theorem~5.1 in][p.\,30]{billingsley68}, it follows that 
\begin{equation} \label{EqAppConvFirstPart}
\begin{split}
	&\Mean_{\Theta_{l}\otimes \lambda}\left[ \Psi \cdot \bigl(M_{f}^{\mathfrak{m}}(t_{1}\wedge \tau _{k})-M_{f}^{\mathfrak{m}}(t_{0}\wedge \tau_{k})\bigr)\right] \\
	&\quad \stackrel{l\to \infty}{\longrightarrow}\; \Mean_{\Theta\otimes \lambda}\left[ \Psi \cdot \bigl(M_{f}^{\mathfrak{m}}(t_{1}\wedge \tau_{k}) - M_{f}^{\mathfrak{m}}(t_{0}\wedge \tau_{k})\bigr)\right].
\end{split}
\end{equation}
Since $(\Theta_{l})_{l\in\mathbb{N}}$ converges to $\Theta$ in $\prbms[2]{\mathcal{Z}}$, we have that $\{\Theta_{l} : l\in \mathbb{N}\}\cup \{\Theta \}$ is compact in $\prbms[2]{\mathcal{Z}}$. By continuity of projections, dominated convergence and the definition of $\mathrm{d}_{\mathcal{Z}}$, we have $\lim_{l\to \infty} \mathrm{d}_{2}\bigl(\mathfrak{m}_{l}(t),\mathfrak{m}(t)\bigr) = 0$ uniformly in $t\in [0,T]$. This together with Assumption~\hypref{HypLipschitz} and the construction of $\tau _{k}$ implies that
\begin{equation*}
	\sup_{t\in [0,T],\hat{z}\in \hat{\mathcal{Z}}} \left|M_{f}^{\mathfrak{m}_{l}}(t\wedge \tau_{k}(\hat{z}),\hat{z}) - M_{f}^{\mathfrak{m}}(t\wedge \tau_{k}(\hat{z}),\hat{z})\right| \stackrel{l\to \infty}{\longrightarrow} 0.
\end{equation*}
Since $\Psi$ is bounded, it follows by dominated convergence that
\[
\begin{split}
	& \left|\Mean_{\Theta_{l}\otimes \lambda}\left[ \Psi \cdot \left(M_{f}^{\mathfrak{m}}(t_{1}\wedge \tau_{k}) - M_{f}^{\mathfrak{m}}(t_{0}\wedge \tau_{k})\right) \right] \right. \\
	&\quad \left. - \Mean_{\Theta_{l}\otimes \lambda }\left[ \Psi \cdot \left(M_{f}^{\mathfrak{m}_{l}}(t_{1}\wedge \tau_{k}) - M_{f}^{\mathfrak{m}_{l}}(t_{0}\wedge \tau_{k}) \right) \right] \right| \stackrel{l\rightarrow \infty }{\longrightarrow } 0.
\end{split}
\end{equation*}
In combination with \eqref{EqAppConvFirstPart}, this implies $\Phi(\Theta_{l})\to \Phi(\Theta)$ as $l\to \infty$.

By hypothesis, the sequence $(Q^{n})_{n\in \mathbb{I}}$ of $\prbms[2]{\mathcal{Z}}$-valued random variables converges to $Q$ in distribution. Hence the mapping theorem and the continuity of $\Phi$ imply $\Phi(Q^{n}) \to \Phi(Q)$ in distribution as $n\to \infty$. Let $n\in \mathbb{I}$. By construction of $Q^{n}$ and Fubini's theorem, for every $\omega\in \Omega_{n}$, 
\begin{align*}
	\Phi (Q_{\omega}^{n})& =\Mean_{Q_{\omega}^{n}\otimes \lambda} \left[\Psi\cdot \bigl(M_{f}^{\mu_{\omega }^{n}}(t_{1}\wedge \tau_{k}) - M_{f}^{\mu_{\omega}^{n}}(t_{0}\wedge \tau_{k})\bigr)\right] \\
	\begin{split}
	&= \frac{1}{n} \sum_{i=1}^{n} \int_{0}^{1} \Psi\bigl((X^{n}_{i}(\cdot,\omega),\rho^{n,i}_{\omega},W^{n}_{i}(\cdot,\omega)),a\bigr) \\
	&\qquad \cdot\Biggl( f\bigl(X^{n}_{i}(t_{1}\wedge \tau^{n,i}_{k}(\omega,a),\omega),W^{n}_{i}(t_{1}\wedge \tau^{n,i}_{k}(\omega,a),\omega)\bigr) \\
	&\qquad  - f\bigl(X^{n}_{i}(t_{0}\wedge \tau^{n,i}_{k}(\omega,a),\omega),W^{n}_{i}(t_{0}\wedge \tau^{n,i}_{k}(\omega,a),\omega)\bigr)  \\
	&\qquad - \int_{t_{0}\wedge \tau^{n,i}_{k}(\omega,a)}^{t_{1}\wedge \tau^{n,i}_{k}(\omega,a)} \mathcal{A}^{\mu^{n}_{\omega}}_{u^{n}_{i}(s,\omega),s}(f)\bigl(X^{n}_{i}(s,\omega),W^{n}_{i}(s,\omega)\bigr)ds \Biggr) da,
	\end{split}
\end{align*}
where $\mathcal{A}$ is defined by \eqref{ExGenerator} and $\tau_{k}^{n,i}(\omega,a)$ is defined like $\tau_{k}((\phi,r,w),a)$ with $\phi$ replaced by $X^{n}_{i}(\cdot,\omega )$, $r$ replaced by $\rho_{\omega}^{n,i}$, the relaxed control corresponding to $u^{n}_{i}(\cdot,\omega)$, and $w$ replaced by $W^{n}_{i}(\cdot,\omega )$.

Let $a\in [0,1]$. By It\^{o}'s formula, it holds $\Prb_{n}$-almost surely that
\begin{align*}
\begin{split}
	&f\bigl(X^{n}_{i}(t_{1}\wedge \tau_{k}^{n,i}), W^{n}_{i}(t_{1}\wedge \tau_{k}^{n,i})\bigr) - f\bigl(X^{n}_{i}(t_{0}\wedge \tau_{k}^{n,i}),W^{n}_{i}(t_{0}\wedge  \tau_{k}^{n,i})\bigr) \\
	&\quad - \int_{t_{0}\wedge\tau_{k}^{n,i}}^{t_{1}\wedge\tau_{k}^{n,i}} \mathcal{A}^{\mu^{n}}_{u_{i}^{n}(s),s}(f)\bigl(X^{n}_{i}(s),W^{n}_{i}(s)\bigr)ds
\end{split}\\
	\begin{split}
	&= \int_{t_{0}\wedge\tau_{k}^{n,i}}^{t_{1}\wedge\tau_{k}^{n,i}} \trans{\nabla_{x}f\bigl(X^{n}_{i}(s),W^{n}_{i}(s)\bigr)} \sigma\bigl(s,X^{n}_{i}(s),\mu^{n}(s)\bigr) dW^{n}_{i}(s) \\
	&\quad + \int_{t_{0}\wedge\tau_{k}^{n,i}}^{t_{1}\wedge\tau_{k}^{n,i}} \trans{\nabla_{y}f\bigl(X^{n}_{i}(s),W^{n}_{i}(s)\bigr)} dW^{n}_{i}(s),
	\end{split}
\end{align*}
where $\tau_{k}^{n,i}=\tau_{k}^{n,i}(\cdot,a)$ and $\tau_{k}^{n,i}$, $\mu^{n}$, $X^{n}_{i}$, $u^{n}_{i}$ all live on $(\Omega_{n},\mathcal{F}^{n})$. By Fubini's theorem and Jensen's inequality, it follows that
\begin{align*}
	&\Mean_{n}\left[ \Phi(Q^{n})^{2}\right] \\
	&\leq \int_{0}^{1} \Mean_{n}\left[ \Mean_{Q_{\omega }^{n}}\left[ \Psi (\cdot,a)\cdot \left( M_{f}^{Q^{n}_{\omega}}(t_{1}\wedge \tau_{k}(\cdot,a)) - M_{f}^{Q^{n}_{\omega}}(t_{0}\wedge \tau_{k}(\cdot,a))\right) \right]^{2}\right] da.
\end{align*}
Let again $a\in [0,1]$. By the It{\^o} isometry, the independence of the Wiener processes $W^{n}_{1},\ldots, W^{n}_{n}$, and because $\Psi(\cdot,a)$ is $\mathcal{G}_{t_{0}}$-measurable and $\tau_{k}(\cdot,a)$ is a stopping time with respect to $(\mathcal{G}_{t})$, it holds that
\begin{align*}
	& \Mean_{n}\left[ \Mean_{Q^{n}_{\omega}}\left[ \Psi (.,a)\cdot \left( M_{f}^{Q^{n}_{\omega}}(t_{1}\wedge \tau_{k}(\cdot,a)) - M_{f}^{Q^{n}_{\omega}}(t_{0}\wedge \tau_{k}(\cdot,a))\right) \right]^{2}\right] \\
	&=\begin{aligned}[t] 
	\Mean_{n}\Biggl[ \Bigl( \frac{1}{n} \sum_{i=1}^{n} \int_{t_{0}\wedge\tau^{n,i}_{k}(\cdot,a)}^{t_{1}\wedge\tau^{n,i}_{k}(\cdot,a)} \Psi(\cdot,a)\cdot \mathbf{1}_{\{\tau^{n,i}_{k}(\cdot,a)\geq t_{0}\}}\cdot \Bigl( \trans{\nabla_{y}f\bigl(X^{n}_{i}(s),W^{n}_{i}(s)\bigr)} \\
	+ \trans{\nabla_{x}f\bigl(X^{n}_{i}(s),W^{n}_{i}(s)\bigr)} \sigma\bigl(s,X^{n}_{i}(s),\mu^{n}(s)\bigr)\Bigr) dW^{n}_{i}(s) \Bigr)^{2} \Biggr]
	\end{aligned}\\
	&=\begin{aligned}[t] 
	\frac{1}{n^{2}} \sum_{i=1}^{n} \Mean_{n}\Biggl[ \int_{t_{0}\wedge\tau^{n,i}_{k}(\cdot,a)}^{t_{1}\wedge\tau^{n,i}_{k}(\cdot,a)} \Big| \Psi(\cdot,a)\cdot \mathbf{1}_{\{\tau^{n,i}_{k}(\cdot,a)\geq t_{0}\}}\cdot \Bigl( \trans{\nabla_{y}f\bigl(X^{n}_{i}(s),W^{n}_{i}(s)\bigr)} \\
	+ \trans{\nabla_{x}f\bigl(X^{n}_{i}(s),W^{n}_{i}(s)\bigr)} \sigma\bigl(s,X^{n}_{i}(s),\mu^{n}(s)\bigr)\Bigr) \Big|^{2} ds \Biggr]
	\end{aligned}\\
	& \stackrel{n\to \infty }{\longrightarrow} 0.
\end{align*}

Since $(\Phi(Q^{n}))_{n\in\mathbb{I}}$ converges to $\Phi(Q)$ in distribution, it follows that for each $(k,t_{0},t_{1},\Psi,f) \in \mathcal{T}$ we can choose a set $Z_{(k,t_{0},t_{1},\Psi,f)} \in \mathcal{F}$ such that $\Prb(Z_{(k,t_{0},t_{1},\Psi,f)}) = 0$ and
\[
	\Phi(Q_{\omega}) = \Phi_{(k,t_{0},t_{1},\Psi,f)}(Q_{\omega}) = 0 \text{ for all } \omega \in \Omega \setminus Z_{(k,t_{0},t_{1},\Psi,f)}.
\]
Let $Z$ be the union of all sets $Z_{(k,t_{0},t_{1},\Psi,f)}$, $(k,t_{0},t_{1},\Psi,f) \in \mathcal{T}$. Since $\mathcal{T}$ is countable, we have $Z \in \mathcal{F}$, $\Prb(Z) = 0$ and
\begin{equation*}
	\Phi_{(k,t_{0},t_{1},\Psi,f)}(Q_{\omega}) = 0 \quad\text{for all } \omega \in \Omega \setminus Z,\text{ all } (k,t_{0},t_{1},\Psi,f) \in \mathcal{T}.
\end{equation*}
By definition of $\Phi$, this implies that, for every test function $f$, $M^{\mu_{\omega}}_{f}$ is a $(\mathcal{G}_{t})$-local martingale under $Q_{\omega}$ for $\Prb$-almost every $\omega\in \Omega$.

\end{appendix}

%Bibliography
\bibliographystyle{abbrvnat}
%\bibliography{MathLit}

\end{document}